\documentclass[oneside,final]{amsart}
\usepackage{geometry}

\usepackage[utf8]{inputenc}
\usepackage{hyperref}
\usepackage[english]{babel}
\usepackage{enumerate}
\usepackage{geometry}
\usepackage{aliascnt}
\usepackage{amssymb}
\usepackage{amsmath}
\usepackage{amsthm}
\usepackage{verbatim}
\usepackage[normalem]{ulem}
\usepackage{mathtools}
\usepackage{esint} 

\usepackage[capitalize]{cleveref}
\crefname{equation}{}{}
\usepackage[right]{showlabels}

\usepackage{tikz}
\usepackage{tikzscale}
\usepackage{graphicx}
\usepackage[labelformat=simple]{subcaption}
\usepackage{mathscinet} 
\usepackage{stmaryrd} 

\usepackage[title]{appendix}

\usepackage{color}

\numberwithin{equation}{section} 

\usepackage{dsfont} 
\usepackage{hyperref} 

\usepackage{autonum} 

\allowdisplaybreaks[4] 

\theoremstyle{plain}

\newtheorem{proposition}{Proposition}[section]

\newaliascnt{lemma}{proposition} 

\newtheorem{lemma}[lemma]{Lemma}

\aliascntresetthe{lemma} 
\Crefname{lemma}{Lemma}{Lemmas}

\newaliascnt{theorem}{proposition} 
\newtheorem{theorem}[theorem]{Theorem}

\aliascntresetthe{theorem} 

\newaliascnt{corollary}{proposition} 
\newtheorem{corollary}[corollary]{Corollary}

\aliascntresetthe{corollary}

\newaliascnt{hypothesis}{proposition}

\aliascntresetthe{hypothesis} 

\theoremstyle{definition}

\newaliascnt{definition}{proposition} 
\newtheorem{definition}[definition]{Definition}

\aliascntresetthe{definition} 
\Crefname{definition}{Definition}{Definitions}

\newaliascnt{problem}{proposition} 

\aliascntresetthe{problem} 

\newaliascnt{example}{proposition} 
\newtheorem{example}[example]{Example}

\aliascntresetthe{example} 

\theoremstyle{remark}

\newaliascnt{remark}{proposition} 
\newtheorem{remark}[remark]{Remark}

\aliascntresetthe{remark}

\newaliascnt{notation}{proposition} 
\newtheorem{notation}[notation]{Notation}

\aliascntresetthe{notation} 

\def\equationautorefname~#1\null{%
	(#1)\null
}
 
\newcommand{\sgn}{sign}
\newcommand{\R}{\mathbb{R}}
\newcommand{\C}{\mathbb{C}}
\newcommand{\Z}{\mathbb{Z}}
\newcommand{\N}{\mathbb{N}}
\newcommand{\E}{\mathcal{E}}

\newcommand{\curv}{\vec{\kappa}}
\newcommand{\scurv}{\kappa}
\renewcommand{\S}{\mathbb{S}}
\newcommand{\X}{\mathbb{X}}

\newcommand{\A}{\mathcal{A}}

\renewcommand{\H}{\mathbb{H}}
\newcommand{\Haus}{\mathcal{H}}
\newcommand{\W}{\mathcal{W}}
\newcommand{\tr}{\mathrm{tr}}

\newcommand{\defeq}{\vcentcolon=}
\newcommand{\Ll}{\mathcal{L}}

\newcommand{\U}{\mathcal{U}}

\newcommand{\cn}{\mathrm{cn}}
\newcommand{\sn}{\mathrm{sn}}
\newcommand{\dn}{\mathrm{dn}}
\newcommand{\am}{\mathrm{am}}

\newcommand*{\dd}{\mathop{}\!\mathrm{d}}

\newcommand{\Real}{\mathrm{Re}}
\newcommand{\Imag}{\mathrm{Im}}

\def\nicefrac#1#2{%
    \raise.5ex\hbox{$#1$}%
    \kern-.15em/\kern-.05em%
    \lower.25ex\hbox{$#2$}}

\begin{document}

\title[Willmore flow with Dirichlet boundary conditions]{On the convergence of the Willmore flow with Dirichlet boundary conditions}
\author{Manuel Schlierf}
\address{Institute of Applied Analysis, Helmholtzstra\ss e 18, 89081 Ulm, Germany.}
\email{manuel.schlierf@uni-ulm.de}


\begin{abstract}
    Very little is yet known regarding the Willmore flow of surfaces with Dirichlet boundary conditions. We consider surfaces with a rotational symmetry as initial data and prove a global existence and convergence result for solutions of the Willmore flow with initial data below an explicit, sharp energy threshold. Strikingly, this threshold depends on the prescribed boundary conditions --- it can even be made to be $0$. We show sharpness for some critical boundary data by constructing surfaces above this energy threshold so that the corresponding Willmore flow develops a singularity. Finally, a Li-Yau inequality for open curves in $\H^2$ is proved.
\end{abstract}

\maketitle


\section{Introduction and main results}

For an immersion $f\colon\Sigma\to\R^3$ of an oriented surface $\Sigma$ with boundary, one defines its \emph{Willmore energy} by
\begin{equation}
    \W(f) \defeq \int_{\Sigma} H^2\dd\mu.
\end{equation}
Here the following notation is used. The surface $\Sigma$ is equipped with the pull-back metric $g_f=f^*\langle\cdot,\cdot\rangle$ and $\mu$ denotes the induced measure on $\Sigma$. In local coordinates $(x^{(1)},x^{(2)})$ in the orientation of $\Sigma$, define $\nu=\frac{1}{|\partial_1 f\times \partial_2f|}\partial_1f\times \partial_2 f$. Then $\nu$ can be extended to a global normal field along $f$.  Moreover, $A_{ij}=\langle \partial^2_{ij}f,\nu\rangle$ denotes the second fundamental form of $f$ in local coordinates and $H=\frac12\tr A=\frac12g^{ij}A_{ij}$ is the mean curvature. Then $A^0_{ij}=A_{ij}-Hg_{ij}$ denotes the trace-free second fundamental form, and we write
\begin{equation}
    \W_0(f)\defeq \int_{\Sigma} |A^0|^2\dd\mu.
\end{equation}
Initially, the study of the Willmore energy focused mainly on closed surfaces, i.e., the case $\partial\Sigma=\emptyset$. Among compact, closed immersed surfaces in $\R^3$, round spheres are the only minimizers of $\W$ and have energy $4\pi$. Remarkably, for such surfaces, the Willmore functional is invariant with respect to smooth conformal transformations.

Regarding the minimization of the Willmore energy among surfaces with boundary, there are several results on the existence of minimizers and critical points where different boundary conditions are prescribed, cf. \cite{schaetzle2010,dallacquadeckelnickgrunau2008,dallacquafroehlichgrunauschieweck2011,dallacquadeckelnickwheeler2013,bergnerdallacquafroehlich2010,deckelnickgrunau2007,deckelnickgrunau2009,novagapozzetta2020,eichmanngrunau2019}. In the recent contribution \cite{metsch2022}, Metsch covers the area-preserving Willmore flow with free boundary conditions. More precisely, he considers surfaces which are suitably close to a half-sphere with a small radius and which are sliding on the boundary of a domain while meeting it orthogonally. In \cite{ruppspener2020}, the one-dimensional analog of the Willmore flow, the so-called elastic flow with Dirichlet (or equivalently \emph{clamped}) boundary conditions is studied under a length preserving constraint. The authors prove global existence and convergence to constrained critical points for initial data only in the energy space $W^{2,2}$.

Early work by Kuwert-Schätzle in \cite{kuwertschaetzle2001} sparked great interest in the analytical properties, especially regarding long-time behavior, of the Willmore flow. Together with their results in \cite{kuwertschaetzle2004} and \cite{kuwertschaetzle2002}, employing a suitable blow-up analysis, Kuwert-Schätzle managed to prove the following. If the initial datum $f_0$ is a smooth, spherical immersion (that is, $\Sigma=\S^2$) whose Willmore energy is below $8\pi$, then the Willmore flow exists globally and, after suitable reparametrization, converges to a round sphere, i.e. to the global minimizer of the Willmore functional. Further, as it turns out, the energy constraint of $8\pi$ is in general sharp. Indeed, in \cite{blatt2009}, Blatt constructs spherical immersions whose Willmore energy approaches $8\pi$ arbitrarily close from above, but the Willmore flow starting in each of those immersions develops a singularity. Various numerical studies on the singular behavior of the Willmore flow are carried out in \cite{mayersimonett2002,barrettgarckenuernberg2021}.

In the case of immersed tori with a rotational symmetry in $\R^3$, similar findings are made in \cite{dallacquamullerschatzlespener2020}. As long as the Willmore energy of the initial torus does not exceed $8\pi$, Dall'Acqua-Müller-Schätzle-Spener prove global existence and convergence to the Clifford torus, the global Willmore minimizer among immersed tori. Further, they construct singular examples showing the optimality of the threshold.

The aim of this work is to give a contribution to the Willmore flow with \emph{Dirichlet} boundary conditions. This being to the author's knowledge the first work on this problem, due to its complexity, we allow ourselves to focus on a special class of initial data. This enables to go beyond perturbations of critical points of the Willmore functional.

More precisely, our analysis focuses on initial data topologically given as a cylinder $\Sigma=[0,1]\times\S^1$ with a rotational symmetry satisfying prescribed Dirichlet boundary conditions. Denote by $\eta$ the outward pointing unit conormal on the boundary $\partial\Sigma$, cf. \cite[Proposition 2.17]{lee2018}. In particular, $\frac{\partial f}{\partial\eta}\colon\partial\Sigma\to\R^3$ is tangential to the surface and orthogonal to its boundary. We show global existence and convergence of solutions to the quasi-linear fourth-order Willmore boundary- and initial value problem below an explicit energy threshold that depends only on the data of the problem. Denoting by $\H^2$ the hyperbolic half-plane and by $\Delta_g$ the Laplace-Beltrami operator on $(\Sigma,g)$, our main result is the following.
\begin{theorem}\label{thm:th1}
    Let $u_0\colon[0,1]\to\H^2$, $u_0=(u_0^{(1)},u_0^{(2)})^t$, denote an immersed profile curve and  $f_{u_0}$ the associated surface of revolution on $\Sigma=[0,1]\times\S^1$, cf. \eqref{eq:surf-rev}. If $\W_0(f_{u_0})\leq 8\pi$ or, equivalently,
    \begin{equation}\label{eq:will-thresh}
        \W(f_{u_0}) \leq 4\pi - 2\pi \frac{\partial_xu_0^{(2)}}{|\partial_xu_0|}\Big|_0^1
    \end{equation}
    and if $f\colon [0,T)\times\Sigma\to\R^3$ is a maximal solution to the Willmore flow with Dirichlet boundary conditions, i.e., to
    \begin{equation}\label{eq:will-flow}
        \begin{cases}
            \partial_t f = - (\Delta_{g_f} {H} + |A^0|^2H)\nu &\text{in $[0,T)\times\Sigma$}\\
            f(0,\cdot) = f_{u_0}&\text{in $\Sigma$}\\
            f(t,y) = f_{u_0}(y) &\text{for $0\leq t < T$ and $y\in\partial\Sigma$}\\
            \frac{\partial f}{\partial\eta}(t,y) = \frac{\partial f_{u_0}}{\partial\eta}(y) &\text{for $0\leq t < T$ and $y\in\partial\Sigma$},
        \end{cases}
    \end{equation}
    then $T=\infty$. Moreover, $f(t,\cdot)$ is a surface of revolution for all $t\geq 0$ and converges up to reparametrization smoothly to a Willmore surface of revolution $f_{\infty}$ for $t\to\infty$. Lastly, $f(t,\cdot)$ for all $t\geq 0$ and $f_{\infty}$ are necessarily embedded.
\end{theorem}

We actually show global existence and convergence if one replaces \eqref{eq:will-thresh} by the assumption that the hyperbolic lengths of the profile curves of $f(t,\cdot)$ remain uniformly bounded in time.

In comparison to the aforementioned results in \cite{kuwertschaetzle2004,dallacquamullerschatzlespener2020}, the energy threshold \eqref{eq:will-thresh} now depends on the Dirichlet boundary data. Strikingly, as a new feature appearing in this problem, depending on the data, it can attain any value in $[0,8\pi]$. 

Knowing that, thanks to the well-posedness of the problem, solutions to \eqref{eq:will-flow} retain their rotational symmetry along the flow, we rewrite \eqref{eq:will-flow} as an evolution equation for the profile curves. The hyperbolic plane $\H^2$ naturally comes up when studying Willmore surfaces of revolution. Consider the hyperbolic elastic energy
\begin{equation}\label{eq:def-el-en}
    \E(u)\defeq  \int_I |\curv|_g^2\dd s
\end{equation}
where $\curv$ is the hyperbolic curvature. Going back to observations in \cite{bryantgriffiths1986}, $\E(u)$, $\W_0(f_u)$ and the Willmore energy $\W(f_u)$ are related by
\begin{equation}\label{eq:el-vs-will}
    \frac{1}{\pi}\W_0(f_u)=\E(u) \quad\text{and}\quad \frac{2}{\pi}\W(f_u) = \E(u) - 4 \frac{\partial_xu^{(2)}}{|\partial_xu|}\Big|_{\partial I}.
\end{equation}
For the second equation, also see \cite[Equation~(2.11)]{eichmanngrunau2019}. More details, especially some explicit computations are included in \Cref{app:geom-form-surf-rev}.

However, the corresponding gradient flows for $\E$ and $\W$ differ. Even though the evolution of the profile curves of surfaces evolving by Willmore flow resembles the equation for the hyperbolic elastic flow, cf. \Cref{sec:prelim}, they differ by a non-constant factor that becomes singular as the curves approach the rotation-axis, cf.  \cite{dallacquaspener2018}. Therefore, it is not clear why results on the elastic flow should aid in understanding the Willmore flow --- even in the rotationally symmetric setting. The methods based on interpolation inequalities used in the study of many non-linear evolution PDEs can however be adapted to the problem. The energy threshold provides control of the singular factor. To the author's knowledge, this is the first contribution where \eqref{eq:will-flow} is solved with an argument of this kind. In particular, the method differs from the one employed in \cite{kuwertschaetzle2004,dallacquamullerschatzlespener2020}.

Note that \Cref{thm:th1} also yields existence of Willmore surfaces of revolution of cylindrical type with prescribed Dirichlet boundary data. Fixing $p_y\in\H^2$ and $\tau_y\in\S^1\subseteq T_{p_y}\H^2$ for $y=0,1$, if \eqref{eq:will-thresh} is satisfied for at least one immersion $u_0$ with $u_0(y)=p_y$ and $\partial_su_0(y)=\tau_y$ for $y=0,1$, then the theorem yields the existence of a Willmore surface of revolution with the induced Dirichlet boundary data. There are already corresponding results in the literature. Consider the case where $\tau_y^{(2)}=0$ and $\tau_y^{(1)}>0$ for $y=0,1$. Then our threshold \eqref{eq:will-thresh} becomes $4\pi$ and, correspondingly, in \cite[Theorem 1.1]{eichmanngrunau2019}, the authors prove the existence of minimizers of $\W$ in the class of surfaces of revolution with Dirichlet boundary data induced by $p_y$ and $\tau_y$ as long as the corresponding infimum is below $4\pi$. Uniqueness of such critical points however is a delicate question and counterexamples in the class of cylindrical Willmore surfaces of revolution are given in \cite{eichmann2016}. Further existence results for minimizers of the Willmore functional with Dirichlet boundary data in the class of cylindrical surfaces of revolution are obtained in \cite{dallacquadeckelnickgrunau2008,dallacquafroehlichgrunauschieweck2011}.

Furthermore, we construct initial data for which the Willmore flow \eqref{eq:will-flow} becomes singular. More precisely, we give a sequence of initial data satisfying fixed critical Dirichlet boundary conditions showing that \eqref{eq:will-thresh} is sharp for this choice.

\begin{theorem}\label{prop:sing-ex}
    There are immersed profile curves $u_{0}^n\colon[0,1]\to\H^2$, $n\in\N$, with boundary data $u_0^n(0)=u_{0}^n(1)=(0,1)^t$ and $\partial_su_0^n(0)=-(0,1)^t$, $\partial_su_0^n(1)=(0,1)^t$ satisfying
    \begin{equation}\label{eq:intro-sing-ex-1}
        \W(f_{u_0^n}) \searrow 0 \quad\text{for $n\to\infty$}
    \end{equation} 
    such that, for any $n\in\N$, the maximal solution $f^n$ of \eqref{eq:will-flow} with $u_0=u_0^n$ develops a singularity. More precisely, the hyperbolic lengths of the profile curves of $f^n(t,\cdot)$ are unbounded in time $t$.
\end{theorem}

Since the hyperbolic length is not uniformly bounded in time, one either obtaines that the maximal existence time is finite or that reparametrizations of the solution do not converge for $t\to\infty$.

In the construction of the singular examples of \Cref{prop:sing-ex}, the setting of open curves introduces several difficulties compared to the case of closed curves. For instance, the singular examples constructed in \cite{blatt2009,muellerspener2020,dallacquamullerschatzlespener2020} are based upon the concept of winding numbers. The winding number is constant along the flow and this fact allows to rule out all possible limits. However, for open curves, arguments based on winding numbers cannot be applied in this way. One original contribution of this work is a new argument without the aid of such topological invariants. Furthermore, understanding all potential limits is significantly more involved now that there is no closing condition.

The last statement in the formulation of \Cref{thm:th1} is based on a result which may be of interest by itself. Namely, we prove the following Li-Yau inequality (cf. \cite{liyau1982}) for open curves in $\H^2$. 

\begin{theorem}
    Any immersion $u\in W^{2,2}([0,1],\H^2)$ with $\E(u)\leq 8$ is an embedding.
\end{theorem} 

We briefly comment on different naming conventions for the boundary conditions in \eqref{eq:will-flow} which we either refer to as \emph{Dirichlet} or \emph{clamped} boundary conditions. For fourth-order problems, both conventions are used in the literature to describe boundary conditions of this type, cf. \cite[(1.9) on p.7 and (2.20) on p.35]{gazzolagrunausweers2010}. For instance, while in \cite{pozzetta2021,palmer2000}, the term \emph{clamped} is used for the boundary data in the stationary problem associated to \eqref{eq:will-flow}, in \cite{dallacquadeckelnickgrunau2008,dallacquafroehlichgrunauschieweck2011,eichmanngrunau2019,schaetzle2010}, the conditions are referred to as \emph{Dirichlet}. 

This paper is structured as follows. In \Cref{sec:prelim}, the geometric background is developed. Firstly, the evolution equation for the profile curves is recalled. Afterwards, in \Cref{sec:priori-bounds}, the singular factor in the evolution equation of the profile curves is controlled by suitable a-priori bounds for the distance of a surface of revolution to the rotation axis. \Cref{sec:lte} is dedicated to proving global existence of \eqref{eq:will-flow} and subconvergence. Using a \L ojasiewicz-Simon gradient inequality for $\E$, the subconvergence is promoted to convergence in \Cref{sec:conv}. \Cref{sec:opti} is dedicated to the rigorous construction of singular examples for \eqref{eq:will-flow}. To this end, various properties for segments of hyperbolic elastica which may be interesting in their own right are proved. Finally, \Cref{sec:liyau} contains a proof of the Li-Yau inequality for open curves in $\H^2$. 


\section{Preliminaries}\label{sec:prelim}

Throughout this article, we consider the half-plane model for the hyperbolic space, i.e., the Riemannian manifold $(\H^2,g)$ with $\H^2\defeq \{(y^{(1)},y^{(2)})^t\in\R^2:y^{(2)}>0\}$ and $g_{(y^{(1)},y^{(2)})} = \frac{1}{(y^{(2)})^2} \langle \cdot,\cdot\rangle.$ where $\langle\cdot,\cdot\rangle$ and $|\cdot|$ denote the Euclidean scalar product and norm. Given a smooth curve $u\colon I\to\H^2$ on a compact interval $I\subseteq\R$, we define $\dd s\defeq \dd s_u \defeq |\partial_xu|_g \dd x$ and $\partial_s\,\cdot \defeq \partial_{s_u}\cdot \defeq \frac{\partial_x\,\cdot}{|\partial_xu|_g}$ so that $\Ll_{\H^2}(u)\defeq \int_I\dd s$. For a smooth vector field $X=(X^{(1)},X^{(2)})^t\colon I\to \R^2$ along $u$,
\begin{equation}\label{eq:cov-der-h2}
    \nabla_x X = \begin{pmatrix}
        \partial_xX^{(1)}-\frac{1}{u^{(2)}}(X^{(1)}\partial_xu^{(2)}+X^{(2)}\partial_xu^{(1)})\\
        \partial_xX^{(2)}+\frac{1}{u^{(2)}}(X^{(1)}\partial_xu^{(1)}-X^{(2)}\partial_xu^{(2)})
    \end{pmatrix}
\end{equation}
denotes its covariant derivative in the global coordinates $(y^{(1)},y^{(2)})$ making the identification $\partial_{y^{(1)}}=(1,0)^t$ and $\partial_{y^{(2)}}=(0,1)^t$. Moreover, we denote $\nabla_s \, \cdot =\frac{1}{|\partial_xu|_g}\nabla_x\,\cdot\,$. The curvature $\curv$ of $u$ is
\begin{equation}\label{eq:def-curv}
    \curv \defeq \nabla_s \partial_s u
\end{equation}
and its elastic energy $
    \E(u)\defeq  \int_I |\curv|_g^2\dd s$. 
Consider a smooth family of immersions $u\colon(-\varepsilon,\varepsilon)\times I\to\H^2$ and set $V=\partial_tu|_{t=0}$. If $V=0$ and $\nabla_s^{\bot}V=0$ on $\partial I$, As in \cite[Remark 2.5]{dallacquaspener2017}, one finds
\begin{equation}\label{eq:1-var-elen}
    \frac{\dd}{\dd t} \E(u(t,\cdot))\Big|_{t=0} = \int_I \bigr\langle 2(\nabla_s^{\bot})^2\curv + |\curv|_g^2\curv - 2\curv, V \bigl\rangle_g\dd s_u
\end{equation}
where $\nabla_s^{\perp} \cdot \defeq \nabla_s\cdot \, -\, \langle \nabla_s\cdot\,,\partial_su\rangle\partial_su$. Therefore, for an immersion $u\colon I\to\H^2$, we define
\begin{equation}\label{eq:nabla-L2}
    \nabla\E (u)\defeq\nabla_{L^2(\dd s_u)}\E (u)\defeq 2(\nabla_s^{\bot})^2\curv + |\curv|_g^2\curv - 2 \curv.
\end{equation}
\begin{remark}\label{rem:el-en-isom}
    Clearly, $\E$ is invariant with respect to isometries of $\H^2$. Identifying $\H^2$ with the subset $\{x+iy:y>0\}\subseteq \C$
    \begin{equation}
        \H^2\to\H^2,\quad z=x+iy \mapsto \frac{az+b}{cz+d}
    \end{equation}
    is an isometry of $\H^2$ for any $a,b,c,d\in\R$ with $ad-bc>0$.
\end{remark}
Furthermore, $f_u$ denotes the surface of revolution associated to $u$ where 
\begin{equation}\label{eq:surf-rev}
    f_u\colon I\times \S^1 \to \R^3, \quad (x,\theta) \mapsto (u^{(1)}(x),u^{(2)}(x)\cos\theta,u^{(2)}(x)\sin\theta)^t
\end{equation}
where we identify $\S^1=\R/(2\pi\Z)$. For the induced area measure $\mu_{f_u}$ on $I\times\S^1$, 
\begin{equation}\label{eq:sur-rev-area}
    \mu_{f_u}(I\times\S^1) = 2\pi \int_I u^{(2)} |\partial_xu|\dd x.
\end{equation}
Furthermore, if $g\colon[a,b]\to (0,\infty)$ is a smooth map and $u(x)=(x,g(x))^t$ for $x\in[a,b]$, one finds as in \cite[Equation (2.14)]{eichmanngrunau2019} that
\begin{equation}\label{eq:el-en-graphs}
    \E(u) =  \int_a^b \frac{(\partial_x^2g(x))^2g(x)}{(1+(\partial_xg(x))^2)^{\frac{5}{2}}}+\frac{1}{g(x)\sqrt{1+(\partial_xg(x))^2}}\dd x + 2 \frac{\partial_x g}{\sqrt{1+(\partial_xg)^2}}\Big|_{a}^b.
\end{equation}

The equality in \eqref{eq:el-vs-will} yields that, for fixed clamped boundary data, if $f_u$ is a Willmore surface, then $u$ is a critical point of $\E$. For the converse:

\begin{lemma}[{\cite[Thm. 4.1]{dallacquaspener2018}}]
    For a smooth immersion $u\colon I\to\H^2$, one has 
    \begin{equation}\label{eq:l2-grads}
        (\nabla_s^{\bot})^2\curv + \frac{1}{2}|\curv|^2_g\curv - \curv = 2(u^{(2)})^4 \left(\Delta_{g_{f_u}} H+|A^0|^2H\right) \vec{n}
    \end{equation}
    where $\vec{n}=\frac{1}{|\partial_xu|}(\partial_xu^{(2)},-\partial_xu^{(1)})^t$. 
\end{lemma}

We obtain the following result motivating the studies of this article.

\begin{corollary}
    Consider $u\colon[0,T)\times I\to \H^2$ smooth solving
    \begin{equation}\label{eq:wf-eq}
        \partial_t u = -\frac{1}{2(u^{(2)})^4} \bigl( (\nabla_s^{\bot})^2\curv + \frac{1}{2}|\curv|_g^2\curv -\curv \bigr) = -\frac{1}{4(u^{(2)})^4} \nabla\E(u)
    \end{equation}
    on $[0,T)\times I$. Then $f_u$ solves the Willmore flow equation 
    \begin{equation}
        \partial_tf_u = - \left( \Delta_{g_{f_u}} H + |A^0|^2H \right) \nu
    \end{equation}
    on $[0,T)\times (I\times \S^1)$ where $\nu=\frac{1}{|\partial_xu|}(\partial_xu^{(2)}(x),-\partial_xu^{(1)}(x)\cos\theta,-\partial_xu^{(1)}(x)\sin\theta)^t$.
\end{corollary}

\subsection{A-priori bounds for curves with elastic energy below 8}\label{sec:priori-bounds}

The following result slightly generalizes the argument of \cite[Lemma 4.2]{eichmanngrunau2019}.
\begin{lemma}\label{lem:bbl}
    Consider a sequence $(u_n)_{n\in\N}$ of smooth, immersed curves $[0,1]\to \H^2$ such that, for $\alpha>0$, $u_n^{(2)}(0)\geq\alpha$ and $u_n^{(2)}(1)\geq\alpha$ for each $n\in\N$. Furthermore, suppose that
    \begin{equation}\label{eq:bbl-0}
        \sup_{n\in\N} \E(u_n) < 8.
    \end{equation}
    Then there exists $c>0$ with $u_n^{(2)}\geq c$ for any $n\in\N$.
\end{lemma}
\begin{proof}
    Suppose that the statement is false. That is, choosing $(x_n)_{n\in\N}\subseteq[0,1]$ with $u_n^{(2)}(x_n)=\min_{[0,1]}u_n^{(2)}$, we have $u_n^{(2)}(x_n)\to 0$ after passing on to a subsequence. Since $u_n^{(2)}(0)\geq \alpha$ and $u_n^{(2)}(1)\geq\alpha$, we may w.l.o.g. suppose that $u_n^{(2)}(x_n)<\alpha$ and therefore $x_n\in (0,1)$ with $\partial_xu_n^{(2)}(x_n)=0$ for any $n\in\N$.

    \textbf{Claim 1.} There exist $0\leq a_n\leq x_n^-\leq x_n\leq x_n^+\leq b_n\leq 1$ such that 
    \begin{equation}\label{eq:bbl-1}
        \partial_xu_n^{(2)}(a_n)\leq 0, u_n^{(2)}(a_n)=\alpha\quad\text{and}\quad \partial_xu_n^{(2)}(b_n)\geq 0,u_n^{(2)}(b_n)=\alpha,
    \end{equation}
    and 
    \begin{equation}\label{eq:bbl-1.1}
        \bigg| \frac{\partial_xu_n^{(2)}}{|\partial_xu_n|}(x_n^-) \bigg| , \bigg| \frac{\partial_xu_n^{(2)}}{|\partial_xu_n|}(x_n^+) \bigg|\to 1
    \end{equation}
    for $n\to\infty$. Notice that \eqref{eq:bbl-1.1} and the fact that $\partial_xu_n^{(2)}(x_n)=0$ yield $x_n^-<x_n<x_n^+$ for sufficiently large values of $n$.

    \begin{proof}[Proof of Claim 1.]\renewcommand{\qedsymbol}{}
        Since $u_n^{(2)}(x_n)<\alpha$, setting 
        \begin{equation}
            a_n = \sup\{x\in[0,x_n]:u_n^{(2)}(x)\geq \alpha\}\quad\text{and}\quad b_n=\inf\{ x\in [x_n,1]:u_n^{(2)}(x)\geq \alpha \},
        \end{equation}
        \eqref{eq:bbl-1} follows. Proceeding by contradiction, suppose that no such sequence $(x_n^-)$ exists. After passing to a subsequence without relabeling, we may w.l.o.g. assume 
        \begin{equation}\label{eq:bbl-2}
            \sup_{x\in[a_n,x_n]} \bigg| \frac{\partial_xu_n^{(2)}}{|\partial_xu_n|}(x) \bigg|\leq 1-\delta.
        \end{equation}
        Particularly, \eqref{eq:bbl-2} yields
        \begin{equation}
            0\leq (\partial_xu_n^{(2)})^2(1-(1-\delta)^2)\leq (1-\delta)^2(\partial_xu_n^{(1)})^2.
        \end{equation}
        Hence, since each $u_n$ is an immersion, the property $|\partial_xu_n^{(1)}|>0$ follows on $[a_n,x_n]$ for any $n\in\N$. Particularly, one can reparametrize $u_n$ as a graph on $[a_n,x_n]$. Indeed, define $\varphi_n\colon [a_n,x_n]\to [u_n^{(1)}(a_n),u_n^{(1)}(x_n)]$ by $\varphi_n\defeq(u_n^{(1)})|_{[a_n,x_n]}$. Since $|\partial_xu_n^{(1)}|>0$, we obtain that $\varphi_n$ is a diffeomorphism and $\psi_n\defeq\varphi_n^{-1}$ satisfies $\partial_y\psi_n = \frac{1}{\partial_xu_n^{(1)}\circ\psi_n}$. Defining $g_n\colon [u_n^{(1)}(a_n),u_n^{(1)}(x_n)]\to (0,\infty)$ by $y\mapsto u_n^{(2)}\circ \psi_n(y)$, we have 
        \begin{equation}
            u_n(\psi_n(y)) = (u_n^{(1)}(\psi_n(y)),g_n(y))^t = (y,g_n(y))^t
        \end{equation} 
        for $y\in [u_n^{(1)}(a_n),u_n^{(1)}(x_n)]$. Moreover, observe that 
        \begin{equation}\label{eq:bbl-3}
            \partial_x u_n^{(1)}=\partial_x\varphi_n\quad\text{and}\quad \partial_xu_n^{(2)}=\partial_x(g_n\circ\varphi_n) = \partial_yg_n\circ\varphi_n\cdot\partial_x\varphi_n.
        \end{equation}
        Thus, \eqref{eq:bbl-2} and \eqref{eq:bbl-3} yield
        \begin{equation}
            \sup_{y\in [u_n^{(1)}(a_n),u_n^{(1)}(x_n)]} \bigg|  \frac{\partial_yg_n}{\sqrt{1+(\partial_yg_n)^2}} (y)  \bigg| \leq 1-\delta.
        \end{equation}
        Particularly, since $1-\delta <1$, one obtains that there exists $C>0$ with
        \begin{equation}\label{eq:bbl-4}
            \sup_{y\in[u_n^{(1)}(a_n),u_n^{(1)}(x_n)]}|\partial_yg_n(y)| \leq C\quad\text{for all }n\in\N.
        \end{equation}
        Using \eqref{eq:el-en-graphs}, we conclude that
        \begin{equation}
            \begin{aligned}
                8&\geq \E(u_n) \geq \E(u_n|_{[a_n,x_n]}) = \E(u_n\circ\psi_n |_{[u_n^{(1)}(a_n),u_n^{(1)} (x_n)]}) \\&\underset{\eqref{eq:el-en-graphs}}{\geq} \int_{u_n^{(1)}(a_n)}^{u_n^{(1)}(x_n)} \frac{1}{g_n\sqrt{1+(\partial_yg_n)^2}} \dd y-4= \int_{u_n^{(1)}(a_n)}^{u_n^{(1)}(x_n)} \frac{\sqrt{1+(\partial_yg_n)^2}}{g_n\cdot(1+(\partial_yg_n)^2)} \dd y -4\\
                &\underset{\eqref{eq:bbl-4}}{\geq} \frac{1}{1+C^2} \cdot \bigg|\int_{u_n^{(1)}(a_n)}^{u_n^{(1)}(x_n)} \frac{\partial_yg_n}{g_n} \dd y\bigg| -4
                = \frac{1}{1+C^2} \cdot \int_{u_n^{(2)}(x_n)}^{\alpha} \frac{1}{r}\dd r -4 \to\infty
            \end{aligned}
        \end{equation}
        since $u_n^{(2)}(x_n)\to 0$. As this is a contradiction, one obtains the existence of the sequence $(x_n^-)$. Similarly, working on $[x_n,b_n]$, one obtains the existence of $(x_n^+)$. This concludes the proof of Claim 1.
    \end{proof}
    Consider now $n\in\N$. Then we distinguish two cases. Firstly, suppose that $\partial_xu_n^{(2)}(x_n^-)<0$. Using \eqref{eq:el-vs-will},
    \begin{equation}
        \begin{aligned}
            \E(u_n|_{[0,x_n]}) &\geq \E(u_n|_{[x_n^-,x_n]}) \underset{\eqref{eq:el-vs-will}}{=} \frac{2}{\pi} \W(f_{u_n|_{[x_n^-,x_n]}}) + 4 \frac{\partial_xu_n^{(2)}}{|\partial_xu_n|}\Big|_{x_n^-}^{x_n}\\
            &\geq 4\bigg( \frac{\partial_xu_n^{(2)}(x_n)}{|\partial_xu_n(x_n)|} - \frac{\partial_xu_n^{(2)}(x_n^-)}{|\partial_xu_n(x_n^-)|}\bigg)=-4 \frac{\partial_xu_n^{(2)}(x_n^-)}{|\partial_xu_n(x_n^-)|}.
        \end{aligned}
    \end{equation}
    Secondly, if $\partial_xu_n^{(2)}(x_n^-)\geq 0$, then
    \begin{equation}
        \begin{aligned}
            \E(u_n|_{[0,x_n]}) &\geq \E(u_n|_{[a_n,x_n^-]}) \underset{\eqref{eq:el-vs-will}}{=} \frac{2}{\pi} \W(f_{u_n|_{[a_n,x_n^-]}}) + 4 \frac{\partial_xu_n^{(2)}}{|\partial_xu_n|}\Big|_{a_n}^{x_n^-}\\
            &\geq 4\bigg( \frac{\partial_xu_n^{(2)}(x_n^-)}{|\partial_xu_n(x_n^-)|} - \frac{\partial_xu_n^{(2)}(a_n)}{|\partial_xu_n(a_n)|}\bigg) \underset{\eqref{eq:bbl-1}}{\geq} 4 \frac{\partial_xu_n^{(2)}(x_n^-)}{|\partial_xu_n(x_n^-)|}.
        \end{aligned}
    \end{equation}
    Using \eqref{eq:bbl-1.1}, we thus obtain that
    \begin{equation}\label{eq:bbl-5}
        \E(u_n|_{[0,x_n]}) \geq 4 + o(1).
    \end{equation}
    Similarly, suppose that $\partial_xu_n^{(2)}(x_n^+)>0$. Using \eqref{eq:el-vs-will},
    \begin{equation}
        \begin{aligned}
            \E(u_n|_{[x_n,1]}) &\geq \E(u_n|_{[x_n,x_n^+]}) \underset{\eqref{eq:el-vs-will}}{=} \frac{2}{\pi} \W(f_{u_n|_{[x_n,x_n^+]}}) + 4 \frac{\partial_xu_n^{(2)}}{|\partial_xu_n|}\Big|_{x_n}^{x_n^+}\\
            &\geq 4\bigg( \frac{\partial_xu_n^{(2)}(x_n^+)}{|\partial_xu_n(x_n^+)|} - \frac{\partial_xu_n^{(2)}(x_n)}{|\partial_xu_n(x_n)|} \bigg)=4 \frac{\partial_xu_n^{(2)}(x_n^+)}{|\partial_xu_n(x_n^+)|} .
        \end{aligned}
    \end{equation}
    Conversely, if $\partial_xu_n^2(x_n^+)\leq 0$, then
    \begin{equation}
        \begin{aligned}
            \E(u_n|_{[x_n,1]}) &\geq \E(u_n|_{[x_n^+,b_n]}) \underset{\eqref{eq:el-vs-will}}{=} \frac{2}{\pi} \W(f_{u_n|_{[x_n^+,b_n]}}) + 4 \frac{\partial_xu_n^{(2)}}{|\partial_xu_n|}\Big|_{x_n^+}^{b_n}\\
            &\geq 4\bigg( \frac{\partial_xu_n^{(2)}(b_n)}{|\partial_xu_n(b_n)|} - \frac{\partial_xu_n^{(2)}(x_n^+)}{|\partial_xu_n(x_n^+)|}\bigg) \underset{\eqref{eq:bbl-1}}{\geq} -4 \frac{\partial_xu_n^{(2)}(x_n^+)}{|\partial_xu_n(x_n^+)|}.
        \end{aligned}
    \end{equation}
    Using again \eqref{eq:bbl-1.1}, we also obtain $\E(u_n|_{[x_n,1]}) \geq 4 + o(1)$, so together with \eqref{eq:bbl-5}, $\E(u_n) = \E(u_n|_{[0,x_n]}) + \E(u_n|_{[x_n,1]}) \geq 8 + o(1)$
    which contradicts \eqref{eq:bbl-0}.
\end{proof}

\begin{remark}
    The energy bound in \eqref{eq:bbl-0} is actually sharp. Indeed, for some $a>0$ and for $\varepsilon>0$ sufficiently small, consider the catenary curves $u_{\varepsilon}\colon [-a,a]\to \H^2$, $x\mapsto (x,{\varepsilon}\cdot \cosh(x/{\varepsilon}))$. Using \eqref{eq:el-en-graphs} with $g_{\varepsilon}(x)\defeq {\varepsilon}\cosh(x/{\varepsilon})$, one computes
    \begin{equation}
        \begin{aligned}
            \E(u_{\varepsilon}) &= \frac{2}{{\varepsilon}}\int_{-a}^{a} \frac{1}{\cosh^2(x/{\varepsilon})}\dd x + 2\bigg[ \frac{\partial_xg_{\varepsilon}}{\sqrt{1+(\partial_xg_{\varepsilon})^2}}(a) -  \frac{\partial_xg_{\varepsilon}}{\sqrt{1+(\partial_xg_{\varepsilon})^2}}(-a) \bigg]\\
            &= 4\frac{e^{2a/\varepsilon}-1}{e^{2a/\varepsilon}+1} + 4 \tanh(a/\varepsilon) \to 8
        \end{aligned}
    \end{equation}
    for $\varepsilon\searrow 0$. Hence, $(u_{\varepsilon})_{\varepsilon>0}$ satisfy $\sup_{\varepsilon>0}\E(u_{\varepsilon})=8$, $u_{\varepsilon}^{(2)}(\pm a)\geq \frac{3a}{2}>0$ for all $\varepsilon>0$, but $u_{\varepsilon}^{(2)}(0)=\varepsilon\to 0$. \Cref{lem:bbl} might seem counter-intuitive since $\E$ is invariant with respect to scaling by any parameter $\varepsilon>0$ (cf. \Cref{rem:el-en-isom}). However, the condition $\smash{u_n^{(2)}(0),u_n^{(2)}(1)\geq \alpha}$ prevents unfavorable scaling-effects.
\end{remark}

\begin{lemma}\label{lem:el}
    Consider a sequence $(u_n)_{n\in\N}$ of smooth, immersed curves $[0,1]\to \H^2$ such that, for $z_1,z_2\in\H^2$, $u_n(0)=z_1$ and $u_n(1)=z_2$ for each $n\in\N$. If
    \begin{equation}\label{eq:el-0}
        \sup_{n\in\N} \E(u_n) < 8,
    \end{equation}
    then there exists $C>0$ with $|u_n|\leq C$ for any $n\in\N$.
\end{lemma}
\begin{proof}
    Suppose that $(|u_n|)_{n\in\N}$ is unbounded. As in the proof of \cite[Lemma 4.3]{eichmanngrunau2019}, one constructs an inversion (and thus an isometry) $R$ of the hyperbolic plane $\H^2$ such that, after passing to a subsequence, $\min_{[0,1]} (R\circ u_n)^{(2)}\to 0$. For instance, one might choose $R(z)\defeq -\frac{1}{z}$ for $z\in\H^2\subseteq\C$. Using \Cref{rem:el-en-isom} and \eqref{eq:el-0}, $
        \sup_{n\in\N} \E(R\circ u_n) < 8 $ and $ (R\circ u_n)(0) = R(z_1),\, (R\circ u_n)(1)=R(z_2) $  for all $n\in\N$.
    By \Cref{lem:bbl}, $(R\circ u_n)^{(2)}$ is uniformly bounded from below, a contradiction.
\end{proof}

\begin{corollary}\label{cor:el-ctrl-len}
    Consider $p_1,p_2\in\H^2$ and a subset $\U$ of the class of smooth immersions $[a,b]\to\H^2$ with $u(a)=p_1$ and $u(b)=p_2$ for all $u\in\U$. If $\sup_{u\in\U} \E(u) < 8$, then $\sup_{u\in\U} \Ll_{\H^2}(u) <\infty$.
\end{corollary}
\begin{proof}
    By \Cref{lem:bbl,lem:el}, there exists $K\subseteq \H^2$ compact with $u([a,b])\subseteq K$ for all $u\in\U$. Particularly, there exists $c>0$ with $u^{(2)}>c$ for all $u\in\U$. 
    
    Fix now any $u\in\U$. For the associated surface of revolution $f_u$, one has
    \begin{equation}\label{eq:sc-1*}
        \mu_{f_u}([a,b]\times\S^1) = 2\pi \int_I u^{(2)} |\partial_xu|\dd x = 2\pi \int_I (u^{(2)})^2 |\partial_xu|_{g}\dd x \geq  2\pi c^2 \Ll_{\H^2}(u).
    \end{equation}
    Moreover, the area can be estimated by the Willmore energy, the diameter and a term only depending on the boundary-set $f_u(\{a,b\}\times \S^1)$, cf. \cite[bottom of p.538]{novagapozzetta2020}. 
    Since the diameter is controlled by $K$ and since, by \eqref{eq:el-vs-will}, $\W(f_u) \leq \frac{\pi}{2} \cdot (8+8) = 8\pi$, $\sup_{u\in\U} \Ll_{\H^2}(u) \leq C(\frac{1}{c},K,8\pi,p_1,p_2)$ which concludes the proof.
\end{proof}

\begin{remark}\label{rem:len-cpct-ctrl}
    Let $p\in\H^2$ and consider a subset $\U$ of the class of smooth immersions $[a,b]\to\H^2$ with $u(b)=p$. If $\smash{\sup_{u\in\U}\Ll_{\H^2}(u)<\infty}$, then there clearly exists $K\subseteq\H^2$ compact with $u([a,b])\subseteq K$ for all $u\in\U$.
\end{remark}

\section{Global existence and subconvergence}\label{sec:lte}

\subsection{Evolution equations}
Consider the following evolution equations of the relevant geometric quantities.

\begin{lemma}[{\cite[Lemma~2.4]{dallacquaspener2017}}]\label{lem:ev-eq}
    Let $u\colon[0,T)\times I\to \H^2$ be a smooth, immersed curve with $\partial_t u=V$ such that $\langle V,\partial_su \rangle_g=0$. Further, consider a vector field $N\colon [0,T)\times I\to T\H^2$ which is normal to $\partial_s u$. Then the following formulae hold.
    \begin{align}
        \partial_t(\dd s_u) &= (\partial_s\varphi-\langle V,\curv \rangle_g )\dd s_u\label{eq:ev-eq-3}\\
        \nabla_s N &= \nabla_s^{\bot}  N-\langle N,\curv \rangle_g \partial_su\label{eq:ev-eq-5}\\
        \nabla_t \partial_su &= \nabla_s^{\bot}  V\label{eq:ev-eq-6}\\
        \quad(\nabla_t^{\bot} \nabla_s^{\bot}  - \nabla_s^{\bot} \nabla_t^{\bot})  N &= \langle V,\curv\rangle_g \nabla_s^{\bot}  N+\langle N,\curv\rangle_g \nabla_s^{\bot} V-\langle N,\nabla_s^{\bot} V\rangle_g \curv \qquad\label{eq:ev-eq-9}\\
        \nabla_t^{\bot}  \curv &= (\nabla_s^{\bot} )^2 V+\langle V,\curv \rangle_g \curv - V.\label{eq:ev-eq-10}
    \end{align} 
\end{lemma}

\begin{notation}\label{not:tensors}
    Firstly, consider integers $a,b,d\in\N_0$. We denote terms of type 
    \begin{gather}
        T_u(\partial_su,\dots,\partial_su,(\nabla_s^{\bot})^{i_1}\curv,\dots,(\nabla_s^{\bot})^{i_{b-1}}\curv) \cdot (\nabla_s^{\bot})^{i_{b}}\curv \quad
        \text{or}\label{eq:not-tens-1} \\ 
        T_u(\partial_su,\dots,\partial_su,(\nabla_s^{\bot})^{i_1}\curv,\dots,(\nabla_s^{\bot})^{i_{b}}\curv)
    \end{gather}
    by $P_a^{b,d}$ where $T$ is a smooth tensor field on $\H^2$ with  $\sum_{j=1}^b i_j = a $ and $\displaystyle\max_{1\leq i\leq b}i_j \leq d$.
\end{notation}
\begin{remark}\label{rem:def-C(P,K)}
    For a particular term $P_a^{b,d}$ and compact $K\subseteq \H^2$ containing the trace of $u$, define
    \begin{equation}\label{eq:def-C(P,K)-1}
        C(P_a^{b,d},K) \defeq \max_{p\in K} \max_{\substack{v_1,v_2,\dots\in T_p\H^2,\\|v_1|_g,|v_{2}|_g,\dots\leq 1}} T_p(v_1,v_2,\dots).
    \end{equation}
    Writing $|x|_g=|x|$ with an abuse of notation for reals $x\in\R$, one then obtains
    \begin{equation}\label{eq:def-C(P,K)}
        |P_a^{b,d}|_g \leq C(P_a^{b,d},K) \cdot |(\nabla_s^{\bot})^{i_1}\curv|_g\cdots |(\nabla_s^{\bot})^{i_b}\curv|_g.
    \end{equation}
\end{remark}
\begin{notation}\label{not:tensors-1}
    Consider integers $A,B,d\in\N_0$ with $B\geq 1$. We write $\llbracket A,B\rrbracket_d$ for any finite sum of terms of type $P_a^{b,d}$ where 
    \begin{equation}\label{eq:req-A-B}
        a+\frac{b}{2} \leq A+\frac{B}{2}.
    \end{equation}
\end{notation}

\begin{example}
    The requirement on $A,B$ in \eqref{eq:req-A-B} describes the right algebra to apply an interpolation inequality (cf. \Cref{prop:gni-2}). Moreover, the structure in \eqref{eq:not-tens-1} enables us to easily keep track of (derivatives of) lower-order terms in \eqref{eq:wf-eq}. Indeed, consider for instance 
    \begin{equation}\label{ex:tens-str}
        -\frac{1}{2(u^{(2)})^4}  |\curv|_g^2\curv = T_u(\curv,\curv)\curv = \llbracket 0,3\rrbracket_0
    \end{equation} 
    with $T_{(x,y)}(v_1,v_2)\defeq -\frac{1}{2 y^4} \langle v_1,v_2\rangle_g$. By differentiating, we obtain
    \begin{equation}
        \nabla_s^{\bot}\bigg( - \frac{1}{2(u^{(2)})^4} |\curv|_g^2\curv \bigg)  
        = \frac{2}{(u^{(2)})^5}|\scurv|_g^2\cdot\partial_su^{(2)}\cdot  \curv + 2T_u(\nabla_s^{\bot}\curv,\curv)\curv + T(\curv,\curv)\nabla_s^{\bot}\curv.
    \end{equation}
    The lower order terms in this computation can be efficiently taken care of by using the tensorial structure in \eqref{ex:tens-str} as in the following lemma.
\end{example}

\begin{lemma}\label{lem:diff-tensors}
    Using Notation~\ref{not:tensors-1},
    \begin{equation}\label{eq:diff-tensors}
        \nabla_s^{\bot} \llbracket A,B\rrbracket _d = \llbracket A+1,B\rrbracket_{d+1}.
    \end{equation}
\end{lemma}
\emph{On the notation.} By \eqref{eq:not-tens-1}, $\llbracket A,B\rrbracket_d$ is either scalar or a vector field normal to $u$. In the scalar case, we understand that $\nabla_s^{\bot} \llbracket A,B\rrbracket _d\defeq \partial_s \llbracket A,B\rrbracket_d$ and proceed analogously for further directional derivatives such as $\nabla_t^{\bot}$.
\begin{proof}
    Consider any scalar term $P_a^{b,d}$ with $a+\frac{b}{2}\leq A+\frac{B}{2}$ given as in \eqref{eq:not-tens-1}. Firstly,
    \begin{align}
        \partial_s T_u(\partial_su,\dots\, &,\partial_su,(\nabla_s^{\bot})^{i_1}\curv,\dots,(\nabla_s^{\bot})^{i_b}\curv) = (\nabla T)_u(\partial_su,\dots,(\nabla_s^{\bot})^{i_b}\curv,\partial_su) \\
        &\quad + T_u(\curv,\partial_su,\dots,\partial_su,(\nabla_s^{\bot})^{i_1}\curv,\dots,(\nabla_s^{\bot})^{i_b}\curv)\\
        &\quad + T_u(\partial_su,\dots,\partial_su,\curv,(\nabla_s^{\bot})^{i_1}\curv,\dots,(\nabla_s^{\bot})^{i_b}\curv)\\
        &\quad + T_u(\partial_su,\dots,\partial_su,\nabla_s(\nabla_s^{\bot})^{i_1}\curv,\dots,(\nabla_s^{\bot})^{i_b}\curv)\\        
        &\quad + T_u(\partial_su,\dots,\partial_su,(\nabla_s^{\bot})^{i_1}\curv,\dots,\nabla_s(\nabla_s^{\bot})^{i_b}\curv).
    \end{align}
    Using \eqref{eq:ev-eq-5},
     $\partial_s\left(T_u(\partial_su,\dots,(\nabla_s^{\bot})^{i_b}\curv)\right) = \llbracket A+1,B\rrbracket_{d+1}$ and thus \eqref{eq:diff-tensors}. 
\end{proof}

\begin{remark}\label{rem:prop-of-V}
    Consider $u\colon[0,T)\times I\to M$ smooth and 
    \begin{equation}\label{eq:def-V}
        V=a(u)\cdot((\nabla_s^{\bot})^2\curv +\frac{1}{2}|\curv|_g^2\curv - \curv)
    \end{equation}
    where $a\colon \H^2\to (-\infty,0)$ is smooth. Then, writing $a=a(u)$ for short, 
    \begin{equation}\label{eq:prop-of-V}
        \begin{aligned}
            V&=a(\nabla_s^{\bot})^2\curv+ \llbracket 0,3 \rrbracket_0 = \llbracket 2,1 \rrbracket_2\quad\text{and thus}\\
            \nabla_s^{\bot} V &= a (\nabla_s^{\bot})^3\curv + \partial_sa(\nabla_s^{\bot})^2\curv + \llbracket1,3\rrbracket_1 = \llbracket 3,1 \rrbracket_3,\\
            (\nabla_s^{\bot})^2V &= a(\nabla_s^{\bot})^4\curv + 2\partial_sa(\nabla_s^{\bot})^3\curv + \llbracket 2,3\rrbracket_2.
        \end{aligned}
    \end{equation}
    Thus, by \eqref{eq:ev-eq-10},
    \begin{equation}\label{eq:nablat-curv}
        \nabla_t^{\bot} \curv = (\nabla_s^{\bot})^2 V + \llbracket 2,3 \rrbracket_2= a(\nabla_s^{\bot})^4\curv + 2\partial_sa(\nabla_s^{\bot})^3\curv + \llbracket 2,3\rrbracket_2=\llbracket 4,1\rrbracket_4. 
    \end{equation}
\end{remark}

\begin{lemma}\label{lem:tech-lemm}
    Consider $u\colon[0,T)\times I\to \H^2$ smooth with $\partial_t u = a(u) \nabla_{L^2(\dd s_u)}\E(u)$ where $a\colon\H^2\to(-\infty,0)$ is smooth and write $\phi_l\defeq(\nabla_s^{\bot} )^l\curv$. For any $l\in\N_0$ and $k,m\in\N$, writing $a=a(u)$ for short, we then obtain the following formulae.
    \begin{align}
        \nabla_t^{\bot}(\nabla_s^{\bot})^k\phi_l - (\nabla_s^{\bot})^k\nabla_t^{\bot} \phi_l &= \llbracket k+l+2,3 \rrbracket_{\max\{k+l,k+2\}} \label{eq:tech-lemm-1},\\
        (\nabla_t^{\bot}) \llbracket A,B\rrbracket_d &= \llbracket 4+A,B\rrbracket_{4+A} \label{eq:tech-lemm-2},\\
        (\nabla_t^{\bot})^m u - a^m (\nabla_s^{\bot})^{4m-2}\curv &=  2(m-1)m \cdot a^{m-1}\partial_sa\cdot (\nabla_s^{\bot})^{4m-3}\curv \\
        &\quad+ \llbracket 4m-4,3 \rrbracket_{4m-4}\label{eq:tech-lemm-4}.
    \end{align}
\end{lemma}
\begin{proof}
    For \eqref{eq:tech-lemm-1}, observe that, for $A\in\N$, \eqref{eq:ev-eq-9} yields
    \begin{equation}
    \begin{aligned}
        \nabla_t^{\bot}\nabla_s^{\bot}\phi_A - \nabla_s^{\bot}\nabla_t^{\bot}\phi_A &= \langle V,\curv\rangle_g\nabla_s^{\bot}  \phi_A+\langle \phi_A,\curv\rangle_g \nabla_s^{\bot} V-\langle \phi_A,\nabla_s^{\bot} V\rangle_g\curv\\
        &= \llbracket A+3,3\rrbracket_{\max\{A+1,3\}}
    \end{aligned}
    \end{equation}
    using \Cref{rem:prop-of-V}. Since
    \begin{equation}
        \nabla_t^{\bot}(\nabla_s^{\bot})^k\phi_l - (\nabla_s^{\bot})^k\nabla_t^{\bot} \phi_l = \sum_{i=0}^{k-1} (\nabla_s^{\bot})^i \left( \nabla_t^{\bot}\nabla_s^{\bot}\phi_{k+l-1-i} - \nabla_s^{\bot}\nabla_t^{\bot}\phi_{k+l-1-i} \right),
    \end{equation}
    \eqref{eq:tech-lemm-1} follows. For \eqref{eq:tech-lemm-2}, we firstly verify $\nabla_t^{\bot} \phi_l = \llbracket 4+l,1\rrbracket_{4+l}$. By \eqref{eq:nablat-curv}, \eqref{eq:tech-lemm-1},
    \begin{align}
        \nabla_t^{\bot} \phi_l &= \nabla_t^{\bot} (\nabla_s^{\bot})^l \curv = (\nabla_s^{\bot})^l \nabla_t^{\bot} \curv + \llbracket l+2,3\rrbracket_{l+2} \\
        &= (\nabla_s^{\bot})^l \llbracket 4,1\rrbracket_4 + \llbracket l+2,3\rrbracket_{l+2} = \llbracket 4+l,1\rrbracket_{4+l}.
    \end{align}
    Recalling the definition of $\llbracket A,B\rrbracket_d$, this immediately yields \eqref{eq:tech-lemm-2} with a similar argument as in the proof of \Cref{lem:diff-tensors} using \eqref{eq:ev-eq-6} to deal with lower-order terms such as $\nabla_t\partial_su$.  
    Lastly, regarding \eqref{eq:tech-lemm-4}, note that the case $m=1$ is a consequence of $\partial_tu=V$ and \eqref{eq:prop-of-V}. Before showing the general case by induction, consider
    \begin{equation}\label{eq:pr-tl-1.1}
        \partial_t a =  a'(u)\cdot\partial_tu =  a'(u) \cdot V =\llbracket 2,1\rrbracket_2,
    \end{equation}
    using \eqref{eq:prop-of-V}, and
    \begin{equation}\label{eq:pr-tl-1.2}
        \begin{aligned}
            \partial_t\partial_sa &= \partial_t\left( a'(u) \cdot\partial_su \right) = a''(u)(\partial_su,V)  + a'(u) \partial_t\partial_su \\
            &= \llbracket 2,1\rrbracket_2 + a'(u)\cdot \bigl( \partial_t\left[\nicefrac{1}{|\partial_xu|_g}\right] \partial_x u + \partial_s\partial_tu \bigr) \\
            &= \llbracket 2,1\rrbracket_2 + a'(u)\cdot \left( \langle V,\curv\rangle_g \partial_su + \partial_s V \right) = \llbracket 3,1\rrbracket_3,
        \end{aligned}
    \end{equation}
    using \eqref{eq:prop-of-V}. Suppose that \eqref{eq:tech-lemm-4} holds for some $m\in\N$. Using \eqref{eq:tech-lemm-1}, \eqref{eq:tech-lemm-2},
    \begin{equation}\label{eq:pr-tl-2}
        \begin{aligned}
           (\nabla_t^{\bot} )^{m+1}u
            =\,&a^m(\nabla_s^{\bot})^{4m-2}\nabla_t^{\bot}\curv \\
            &+ 2(m-1)m\cdot a^{m-1}\partial_sa\cdot (\nabla_s^{\bot})^{4m-3}\nabla_t^{\bot} \curv + \llbracket 4m,3\rrbracket_{4m}
        \end{aligned}
    \end{equation}
    where we also used \eqref{eq:pr-tl-1.1} and \eqref{eq:pr-tl-1.2} in the last equality. Using \eqref{eq:nablat-curv}, 
    \begin{equation}\label{eq:pr-tl-3}
        \begin{aligned}
            (&\nabla_s^{\bot})^{4m-2}\nabla_t^{\bot}\curv  \\
            &= \sum_{j=0}^{4m-2} {4m-2\choose j} \left( \partial_s^{4m-2-j}a \cdot (\nabla_s^{\bot})^{4+j}\curv + 2\partial_s^{4m-1-j}a \cdot (\nabla_s^{\bot})^{3+j}\curv \right) + \llbracket 4m,3\rrbracket_{4m}\\
            &= a(\nabla_s^{\bot})^{4m+2}\curv + ((4m-2) + 2)\cdot \partial_sa\cdot (\nabla_s^{\bot})^{4m+1}\curv + \llbracket 4m,3\rrbracket_{4m}
        \end{aligned}
    \end{equation}
    and similarly
    \begin{equation}\label{eq:pr-tl-4}
        \begin{aligned}
            (\nabla_s^{\bot})^{4m-3}\nabla_t^{\bot}\curv
            = a (\nabla_s^{\bot})^{4m+1}\curv + \llbracket 4m,3\rrbracket_{4m}.
        \end{aligned}
    \end{equation}
    Combining, \eqref{eq:pr-tl-2}, \eqref{eq:pr-tl-3} and \eqref{eq:pr-tl-4} yields \eqref{eq:tech-lemm-4}.
\end{proof}

\begin{lemma}\label{lem:ibp-lemma}
    Consider $u\colon [0,T)\times  I\to \H^2$ smooth with $\partial_tu=\vcentcolon V$ normal to $\partial_su$. Consider another normal vector field $N\colon[0,T)\times  I\to T\H^2$ along $u$ and, for $a\colon\H^2\to(-\infty,0)$ smooth, define $Y\defeq\nabla_t^{\bot} N + (\nabla_s^{\bot})^2\left( |a(u)|\cdot (\nabla_s^{\bot})^2 N\right)$. If $N,\nabla_s^{\bot} N=0$ on $(0,T)\times \partial I$, then 
    \begin{equation}
        \frac{\dd}{\dd t} \frac{1}{2}\int_I |N|_g^2\dd s + \int_I |a(u)|\cdot \big| (\nabla_s^{\bot})^2N \big|_g^2\dd s = \int_I \langle Y,N \rangle_g\dd s - \frac{1}{2}\int_I \left| N \right|_g^2\langle \curv,V \rangle_g\dd s.
    \end{equation}
\end{lemma}
\begin{proof}
    The claim follows by a simple integration by parts argument. 
\end{proof}

\begin{lemma}
    Under the assumptions of \Cref{lem:tech-lemm}, it holds that $\nabla_s\curv = \nabla_s^{\bot} \curv -|\curv|_g^2\partial_su$ and, for all $m\geq 2$,
    \begin{equation}\label{eq:s-sp-1}
        \nabla_s^m\curv = (\nabla_s^{\bot})^m\curv + \llbracket m-1,2\rrbracket_{m-1}\cdot\partial_su + \llbracket m-2,3\rrbracket_{m-2}.
    \end{equation}
\end{lemma}
\begin{proof}
    Follows from \eqref{eq:ev-eq-5} by induction.
\end{proof}

\subsection{Global existence and subconvergence}

\begin{notation}
    For a smooth immersion $u\colon I\to\H^2$, for $1\leq p<\infty$, write 
    \begin{equation}
        \|(\nabla_s^{\bot})^j\curv\|_p\defeq\Big(\int_I |(\nabla_s^{\bot})^j\curv|_g^p\dd s\Big)^{\frac{1}{p}}\quad\text{and}\quad \|(\nabla_s^{\bot})^j\curv\|_{\infty} \defeq \mathrm{sup}_I |(\nabla_s^{\bot})^j\curv|_g.
    \end{equation}
    Then $\|\curv\|_{m,p}\defeq\sum_{j=0}^m\|(\nabla_s^{\bot})^j\curv\|_p$ for $p\in [1,\infty]$.
\end{notation}

The following interpolation results will be used repeatedly.

\begin{proposition}\label{prop:gni-1}
    Consider a smooth immersion $u\colon I\to \H^2$ with $\Ll_{\H^2}(u)>0$. Let $k\in\N$, $p\in [2,\infty]$ and $0\leq i<k$. Then there exists $C=C(i,k,p,1/\Ll_{\H^2}(u))$ such that, writing $\alpha=(i+\frac{1}{2}-\frac{1}{p})/k$,
    \begin{equation}\label{eq:in-ieq-0}
        \|(\nabla_s^{\bot})^i\curv\|_p \leq C\cdot \|\curv\|_{k,2}^{\alpha} \|\curv\|_{2}^{1-\alpha}.
    \end{equation}
\end{proposition}
\begin{proof}
    This can be proved analogously to \cite[Proposition 4.1]{dallacquaspener2017}.
\end{proof}

\begin{proposition}\label{prop:gni-2}
    Consider a smooth immersion $u\colon I\to \H^2$ with $\Ll_{\H^2}(u)>0$. Further suppose that there exists $K\subseteq \H^2$ compact with $u(I)\subseteq K$. For any $k\in\N$ and integers $a,b,d$ with $b\geq 2$, $d\leq k-1$ and $a+\frac{b}{2}<2k+1$, there is a constant $C$ depending only on $k$, $a$, $b$, ${1}/{\Ll_{\H^2}(u)}$ and $C(P_a^{b,d},K)$ from \eqref{eq:def-C(P,K)-1} with 
    \begin{equation}\label{eq:in-ieq-1}
        \int_I |P_a^{b,d}|_g\dd s \leq C \cdot \|\curv\|_{k,2}^{\gamma}\cdot \|\curv \|_2^{b-\gamma}
    \end{equation}
    where $\gamma=(a+b/2-1)/k$. Moreover, for any $\varepsilon\in (0,1)$,
    \begin{equation}\label{eq:in-ieq-2}
        \int_I |P_a^{b,d}|_g\dd s \leq \varepsilon \|(\nabla_s^{\bot})^k\curv\|_2^2 + C\varepsilon^{-\frac{\gamma}{2-\gamma}}\|\curv\|_2^{2\frac{b-\gamma}{2-\gamma}}+C\|\curv\|_2^b.
    \end{equation}
\end{proposition}
\begin{proof}
    Using \eqref{eq:def-C(P,K)}, the statement is shown just as \cite[Proposition 4.3]{dallacquaspener2017}.
\end{proof}

\begin{theorem}\label{thm:lte}
    Let $u_0\colon I\to \H^2$ be a smooth immersion, $a\colon\H^2\to (-\infty,0)$ smooth and suppose that $u\colon [0,T)\times I\to \H^2$ is a time-maximal solution of
    \begin{equation}\label{eq:lte-1}
        \begin{cases}
            \partial_t u = a(u) ((\nabla_s^{\bot})^2\curv + \frac{1}{2}|\curv|_g^2\curv-\curv )&\text{on } [0,T)\times I\\
            u(0,\cdot)=u_0&\text{on }I\\
            u(t,y)=u_0(y)&\text{for }(t,y)\in [0,T)\times\partial I\\
            \partial_su(t,y)=\partial_su_0(y)&\text{for }(t,y)\in [0,T)\times\partial I.
        \end{cases}
    \end{equation}
    with
    \begin{equation}\label{eq:lte-1.1}
        \sup_{t\in [0,T)} \Ll_{\H^2}(u(t,\cdot)) < \infty.
    \end{equation}
    Then the solution exists globally in time, that is, $T=\infty$.
\end{theorem}
\begin{remark}\label{rem:ste}
    A note on the well-posedness of \eqref{eq:lte-1} is given in Appendix~\ref{app:wp} in the case of the Willmore flow, i.e. for the special choice $a(x,y)=-1/(2y^4)$. Similar arguments also apply in the general case of \eqref{eq:lte-1}.
\end{remark}
\begin{proof}[Proof of \Cref{thm:lte}]
    Firstly, using \eqref{eq:1-var-elen}, 
    \begin{equation}\label{eq:lte-1.2}
        \frac{\dd}{\dd t} \E(u(t,\cdot)) = - 2 \int_I |a(u)| \cdot |(\nabla_s^{\bot})^2\curv+\frac{1}{2}|\curv|_g^2\curv-\curv|_g^2 \leq 0.
    \end{equation}
    Moreover, by \eqref{eq:lte-1.1} and \Cref{rem:len-cpct-ctrl} with $p=u_0(y)$ for some $y\in\partial I$, there exists a compact $K\subseteq\H^2$ with $u(t,x) \in K$ for all $t\in [0,T)$ and $x\in I$.

    Secondly, we establish a lower bound on the length of each $u(t,\cdot)$. Write $I=[\alpha,\beta]$. If $u(\alpha)\neq u(\beta)$, this is clear. Next, if $u(\alpha)$ = $u(\beta)$ but $\partial_su(\alpha)\neq \partial_su(\beta)$, then there exists $v\in \R^2$ with $|v|=1$ such that $\langle \partial_su(\alpha),v\rangle_g \neq \langle \partial_su(\beta),v\rangle_g$. Using \eqref{eq:cov-der-h2}, one computes that
    \begin{align}
        \nabla_sv &= \partial_s v + \frac{1}{u^{(2)}} \begin{pmatrix}
            -v^{(1)}\partial_su^{(2)}-v^{(2)}\partial_su^{(1)}\\ v^{(1)}\partial_su^{(1)} - v^{(2)}\partial_su^{(2)} \end{pmatrix} = \frac{1}{u^{(2)}} \begin{pmatrix}
                -v^{(1)}\partial_su^{(2)}-v^{(2)}\partial_su^{(1)}\\ v^{(1)}\partial_su^{(1)} - v^{(2)}\partial_su^{(2)} \end{pmatrix}       
    \end{align} 
    so that, since $|v|=1$,
    \begin{align}
        |\nabla_sv|_g^2 &= \frac{1}{(u^{(2)})^4} \Big((v^{(1)}\partial_su^{(2)}+v^{(2)}\partial_su^{(1)})^2+(v^{(1)}\partial_su^{(1)}-v^{(2)}\partial_su^{(2)})^2\Big)\\
        &= \frac{1}{(u^{(2)})^4} \Big((\partial_su^{(2)})^2((v^{(1)})^2+(v^{(2)})^2) + (\partial_su^{(1)})^2((v^{(1)})^2+(v^{(2)})^2)\Big) = \frac{|\partial_su|^2}{(u^{(2)})^4} = \frac{1}{(u^{(2)})^2}.
    \end{align}
    Thus, by the fundamental theorem of calculus, writing $\delta=\min_{(x,y)^t\in K}y>0$,
    \begin{equation}
        \begin{aligned}
            0  &< \zeta\defeq  |\langle \partial_su(\beta),v\rangle_g - \langle \partial_su(\alpha),v\rangle_g| \leq \int_I |\partial_s \langle \partial_su,v\rangle_g| \dd s \leq \int_I |\langle \curv,v\rangle_g| + |\nabla_s v|_g\dd s\\
            &\leq \Big((\E(u))^{\frac12}+(\Ll_{\H^2}(u))^{\frac12}\Big)\cdot \bigg(\int_I \frac{1}{(u^{(2)})^2} \dd s\bigg)^{\frac12} \leq (\Ll_{\H^2}(u))^{\frac12}\frac{(\E(u_0))^{\frac12}}{\delta} + \frac{1}{\delta}\Ll_{\H^2}(u)\\
            &\leq \frac{1}{2}\zeta + \Ll_{\H^2}(u)\Big(\frac{1}{2\zeta} \frac{\E(u_0)}{\delta^2} + \frac{1}{\delta}\Big),
        \end{aligned}
    \end{equation}
    using Young's inequality in the last step. This shows boundedness of the hyperbolic length from below. Lastly, if $u(\alpha)=u(\beta)$ and $\partial_su(\alpha)=\partial_su(\beta)$, Fenchel's theorem in $\H^2$ (cf. \cite[Theorem 2.3]{dallacquaspener2017}) yields
    \begin{equation}
        2\pi \leq \int_I |\curv|_g\dd s \leq \sqrt{\Ll_{\H^2}(u)}\sqrt{\E(u)} \leq \sqrt{\E(u_0)}\cdot \sqrt{\Ll_{\H^2}(u)}
    \end{equation}
    and thus again boundedness of $\Ll_{\H^2}(u)$ from below. From now on, write $a=a(u)$.

    \textbf{Step 1:} Time-uniform bounds on $\|(\nabla_s^{\bot})^{l} \curv\|_{2}$. Firstly, we apply \Cref{lem:ibp-lemma} with $N=|a|^{\alpha}(\nabla_t^{\bot})^mu$ for fixed $m\in\N$ where $\alpha\defeq-\frac{1}{2}$. Notice that $N,\nabla_s^{\bot} N=0$ on $[0,T)\times\partial I$ by the boundary conditions, cf. \cite[Remark 2.5]{dallacquapozzi2014}. One obtains
    \begin{equation}\label{eq:lte-2}
        \begin{aligned}
            \frac{\dd}{\dd t}\frac{1}{2}&\int_I |a|^{2\alpha}\cdot|(\nabla_t^{\bot})^mu|_g^2\dd s + \int_I |a|\cdot |(\nabla_s^{\bot})^2[|a|^{\alpha}(\nabla_t^{\bot})^m u]|_g^2\dd s \\
            &= \int_I \langle Y,|a|^{\alpha}(\nabla_t^{\bot})^mu\rangle_g\dd s - \frac{1}{2}\int_I |a|^{2\alpha}\cdot|(\nabla_t^{\bot})^mu|_g^2\langle \curv,V\rangle_g\dd s
        \end{aligned}
    \end{equation}
    where 
    \begin{equation}\label{eq:lte-2.1}
        Y=|a|^{\alpha}(\nabla_t^{\bot})^{m+1}u+(\nabla_s^{\bot})^2\left[|a|\cdot(\nabla_s^{\bot})^2(|a|^{\alpha}\cdot(\nabla_t^{\bot})^mu)\right]+\llbracket 4m,3\rrbracket_{4m}
    \end{equation}
    and $V$ is as in \eqref{eq:prop-of-V}. In computing $Y$ in \eqref{eq:lte-2.1}, terms involving $\partial_ta$ contribute only to $\llbracket 4m,3\rrbracket_{4m}$, cf. \eqref{eq:pr-tl-1.1}. 
    Using $|a|=-a$, we compute
    \begin{equation}\label{eq:lte-3.1}
        \begin{aligned}
            (&\nabla_s^{\bot})^2\left[|a|\cdot(\nabla_s^{\bot})^2(|a|^{\alpha}\cdot(\nabla_t^{\bot})^mu)\right] \\
            &= (\nabla_s^{\bot})^2\Bigl[|a|\cdot\bigl(|a|^{\alpha}(\nabla_s^{\bot})^2(\nabla_t^{\bot})^mu-2\alpha |a|^{\alpha-1}\partial_sa\nabla_s^{\bot}(\nabla_t^{\bot})^mu\bigr) \\
            &\quad+ \llbracket 4m-2,3\rrbracket_{4m-2}\Bigr] \\
            &= -a|a|^{\alpha} (\nabla_s^{\bot})^{4}(\nabla_t^{\bot})^m u -2 (\alpha+1)|a|^{\alpha}\partial_sa(\nabla_s^{\bot})^3(\nabla_t^{\bot})^mu \\
            &\quad-2\alpha |a|^{\alpha} \partial_sa (\nabla_s^{\bot})^3(\nabla_t^{\bot})^m u  + \llbracket 4m,3\rrbracket_{4m}\\
            &= -a|a|^{\alpha} (\nabla_s^{\bot})^{4}(\nabla_t^{\bot})^m u -2 \smash{\underbrace{(2\alpha+1)}_{\mathclap{=\,0\text{ by the above choice of $\alpha$.}}}}|a|^{\alpha}\partial_sa(\nabla_s^{\bot})^3(\nabla_t^{\bot})^mu + \llbracket 4m,3\rrbracket_{4m}
        \end{aligned}
    \end{equation}
    and, using \eqref{eq:tech-lemm-4},
    \begin{equation}\label{eq:lte-3.2}
        \begin{aligned}
            (&\nabla_s^{\bot})^4(\nabla_t^{\bot})^mu\\
            &= a^m (\nabla_s^{\bot})^{4m+2}\curv + 2m(m+1)a^{m-1}\partial_sa (\nabla_s^{\bot})^{4m+1}\curv + \llbracket 4m,3\rrbracket_{4m}.
        \end{aligned}
    \end{equation}
    Hence, combining \eqref{eq:lte-3.1} and \eqref{eq:lte-3.2}, and observing that the terms of order $4m+2$ and $4m+1$ in $\curv$ cancel with the corresponding ones in $(\nabla_t^{\bot})^{m+1}u$, using \eqref{eq:tech-lemm-4} and \eqref{eq:lte-2.1}, $        Y=\llbracket 4m,3\rrbracket_{4m}$.
    Since we would like to apply \Cref{prop:gni-2} to the right-hand-side of \eqref{eq:lte-2} with $k=4m$, we need to make sure that there are no terms of order $4m$ to which we apply the interpolation inequality.

    \textbf{Claim 1.} We have that 
    \begin{equation}
        \int_I \langle Y,|a|^{\alpha}(\nabla_t^{\bot})^m u\rangle \dd s = \int_I \llbracket 8m-2,4\rrbracket_{4m-1}\dd s.
    \end{equation}
    \begin{proof}[Proof of Claim 1.]\renewcommand{\qedsymbol}{}
        Consider any term in $Y=\llbracket 4m,3\rrbracket_{4m}$ which contains a term $(\nabla_s^{\bot})^{4m}\curv$. In general, one such term is given by
        \begin{equation}
            T_u(\underbrace{\partial_su,\dots,\partial_su}_{c\in\N_0\text{-times}},\underbrace{\curv,\dots,\curv}_{b\in\N_0\text{-times}},(\nabla_s^{\bot})^{4m}\curv)^{\perp}
        \end{equation}
        where $T\in\mathcal{T}^{c+b+1}_1(\H^2)$ is a vector-valued tensor field on $\H^2$ and 
        \begin{equation}\label{eq:lte-c1-1}
            4m+\frac12(b+1)\leq 4m+\frac{3}{2}\quad\Longleftrightarrow\quad b\leq 2.
        \end{equation}
        By the product rule for covariant differentiation of tensors and \eqref{eq:ev-eq-5}, \eqref{eq:lte-c1-1},
        \begin{equation}\label{eq:lte-c1-2}
            T(\partial_su,\dots,(\nabla_s^{\bot})^{4m}\curv)^{\perp} = \nabla_s^{\bot}\bigl( \smash{\underbrace{T(\partial_su,\dots,(\nabla_s^{\bot})^{4m-1}\curv)^{\perp}}_{=\,\llbracket 4m-1,3\rrbracket_{4m-1}}} \bigr) + \llbracket 4m,3 \rrbracket_{4m-1}.
        \end{equation}
        Thus, integration by parts yields
        \begin{equation}
            \begin{aligned}
                &\int_I \langle T(\partial_su,\dots,(\nabla_s^{\bot})^{4m}\curv)^{\perp},|a|^{\alpha}(\nabla_t^{\bot})^m u \rangle_g\dd s \\
                &=  \int_I -\langle \llbracket 4m-1,3\rrbracket_{4m-1},\nabla_s^{\bot}\bigl[ |a|^{\alpha} (\nabla_t^{\bot})^m u\bigr]\rangle_g +\langle \llbracket 4m,3\rrbracket_{4m-1},|a|^{\alpha}(\nabla_t^{\bot})^m u\rangle_g \dd s\\
                &=\int_I \llbracket 8m-2,4\rrbracket_{4m-1}\dd s
            \end{aligned}
        \end{equation}
        since $(\nabla_t^{\bot})^mu=0$ on $\partial I$ due to the clamped boundary conditions. So Claim~1 follows.
    \end{proof}
    Altogether, recalling from \eqref{eq:prop-of-V} that $V=\llbracket 2,1\rrbracket_2$ and using \emph{Claim 1}, one obtains
    \begin{equation}\label{eq:lte-3}
        \int_I \langle Y,|a|^{\alpha}(\nabla_t^{\bot})^mu\rangle_g - \frac{1}{2} |a|^{2\alpha}|(\nabla_t^{\bot})^mu|_g^2\langle \curv,V\rangle_g \dd s = \int_I \llbracket 8m-2,4\rrbracket_{4m-1}\dd s.
    \end{equation}
    Furthermore,
    \begin{equation}\label{eq:lte-4.1}
        (\nabla_s^{\bot})^2\Bigl[|a|^{\alpha}\cdot (\nabla_t^{\bot})^m u\Bigr] = |a|^{\alpha} a^m\cdot  (\nabla_s^{\bot})^{4m}\curv + \llbracket 4m-1,1\rrbracket_{4m-1}
    \end{equation}
    by \eqref{eq:tech-lemm-4} and 
    \begin{equation}\label{eq:lte-4.2}
        \int_I |(\nabla_t^{\bot})^mu|_g^2\dd s = \int_I \llbracket 8m-4,2\rrbracket_{4m-2}\dd s.
    \end{equation}
    From  \eqref{eq:lte-2} and \eqref{eq:lte-3}, \eqref{eq:lte-4.1} and \eqref{eq:lte-4.2}, one obtains the following relation.
    \begin{equation}\label{eq:lte-4.3}
        \begin{aligned}
            \frac{\dd}{\dd t}\frac{1}{2}&\int_I |a|^{2\alpha}|(\nabla_t^{\bot})^mu|_g^2\dd s + \frac{1}{2} \int_I |a|^{2\alpha}|(\nabla_t^{\bot})^m u|_g^2\dd s + \frac{1}{2}\int_I  |a|^{2\alpha m+1} |(\nabla_s^{\bot})^{4m} \curv|_g^2\dd s \\
            &\leq \int_I |\llbracket 8m-2,4\rrbracket_{4m-1}|\dd s.
        \end{aligned}
    \end{equation}
    Now we apply the interpolation inequality in \Cref{prop:gni-2} to the above. To this end, apply \eqref{eq:in-ieq-2} to each term $P_a^{b,4m-1}$ in $\llbracket 8m-2,4 \rrbracket_{4m-1}$ from \eqref{eq:lte-4.3} with $k=4m$. This is possible since
    \begin{equation}
        a+\frac{b}{2}\leq 8m-2+\frac{4}{2} = 8m < 8m+1 = 2k+1,\quad 4m-1<4m
    \end{equation}
    and since $\Ll_{\H^2}(u(t,\cdot))$ is uniformly bounded from below by the considerations at the very beginning. Moreover, $\|\curv\|_2$ is uniformly bounded in $t\in [0,T)$ due to \eqref{eq:lte-1.2}. Hence, by \eqref{eq:in-ieq-2} and \Cref{prop:gni-2}, there exists a constant $C_m$ only depending on $k=4m$, the uniform lower bound on the length and on all parameters $a,b$ as well as all constants $\smash{C(P_a^{b,4m-1},K)}$ for the finitely many terms $\smash{P_a^{b,4m-1}}$ comprising $\smash{\llbracket 8m-2,4\rrbracket_{4m-1}}$ from \eqref{eq:lte-4.3} such that 
    \begin{equation}\label{eq:lte-4.5}
        \int_I |\llbracket 8m-2,4\rrbracket_{4m-1}|\dd s \leq \frac{\min_{K}|a|^{2\alpha m+1}}{2} \|(\nabla_s^{\bot})^{4m}\curv\|_2^2 + \frac12 C_m.
    \end{equation}
    Combining \eqref{eq:lte-4.3} and \eqref{eq:lte-4.5},
    \begin{equation}
        \frac{\dd}{\dd t}\int_I |a|^{2\alpha}|(\nabla_t^{\bot})^mu|_g^2\dd s + \int_I |a|^{2\alpha}|(\nabla_t^{\bot})^m u|_g^2\dd s \leq C_m
    \end{equation}
    uniformly in $t\in (0,T)$. Therefore, by Gronwall's inequality and the time-uniform bounds on $|a(u)|$,
    \begin{equation}\label{eq:lte-5}
        \sup_{t\in(0,T)} \|(\nabla_t^{\bot})^mu\|_2(t) < \infty.
    \end{equation}
    By \eqref{eq:tech-lemm-4}, we have
    \begin{equation}\label{eq:lte-5.1}
        \int_I |(\nabla_s^{\bot})^{4m-2}\curv|_g^2\dd s \leq \frac{2}{\min_K|a|^m} \int_I |(\nabla_t^{\bot})^m u|_g^2\dd s + \int_I \llbracket 8m-6,2\rrbracket_{4m-3}\dd s.
    \end{equation}
    Consequently, \eqref{eq:in-ieq-2} with $k=4m-2$ and $\varepsilon=\frac{1}{2}$ applied to the second term on the right-hand-side of \eqref{eq:lte-5.1} and \eqref{eq:lte-5} yield $\sup_{t\in (0,T)}\|(\nabla_s^{\bot})^{4m-2}\curv\|_2(t) <\infty$. 
    Applying again \eqref{eq:in-ieq-2}, we obtain
    \begin{equation}\label{eq:lte-7}
        \sup_{t\in (0,T)}\|(\nabla_s^{\bot})^{l}\curv\|_2(t) < \infty \quad\text{for all $l\in\N$}.
    \end{equation}

    \textbf{Step 2:} Time-uniform bounds on $\||\nabla_s^{l}\curv|_g\|_{\infty}$. Since the length of the curves $u(t,\cdot)$ is uniformly bounded from below, \eqref{eq:in-ieq-0} in \Cref{prop:gni-1} and \eqref{eq:lte-7} yield 
    \begin{equation}\label{eq:lte-8.50}
        \sup_{t\in (0,T)} \||(\nabla_s^{\bot})^l\curv|_g\|_{\infty}(t) <\infty\quad\text{for all $l\in\N$}.
    \end{equation}  
    Using \eqref{eq:s-sp-1}, for any $l\in\N$, one immediately obtains that also
    \begin{equation}\label{eq:lte-8.5}
        \sup_{t\in (0,T)} \||\nabla_s^l\curv|_g\|_{\infty}(t)<\infty\quad\text{for all $l\in\N$}.
    \end{equation}

    \textbf{Step 3:} Time-uniform bounds on $\||\nabla_x^l\curv|_g\|_{\infty}$. Suppose that $T<\infty$. Then one can proceed as in \cite[p.21]{dallacquaspener2017} to conclude from \eqref{eq:lte-8.50} that $\frac1C\leq |\partial_xu|_g\leq C$ for some $C>0$ and 
    \begin{equation}
        \sup_{t\in (0,T)}\bigl\|\partial_x^l |\partial_xu|_g\bigr\|_{\infty}(t) <\infty\quad\text{for all $l\in\N$. }
    \end{equation}
    That is, one obtains control of the parametrization in finite time. Then one easily deduces 
    \begin{equation}\label{eq:lte-9}
        \sup_{t\in (0,T)} \||\nabla_x^l\curv|_g\|_{\infty}(t)<\infty
    \end{equation}
    for any $l\in\N$.

    \paragraph{Conclusion.} Suppose that $T<\infty$. As $\sup_{t\in(0,T)}\Ll_{\H^2}(u(t,\cdot))<\infty$ and by \eqref{eq:lte-9}, one concludes that $u$ extends to a smooth solution of \eqref{eq:lte-1} on $[0,T]$ which contradicts the maximality of $T$ due to short-time existence, cf. \Cref{rem:ste}.
\end{proof}

\begin{proposition}[Subconvergence result]\label{prop:sc}
    Let $u_0\colon  I\to \H^2$ be a smooth immersion and suppose that $u\colon[0,\infty)\times I\to \H^2$ is a smooth solution of \eqref{eq:lte-1} with
    \begin{equation}\label{eq:sc-1}
        \sup_{t\in[0,\infty)} \Ll_{\H^2}(u(t,\cdot))<\infty.
    \end{equation}
    As $(t_k)_{k\in\N}\subseteq (0,\infty)$ diverges to $\infty$, the reparametrizations of $u(t_k,\cdot)\colon I\to \H^2$ on $[0,1]$ with constant $g_{\H^2}$-velocity converge in $C^{\infty}([0,1],\H^2)$ to some $u_{\infty}$ up to subsequence. Moreover, $f_{u_{\infty}}$ is a Willmore surface (equivalently: $u_{\infty}$ is a critical point of $\E$) with the given boundary data.
\end{proposition}
\begin{proof}
    Consider $(t_k)_{k\in\N}\subseteq (0,\infty)$ with $t_k\nearrow\infty$. Denote by $u_k\colon [0,1]\to \H^2$ the constant $g_{\H^2}$-velocity reparametrizations of $u(t_k,\cdot)$. More precisely, choosing smooth diffeomorphisms $\psi_k\colon[0,1]\to I$ with $\partial_x\psi_k=L_k/|\partial_xu(t_k,\cdot)|_g\circ\psi_k$, we set $u_k\defeq u(t_k,\cdot)\circ \psi_k$ where $L_k\defeq\Ll_{\H^2}(u(t_k,\cdot))$. Particularly, $|\partial_xu_k|_g = L_k$ for any $k\in\N$ and 
    \begin{equation}\label{eq:sc-4}
        \nabla_x^m u_k = \left(\nabla_s^m u(t_k,\cdot)\right)\circ\psi_k\cdot (L_k)^m.
    \end{equation}
    Combined with \eqref{eq:lte-8.5} and $\curv=\nabla_s\partial_su$, \eqref{eq:sc-4} and \eqref{eq:sc-1} yield 
    \begin{equation}\label{eq:sc-5}
        \sup_{k\in\N} \|\nabla_x^mu_k\|_{\infty} < \infty\quad\text{for all }m\geq 2.
    \end{equation}
    Moreover, again by \eqref{eq:sc-1},
    \begin{equation}\label{eq:sc-6}
        |\partial_xu_k| \leq \sup_{k\in\N} \Ll_{\H^2}(u_k)=\sup_{k\in\N} \Ll_{\H^2}(u(t_k,\cdot))<\infty.
    \end{equation}
    By \eqref{eq:sc-5}, \eqref{eq:sc-6} and the clamped boundary conditions, $(u_k)_{k\in\N}\subseteq W^{m,2}([0,1])$ is a bounded sequence for any $m\in\N$. Therefore, there exists a subsequence converging to some $u_{\infty}$ in $C^{\infty}([0,1],\H^2)$.

    Finally, we wish to show that $f_{u_{\infty}}$ is a Willmore surface. By \eqref{eq:l2-grads}, it is sufficient to check that $\nabla_{L^2(ds_{u_{\infty}})}\E(u_{\infty}) = 0$. To this end, using \eqref{eq:def-V}, define 
    \begin{equation}
        h(t)\defeq \frac12\langle \nabla\E(u(t,\cdot)),|a|\nabla\E(u(t,\cdot))\rangle_{L^2(\dd s)}=2\int_I  \frac{1}{|a|} |V|_g^2\dd s.
    \end{equation}
    By \eqref{eq:1-var-elen} and \eqref{eq:lte-1}, 
    \begin{equation}
        \frac{\dd}{\dd t} \E(u(t,\cdot)) =  -\int_I \frac12|a|\cdot |\nabla\E(u(t,\cdot))|^2\dd s = - h(t).
    \end{equation}
    One concludes $h\in L^1([0,\infty))$. Moreover, since
    \begin{equation}
        \begin{aligned}
           \frac12 h'(t)
            &= \int_I \frac{ a'(u)\cdot \partial_tu}{a(u)^2} |V|_g^2 - \frac{2}{a(u)} \langle V,\nabla_t^{\bot} V\rangle_g\dd s + \int_I \frac{1}{a(u)} |V|_g^2 \langle \curv,V\rangle_g\dd s\\
            &=\int_I \llbracket 8,2 \rrbracket_6\dd s
        \end{aligned}
    \end{equation}
    by \eqref{eq:ev-eq-3} and \eqref{eq:prop-of-V}, \eqref{eq:lte-8.50} and \eqref{eq:sc-1} yield that $\| h' \|_{L^{\infty}([0,\infty))} < \infty$.
    Thus, $\lim_{t\to\infty} h(t)=0$. This yields $\nabla\E(u_{\infty})=0$ by the smooth convergence established above.
\end{proof}

Together, \Cref{thm:lte}, \Cref{prop:sc} and \Cref{cor:el-ctrl-len} immediately yield the following global existence and subconvergence result.
\begin{corollary}\label{cor:ex-sc-below-8}
    Let $u_0\colon I\to \H^2$ be a smooth immersion satisfying 
    \begin{equation}\label{eq:en-constr}
        \E(u_0)\leq 8.
    \end{equation}
    Furthermore, consider $a\colon\H^2\to (-\infty,0)$ smooth and suppose that $u\colon [0,T)\times I\to \H^2$ is a time-maximal solution of \eqref{eq:lte-1}. Then $T=\infty$ and $u$ sub-converges to a critical point of $\E$.
\end{corollary}
\begin{proof}
    We apply \Cref{thm:lte} and \Cref{prop:sc} suitably. To this end, firstly suppose that $\nabla_{L^2(ds_{u_0})}\E(u_0)=0$. Then $u(t,\cdot)=u_0$ for all $t\geq 0$ by \eqref{eq:lte-1} and the conclusion of the corollary still holds. Conversely, if $\nabla_{L^2(ds_{u_0})}\E(u_0)\neq 0$, then $\partial_t \E(u(t,\cdot))|_{t=0} < 0 $ by \eqref{eq:1-var-elen}, \eqref{eq:nabla-L2} and \eqref{eq:lte-1}. That is, there exists $\varepsilon>0$ such that $\E(u(t,\cdot))\leq \E(u(\varepsilon,\cdot))<8$ for all $t\geq \varepsilon$. \Cref{cor:el-ctrl-len} yields that $\sup_{t\geq\varepsilon}\Ll_{\H^2}(u(t,\cdot))<\infty$ and \eqref{eq:lte-1.1} resp. \eqref{eq:sc-1} follow.
\end{proof}


\section{Convergence}\label{sec:conv}
In this section, we improve the subconvergence result in \Cref{prop:sc}. In particular, we show that the limit $u_{\infty}$ does not depend on the (sub-)sequence of times. 
\subsection{The \L ojasiewicz-{S}imon gradient inequality}

Firstly, we briefly review the arguments used in \cite{dallacquapozzispener2016} to show a \L osjasiewicz-Simon inequality for the Euclidean elastic energy with Dirichlet boundary data to argue that an analogous result also holds for the hyperbolic elastic energy with Dirichlet boundary data.

Fix $I=[a,b]$ and $p_a,p_b\in\H^2$ as well as $\tau_a,\tau_b\in\R^2$, define
\begin{equation}
    \X \defeq \{ u\in W^{4,2}(I,\H^2):|\partial_xu|>0\text{ and }u(y)=p_y,\partial_su(y)=\tau_y\text{ for all $y\in\partial I$} \}
\end{equation}
and choose $\overline{u}\in\X$. Whenever $\U\subseteq L^2(I,\R^2)$, we write $\U^{\perp}\defeq \{f\in\U:\langle f,\partial_x\overline{u}\rangle =0\text{ a.e.}\}$ in the following. Further define $V_c\defeq W^{4,2}(I,\R^2)\cap W_0^{2,2}(I,\R^2)$. 

For $\varepsilon>0$ sufficiently small, for all $f\in W^{4,2,\perp}(I,\R^2)$ with $\|f\|_{W^{4,2}}<\varepsilon$, the curve $\overline{u}+f$ remains immersed and its trace lies in $\H^2$. For such choices of $\varepsilon$, define
\begin{equation}
    U_{\varepsilon}\defeq \{f\in V_c^{\perp}:\|f\|_{W^{4,2}}<\varepsilon\}.
\end{equation}
For such normal perturbations, define the elastic energy
\begin{equation}
    E\colon U_{\varepsilon}\to\R,\quad f\mapsto \E(\overline{u}+f).
\end{equation}
\begin{remark}
    In the following, $\nabla_{L^2(\dd x)}E\colon U_{\varepsilon}\to L^{2,\perp}(I,\R^2)$ refers to the $L^2(\dd x)$-gradient of $E$ such that
    \begin{equation}
        E'(f)(\varphi) = \langle \nabla_{L^2(\dd x)}E(f),\varphi\rangle_{L^2(\dd x)}
    \end{equation}
    for any $f \in U_{\varepsilon}$ and $\varphi\in V_c^{\perp}$. Using \eqref{eq:nabla-L2}, a direct computation and \eqref{eq:1-var-elen} yield that
    \begin{equation}
        \begin{aligned}            
        \nabla_{L^2(\dd x)}E(f) = &\frac{|\partial_x(\overline{u}+f)|}{(\overline{u}^{(2)}+f^{(2)})^3}\nabla_{L^2(\dd s_{\overline{u}+f})}\E(\overline{u}+f) \\
        &- \left\langle \frac{|\partial_x(\overline{u}+f)|}{(\overline{u}^{(2)}+f^{(2)})^3}\nabla_{L^2(\dd s_{\overline{u}+f})}\E(\overline{u}+f),\frac{\partial_x\overline{u}}{|\partial_x\overline{u}|} \right\rangle \frac{\partial_x\overline{u}}{|\partial_x\overline{u}|}.
        \end{aligned}
    \end{equation} 
\end{remark}
\begin{proposition}\label{prop:anal-fred}
    The energy $E$ satisfies
    \begin{enumerate}[(a)]
        \item $E$ is analytic on $U_{\varepsilon}$,
        \item the $L^2(\dd x)$-gradient $\nabla_{L^2(\dd x)}E\colon U_{\varepsilon}\to L^{2,\perp}(I,\R^2)$ is analytic and
        \item the derivative $(\nabla_{L^2(\dd x)} E)'(0)\colon V_c^{\perp}\to L^{2,\perp}(I,\R^2)$ is Fredholm with index zero.
    \end{enumerate}
\end{proposition}
\begin{proof}[Sketch of a proof.]
    The properties (a) and (b) are shown as in \cite[Proof of Theorem 3.1]{dallacquapozzispener2016} using the explicit formulas above.

    For (c), proceeding as in \cite[Proposition~3.8]{dallacquapozzispener2016}, one firstly shows that $(\nabla_{\partial_{s_{\overline{u}}}}^{\perp})^4 \colon V_c^{\perp}\to L^{2,\perp}$ is Fredholm with index zero. Then explicit formulas for the second variation of $\E$ (cf. \cite[Equation~(55)]{pozzetta2022}) give rise to explicit formulas for $(\nabla_{L^2(\dd x)}E)'(0)$ and one concludes (c) as in \cite[Corollary~3.13]{dallacquapozzispener2016}.
\end{proof}

As argued in \cite[Appendix A]{dallacquapozzispener2016}, by \Cref{prop:anal-fred}, the \L ojasiewicz-{S}imon gradient inequality of \cite[Corollary 3.11]{chill2003} can be applied. One obtains a \L ojasiewicz-{S}imon gradient inequality for the elastic energy for normal perturbations. As in \cite[Proof of Theorem 4.8]{ruppspener2020}, \cite[Lemma 4.1]{dallacquapozzispener2016} enables generalizing this inequality to the case for not necessarily normal perturbations $u$ of $\overline{u}$. One obtains the following:

\begin{theorem}\label{thm:Loja-Sim}
    Let $\overline{u}\in\X$ satisfy $\nabla_{L^2(ds_{\overline{u}})}\E(\overline{u})=0$. Then there exist $C,\sigma>0$ and $\theta\in(0,\frac{1}{2}]$ such that, for all $u\in\X$ with $\|u-\overline{u}\|_{W^{4,2}}\leq\sigma$,
    \begin{equation}\label{eq:Loja-Sim}
        |\E(u)-\E(\overline{u})|^{1-\theta}\leq C \||\nabla_{L^2(\dd s_{u})}\E(u)|_g\|_{L^2(\dd s_u)}.
    \end{equation}
\end{theorem}

\subsection{Convergence of the Willmore flow}

We follow the efficient strategy of \cite[Section 4.4]{ruppspener2020} using \eqref{eq:Loja-Sim} in order to upgrade the sub-convergence of \Cref{prop:sc} to full convergence of the constant-speed reparametrizations for $t\to\infty$. Some arguments differ from \cite{ruppspener2020} due to the factor $a(u)$ in \eqref{eq:lte-1}.

For simplicity, suppose that $I=[0,1]$ in the following. Let $T\in (0,\infty]$ and $u\colon [0,T)\times I\to\H^2$ smooth such that $u(t,\cdot)$ is an immersion for each $t\in [0,T)$. As in \cite[Definition 4.9]{ruppspener2020}, we say that the constant speed reparametrization $\widetilde{u}$ of $u$ is given by $\widetilde{u}(t,x)=u(t,\psi(t,x))$ with $\psi(t,\cdot)\colon I\to\ I$ satisfying $\partial_y\psi(t,y)=\Ll_{\H^2}(u(t,\cdot)) / |\partial_xu(t,\psi(t,y))|_g$.

In $\H^2$, one obtains the following variant of \cite[Lemma 4.10]{ruppspener2020} with the same proof.

\begin{lemma}\label{lem:rs-hl}
    Consider $T\in (0,\infty]$ and $u\colon [0,T)\times I\to\H^2$ smooth such that $u(t,y)=p_y$ for $y\in\partial I$ and $t\in [0,T)$. If $\widetilde{u}$ is the constant speed reparametrization of $u$, we have, for all $t\in[0,T)$,
    \begin{equation}\label{eq:rs-hl}
        \||\partial_t \widetilde{u}(t,\cdot)|_g\|_{L^2(\dd x)} \leq \sqrt{{2}/{\Ll_{\H^2}(u(t,\cdot))} + 8 \E(u(t,\cdot))} \cdot \||\partial_tu(t,\cdot)|_g\|_{L^2(ds_{u(t,\cdot)})}.
    \end{equation}
\end{lemma}

\begin{theorem}\label{thm:conv}
    Consider a smooth immersion $u_0\colon  I\to \H^2$ as well as a smooth solution $u\colon[0,\infty)\times I\to \H^2$ of \eqref{eq:lte-1} with
    \begin{equation}\label{eq:conv-1}
        \sup_{t\in[0,\infty)} \Ll_{\H^2}(u(t,\cdot))<\infty.
    \end{equation}
    If $\widetilde{u}(t,\cdot)$ denotes the constant speed reparametrization of $u(t,\cdot)$, then $\widetilde{u}(t,\cdot)\to u_{\infty}$ smoothly for $t\to\infty$ for a critical point $u_{\infty}$ of $\E$ with the boundary conditions of \eqref{eq:lte-1}.
\end{theorem}
\begin{remark}
    With the same arguments as in the proof of \Cref{cor:ex-sc-below-8}, one can argue that the conclusion of the theorem still holds if one replaces \eqref{eq:conv-1} with the assumption that $\E(u_0)\leq 8$.
\end{remark}

\begin{proof}[Proof of \Cref{thm:conv}]
    Using \Cref{prop:sc}, there is a sequence of times $t_k\nearrow \infty$ and a critical point $u_{\infty}$ of $\E$ satisfying the boundary conditions of \eqref{eq:lte-1} such that $\widetilde{u}(t_k,\cdot)\to u_{\infty}$ smoothly. Particularly, by the monotonicity of $t\mapsto \E(u(t,\cdot))$,
    \begin{equation}\label{eq:conv-2}
        \text{$\E(u(t,\cdot))\to\E(u_{\infty})$.}
    \end{equation}
    Consider now two cases. Firstly, suppose that $\E(u(t^*,\cdot)) = \E(u_{\infty})$ for some $t^*\geq 0$. By \eqref{eq:1-var-elen} and \eqref{eq:lte-1}, we have for $t\geq t^*$
    \begin{equation}
        0=\frac{\dd}{\dd t}\E(u(t,\cdot)) = \frac12\int_I a(u) |\nabla\E(u)|_g^2\dd s.
    \end{equation}
    Particularly, $\partial_t u=0$ for all $t\geq t^*$. Hence, the flow is constant and converges.

    For the second case,  suppose that $\E(u(t,\cdot)) > \E(u_{\infty})$ for all $t\geq 0$. Fix $C_{\mathrm{LS}},\sigma$ and $\theta$ as in \Cref{thm:Loja-Sim} with $\overline{u}=u_{\infty}$. Furthermore, consider the mapping $G\colon [0,\infty)\to \R$, $t\mapsto (\E(u(t,\cdot))-\E(u_{\infty}))^{\theta}$. For each $k\in\N$, 
    \begin{equation}
        s_k\defeq \sup\{s\geq t_k:\|\widetilde{u}(t,\cdot)-u_{\infty}\|_{W^{4,2}}<\sigma\text{ for all $t\in[t_k,s)$}\}.
    \end{equation}
    Since $\E$ is invariant with respect to reparametrizations, using \eqref{eq:lte-1} and \eqref{eq:1-var-elen},
    \begin{equation}
        \begin{aligned}
            -G' 
            &= \theta (\E(\widetilde{u})-\E(u_{\infty}))^{\theta-1} \int_I (-\frac12a(u)) |\nabla\E(u)|_g^2\dd s.
        \end{aligned}
    \end{equation}
    By \Cref{rem:len-cpct-ctrl} and \eqref{eq:conv-1}, $ \overline{C}\geq -\frac12a(u)\geq \overline{c}>0$ on $[0,\infty)\times I$. Consequently, using \eqref{eq:Loja-Sim}
    \begin{equation}
        \begin{aligned}
            -G' &\geq \overline{c}\theta (\E(\widetilde{u})-\E(u_{\infty}))^{\theta-1} \||\nabla_{L^2(\dd s_u)}\E(u)|_g\|_{L^2(\dd s_u)}^2\\
            &\geq \frac{\overline{c}\theta}{C_{\mathrm{LS}}} \||\nabla_{L^2(\dd s_u)}\E(u)|_g\|_{L^2(\dd s_u)} =  \frac{\overline{c}\theta}{C_{\mathrm{LS}}} \big({\int_I \left|\frac{a(u)}{a(u)} \nabla_{L^2(\dd s_u)}\E(u)\right|_g^2\dd s_u}\big)^{\frac{1}{2}}\\
            &\geq \frac{\overline{c}\theta}{C_{\mathrm{LS}}\overline{C}} \||\partial_tu|_g\|_{L^2(\dd s_u)} \quad\text{on $[t_k,s_k)$}.
        \end{aligned}
    \end{equation}
    From \eqref{eq:rs-hl} and the lower bounds on $\Ll_{\H^2}(u(t,\cdot))$ deduced at the beginning of the proof of \Cref{thm:lte} and by \Cref{rem:len-cpct-ctrl}, we thus obtain on $[t_k,s_k)$ that
    \begin{equation}\label{eq:conv-4}
        -G' \geq C \|\partial_t\widetilde{u}\|_{L^2(\dd x)}. 
    \end{equation}
    Using \eqref{eq:conv-2}, for $t\in [t_k,s_k)$,
    \begin{equation}\label{eq:conv-3}
        \|\widetilde{u}(t,\cdot)-\widetilde{u}(t_k,\cdot)\|_{L^2(\dd x)} \leq \int_{t_k}^t \|\partial_t\widetilde{u}(\tau,\cdot)\|_{L^2(\dd x)}\dd\tau \leq \frac{1}{C} G(t_k) \to 0
    \end{equation}
    for $k\to\infty$. 

    Suppose now for contradiction, that all $s_k$ are finite. Then \eqref{eq:conv-3} extends to $t=s_k$ by continuity. Applying \Cref{prop:sc} to the sequence $s_k\geq t_k\nearrow \infty$, there exists a critical point $\psi$ of $\E$ such that $\widetilde{u}(s_k)\to\psi$ smoothly for $k\to\infty$. By the definition of $s_k$ and by continuity, $\|\psi-u_{\infty}\|_{W^{4,2}}=\sigma$. One the other hand,
    \begin{equation}
        \|\psi-u_{\infty}\|_{L^2(\dd x)} = \lim_{k\to\infty} \|\widetilde{u}(s_k)-\widetilde{u}(t_k)\|_{L^2(\dd x)} = 0
    \end{equation}
    by \eqref{eq:conv-3}. This is a contradiction!

    Thus, choose $k_0\in\N$ such that $s_{k_0}=\infty$. By the definition of $s_{k_0}$, $\|\widetilde{u}(t,\cdot)-u_{\infty}\|_{W^{4,2}}<\sigma$ for all $t\geq t_{k_0}$. Using \eqref{eq:conv-4}, 
    \begin{equation}
        \|\widetilde{u}(t)-\widetilde{u}(t')\|_{L^2(\dd x)} \leq \int_t^{t'} \|\partial_t\widetilde{u}(\tau,\cdot)\|_{L^2(\dd x)}\dd\tau \leq \frac{G(t)-G(t')}{C} \to 0
    \end{equation}
    as $t,t'\to\infty$. As $\widetilde{u}(t_k,\cdot)\to u_{\infty}$ smoothly, a subsequence-argument now yields that also $\widetilde{u}(t,\cdot)\to u_{\infty}$ smoothly as $t\nearrow \infty$.
\end{proof}


\section{Optimality discussion}\label{sec:opti}
In our proof of long-time existence and convergence, the energy constraint \eqref{eq:en-constr} for the initial datum of the Willmore flow was crucial. In this section, we aim to show that this energy bound is sharp. To this end, a sequence of initial data with elastic energy above but arbitrarily close to $8$ is constructed whose Willmore flow develops a singularity.

As in \cite{muellerspener2020}, the so-called $\lambda$-figure-eights play an essential role in the construction of the initial data leading to singular behavior in the Willmore flow. So firstly, we review what is known for such hyperbolic elastica. Afterwards, some symmetry properties of the Willmore flow are exploited. Notably, in \Cref{rem:el-en-lf8}, we see that the elastic energy of appropriate sections of $\lambda$-figure-eights indeed approaches $8$, the energy threshold of \eqref{eq:en-constr}. Finally, we analyze the tangent vectors at endpoints of those sections of $\lambda$-figure-eights. Then we give the construction of the singular examples. Afterwards, some details in the proof of its singular behavior are filled in.
Altogether, in this section, we prove the following.

\begin{theorem}
    There is a sequence $u_0^n\colon[-1,1]\to\H^2$ of initial data satisfying
    \begin{equation}
        u_0^n(\pm 1)=(0,1)^t,\;\partial_su_0^n(\pm1)=\pm(0,1)^t\quad\text{and}\quad \E(u_0^n)\searrow 8\quad\text{as $n\to\infty$}
    \end{equation}
    and such that the maximal solutions $u^n\colon [0,T_{\mathrm{max},n})\times [-1,1]\to\H^2$ to the Willmore flow (that is \eqref{eq:lte-1} with $a(x,y)=-1/(2y^4)$) with $u^n(0,\cdot)=u_0^n$ each satisfy
    \begin{equation}
        \sup_{t\in[0,T_{\mathrm{max},n})} \mathcal{L}_{\H^2}(u^n(t,\cdot)) = \infty.
    \end{equation} 
\end{theorem}

\begin{remark}
    By \eqref{eq:el-vs-will}, $\E(u_0^n)\searrow 8$ yields $\W(f_{u_0^n})\searrow 0$, cf. \Cref{prop:sing-ex}.
\end{remark}

\subsection{Parametrization of elastica}
Consider any smooth curve $\gamma\colon I\to\H^2$ on an interval $I\subseteq\R$ parametrized by hyperbolic arc-length. Then we fix the smooth normal field $N\colon I\to \C=\R^2$ along $\gamma$ determined by $N=i\cdot \gamma'$. Finally, the hyperbolic \emph{scalar curvature} of $\gamma$ is defined as $\scurv \defeq \langle \curv, N \rangle_g$ with $\curv$ as in \eqref{eq:def-curv}. 
\begin{definition}
    A smooth curve $\gamma\colon I\to\H^2$ is called \emph{elastica} if it is parametrized by hyperbolic arc-length and if its scalar curvature $\scurv$ satisfies
    \begin{equation}\label{eq:elastica-eq}
        2\scurv''+\scurv^3-(\lambda+2)\scurv = 0
    \end{equation}
    for some $\lambda\in\R$. If $\scurv$ satisfies \eqref{eq:elastica-eq} with $\lambda=0$, $\gamma$ is called \emph{free elastica}. Otherwise, $\gamma$ is sometimes referred to as \emph{$\lambda$-constrained elastica}.
\end{definition}

\begin{remark}
    Note that, for $\gamma\colon I\to\H^2$ smooth and parametrized by hyperbolic arc-length, \eqref{eq:elastica-eq} is equivalent to  $\nabla\E(u) = \lambda\curv$.
\end{remark}

The following two results of \cite{langersinger1984} completely characterize hyperbolic elastica:

\begin{lemma}[{\cite[Proposition 2.7]{muellerspener2020}}]
    Let $\gamma\colon I\to\H^2$ be an elastica. Then there is a constant $C\in\R$ such that
    \begin{equation}\label{eq:char-el-0}
        \scurv'^2+\frac{1}{4}\scurv^4-\frac{\lambda+2}{2}\scurv^2=C
    \end{equation}
    and $\zeta\defeq\scurv^2$ is a non-negative solution of 
    \begin{equation}\label{eq:char-el-1}
        (\zeta')^2+\zeta^3-(2\lambda+4)\zeta^2-4C\zeta=0.
    \end{equation}
\end{lemma}

\begin{proposition}[{\cite[Proposition 2.8]{muellerspener2020}}]\label{prop:2.8}
    Each non-negative solution $\zeta$ of \eqref{eq:char-el-1} is global and attains a global maximum $\scurv_0^2\defeq\sup_{x\in\R}\zeta(x)$. Therefore, all non-negative solutions of \eqref{eq:char-el-1} are translations of solutions with the following initial conditions: $\zeta(0)=\scurv_0^2$ and $\zeta'(0)=0$. Further, for $0<\scurv_0^2<\lambda+2$, there exist no $\lambda$-constrained elastica and the cases with $\scurv_0^2\geq \lambda+2$ are exhaustively classified by the following:
    \begin{enumerate}[(a)]
        \item (Circular elastica) $\scurv_0^2=\lambda+2$, $C<0$ and $\zeta(s)=\lambda+2$.
        \item (Orbit-like elastica) $\scurv_0^2\in(\lambda+2,2\lambda+4)$, $C<0$ and $\zeta(s)=\scurv_0^2\dn^2(rs,p)$ where $r=\frac{1}{2}\sqrt{\frac{2\lambda+4}{2-p^2}}$ and $p\in(0,1)$ is such that $\scurv_0^2=\frac{2\lambda+4}{2-p^2}$.
        \item (Asymptotically geodesic elastica) $\scurv_0^2=2\lambda+4$, $C=0$ and $\zeta(s)=\scurv_0^2\frac{1}{\cosh^2(rs)}$ where $r=\frac{1}{2}\sqrt{2\lambda+4}$.
        \item (Wave-like elastica) $\scurv_0^2>2\lambda+4$, $C>0$ and $\zeta(s)=\scurv_0^2\cn^2(rs,p)$ where $r=\frac{1}{2}\sqrt{\frac{2\lambda+4}{2p^2-1}}$ and $p\in(\frac{1}{\sqrt{2}},1)$ is such that $\scurv_0^2=\frac{(2\lambda+4)p^2}{2p^2-1}$.
    \end{enumerate} 
\end{proposition}

\begin{remark}\label{rem:en-circ-el}
    The unique circular elastica up to isometries of $\H^2$ is given by the Clifford torus. Particularly, if $\gamma\colon\R\to\H^2$ is a circular elastica, $\gamma$ is periodic and, for every interval $I$ whose length is greater than the period of $\gamma$, $\E(\gamma|_I)\geq \frac{2}{\pi} 2\pi^2 = 4\pi$. Here we used \eqref{eq:el-vs-will} and that the Clifford torus has $2\pi^2$ as Willmore energy.
\end{remark}

The following existence result on isometries of $\H^2$ will be used repeatedly.

\begin{lemma}[{\cite[Lemma 2.9]{muellerspener2020}}]\label{lem:isom-hyp}
    Let $z\in\H^2$ and $v\in T_z\H^2$ such that $|v|_g=1$. For each $y>0$, there exists an isometry $\Phi$ of $\H^2$ with $\Phi(z)=iy$ and $d\Phi_z(v)=y+0i$.
\end{lemma}

\begin{lemma}\label{lem:ext-el}
    Suppose that $\gamma\colon I\to\H^2$ is an elastica. Then there exists a unique smooth extension to a globally defined elastica $\widetilde{\gamma}\colon\R\to\H^2$.
\end{lemma}
\begin{proof}
    By \Cref{prop:2.8}, for some $\scurv_0\in\R$ with $\scurv_0^2=\sup_{x\in\R}\zeta(x)$ where $\zeta$ is a global solution of \eqref{eq:char-el-1} with $\zeta|_I=\scurv^2$, we have $\scurv(s) = \scurv_0\cdot f(rs)$ where $f$ satisfies $f\in\{\cn(\cdot+s_*,p),\dn(\cdot+s_*,p),\frac{1}{\cosh(\cdot+s_*)},1\}$ for some $s_*\in\R$. By the fundamental theorem of curve theory (cf. \cite[p.417]{blatt2009}), there exists a unique globally defined extension $\widetilde{\gamma}\colon\R\to\H^2$ of $\gamma$ globally parametrized by arc-length with curvature globally given by $\widetilde{\scurv}(s)=\scurv_0\cdot f(r s)$. One then easily checks that $\widetilde{\scurv}$ satisfies \eqref{eq:elastica-eq}, i.e., $\widetilde{\gamma}$ is again an elastica on its domain.
\end{proof}

Whenever we use the unique extension of \Cref{lem:ext-el} in the following, we do not distinguish between $\scurv$ and $\widetilde{\scurv}$.

With these tools, an explicit parametrization for \emph{globally defined} elastica has been achieved in \cite[Theorem 2.22]{muellerspener2020}. The following is a direct consequence of the cited result and \Cref{lem:ext-el}.

\begin{theorem}\label{thm:par-el}
    Let $\gamma\colon I\to\H^2$ be an elastica with non-constant curvature and denote by ${\gamma}$ also its globally defined extension, cf. \Cref{lem:ext-el}. Then
    \begin{equation}\label{eq:def-theta}
        \theta\defeq {\scurv}^2-\lambda+2i{\scurv}'
    \end{equation}
    never vanishes. Choose $s_*\in\R$ with ${\scurv}(s_*)^2=\scurv_0^2$ and suppose that ${\gamma}(s_*)=iy$ and ${\gamma}'(s_*)=y+0i$ for some $y>0$. Then there exist $a,c\in\R$ with $|a|+|c|>0$ satisfying both $ac=-\frac{1}{4}(\lambda^2+4C)$ and $-ay^2+c=(\scurv_0^2-\lambda)y$ such that
    \begin{equation}\label{eq:par-el-1}
        \gamma'=\frac{a\gamma^2+c}{\theta}\quad\text{on $I$}.
    \end{equation}
    Moreover, there exists $z_1\in i\R_{\neq 0}$ such that $\gamma$ is parametrized by 
    \begin{equation}\label{eq:par-el-2}
        \gamma(x)=f\left( \int_{s_*}^x \frac{1}{\theta(s)} \dd s+z_1\right)
    \end{equation}
    where $\theta$ is given as in \eqref{eq:def-theta} and
    \begin{equation}\label{eq:par-el-3}
        f(z)=\begin{cases}
            \sqrt{\frac{c}{a}}\tan(\sqrt{ac}z)&\text{for $a,c>0$,}\\
            -\sqrt{\frac{c}{a}}\cot(\sqrt{ac}z)&\text{for $a,c<0$,}\\
            \sqrt{-\frac{c}{a}}\tanh(\sqrt{-ac}z)&\text{for $ac<0$,}\\
            \frac{1}{az}&\text{for $c=0,a\neq 0$,}\\
            cz&\text{for $c\neq 0, a=0$}.
        \end{cases}
    \end{equation}
\end{theorem}
\textbf{Notation:} An elastica $\gamma\colon I\to\H^2$ with $0\in I$ is called \emph{canonically parametrized} if $\gamma(0)=iy$, $\gamma'(0)=y+0i$ and $\scurv(0)^2=\scurv_0^2$.

\begin{definition}[{\cite[Definition 6.1]{muellerspener2020}}]
    For $\lambda>0$, we call a curve $\gamma$ a \emph{$\lambda$-figure-eight} if $\gamma$ is a $\lambda$-constrained, wave-like elastica with winding number $0$.
\end{definition}

\begin{remark}\label{rem:a-c-for-lf8}
    Using the notation of \Cref{thm:par-el}, for canonically parametrized $\lambda$-figure-eights, one has $ac=-\frac{1}{4}(\lambda^2+4C)$ and $-ay^2+c=(\scurv_0^2-\lambda)y$. By \Cref{prop:2.8} (d), $C>0$ and $\scurv_0^2-\lambda > \lambda+4>0$. So $a<0$ and $c>0$.
\end{remark}

The following lemma is a direct consequence of (the proof of) \cite[Corollary 6.4]{muellerspener2020}.

\begin{lemma}\label{lem:ex-fig8}
    Let $\lambda\in (0,\frac{64}{\pi^2}-2)$. Then, there exists a closed $\lambda$-figure-eight $\gamma_{\lambda}$ with $\gamma_{\lambda}(0)=i$ and $\gamma_{\lambda}'(0)=1+0i$. Moreover, $\E(\gamma_{\lambda})\searrow  16$ for $\lambda\searrow 0$. \\Furthermore, with the notation in \Cref{prop:2.8} (d), $p(\gamma_{\lambda}) \nearrow 1$, $r(\gamma_{\lambda})\to 1$ and $\scurv_0^2(\gamma_{\lambda})\to 4$ for $\lambda\searrow 0$.
\end{lemma}

\subsection{Symmetry of the Willmore flow of $\lambda$-figure-eights}

\begin{remark}\label{rem:sym-scurv-1}
    In this section, the symmetry properties induced by the mapping $P\colon\H^2\to\H^2$, $z=x+iy\mapsto -\overline{z} = -x + iy$ are studied. Let $\gamma\colon I\to\H^2$ be an immersion, $\alpha\in I$ and consider $\widetilde{\gamma}\defeq P\circ \gamma(\alpha-\cdot)$. We argue that $\scurv_{\widetilde{\gamma}}(x) = \scurv_{\gamma}(\alpha-x)$. Using that $P$ is an $\R$-linear isometry of $\H^2$, one finds that $\curv_{\widetilde{\gamma}}(x) = P(\curv_{\gamma} (\alpha-x))$. Using $i\cdot P z = - P(iz)$ for $z\in\C$, one obtains $N_{\widetilde{\gamma}}(x) = P(N_{\gamma}(\alpha-x))$ so that
    \begin{equation}\label{eq:sym-of-scurv}
        \scurv_{\widetilde{\gamma}}(x)=\langle P(\curv_{\gamma}(\alpha-x)),P(N_{\gamma}(\alpha-x))\rangle_g=\scurv_{\gamma}(\alpha-x).
    \end{equation}
\end{remark}

\begin{lemma}\label{lem:fig8-sym}
    Consider a canonically parametrized $\lambda$-figure-eight $\gamma_{\lambda}$ as in \Cref{lem:ex-fig8}. Then, one has that $\gamma_{\lambda}( -{K(p)}/{r} )=\gamma_{\lambda}( {K(p)}/{r} )$ and 
    \begin{equation}\label{eq:fig-8-sym}
        \gamma_{\lambda}(-x)=-\overline{\gamma_{\lambda}(x)}\quad\text{for all $x\in\R$}
    \end{equation}
    where $\overline{\,\cdot\,}$ denotes complex conjugation. 
\end{lemma}
\begin{proof}
    The fact that $\gamma_{\lambda}\left( -{K(p)}/{r}\right) = \gamma_{\lambda}\left( {K(p)}/{r}\right)$ follows from the proof of \cite[Proposition 3.4]{muellerspener2020}. For the second part of the statement, suppose that $\gamma_{\lambda}(0)=i$. With $P$ as in \Cref{rem:sym-scurv-1}, $P\circ \gamma_{\lambda}(0)=P(i)=i$ and $(P\circ \gamma_{\lambda}(-x))'|_{x=0} = 1+0i$, using that $\gamma_{\lambda}'(0)=1+0i$. The fundamental theorem of curve theory now yields \eqref{eq:fig-8-sym}, using \eqref{eq:sym-of-scurv} with $\alpha=0$ and that $\scurv$ is even by \Cref{prop:2.8}.
\end{proof}

\begin{remark}\label{rem:sym-scurv}
    Consider a curve $\gamma\colon[-L,L]\to\H^2$ parametrized by arc-length with $\gamma(-x)=-\overline{\gamma(x)}$ for all $x\in[-L,L]$. Suppose, $\gamma$ has the self-intersection $\gamma(-x)=\gamma(x)=p_x$. Then $p_x\in i\R_{>0}$. Especially, with $x=0$, $\gamma(0)\in i\R_{>0}$. Furthermore, $\Imag(\gamma)$ is an even function. As a consequence, $(\Imag(\gamma))'(0)=0$, i.e., $\gamma'(0)\in\R$. Moreover, by \eqref{eq:sym-of-scurv}, $\scurv(-x) = \scurv(x)$. That is, the signed curvature $\scurv$ is even.
\end{remark}

\begin{lemma}\label{lem:sym-pres}
    Suppose that $u_0\colon [-1,1]\to\H^2$ is a smooth immersion with 
    \begin{equation}\label{eq:sym-pres-1}
        u_0(-x)=-\overline{u_0(x)}\quad\text{for all $x\in[-1,1]$}
    \end{equation}
    and denote by $u\colon [0,\infty)\times[-1,1]\to\H^2$ the solution of the Willmore flow, that is, of \eqref{eq:lte-1} with $a(x,y)=-1/(2y^4)$. Then also 
    \begin{equation}\label{eq:sym-pres-2}
        u(t,-x)=-\overline{u(t,x)}\quad\text{for all $t\geq 0$ and $x\in[-1,1]$}.
    \end{equation}
\end{lemma}
\begin{proof}
    Define $\widetilde{u}(t,x)\defeq Pu(t,-x)$ for all $t\geq 0$ and $x\in [-1,1]$ where $P$ is as in \Cref{rem:sym-scurv-1}. Using that $P$ is a linear isometry of $\H^2$, one verifies that
    \begin{equation}
        \begin{aligned}            
        \partial_t\widetilde{u}(t,x) = a(\widetilde{u}(t,x)) \big[(\nabla_{\partial_{s_{\widetilde{u}}}}^{\perp})^2\curv_{\widetilde{u}}+\frac{1}{2}|\curv_{\widetilde{u}}|_g^2\curv_{\widetilde{u}}-\curv_{\widetilde{u}}\big](t,x).
        \end{aligned}
    \end{equation}
    Particularly, since by assumption $\widetilde{u}(0,x)=Pu_0(-x) = u_0(x)$ for any $x\in[-1,1]$, $\widetilde{u}$ also solves \eqref{eq:lte-1}. \Cref{prop:geo-un} then yields $\widetilde{u}=u$, i.e. \eqref{eq:sym-pres-2}.
\end{proof}

\begin{remark}\label{rem:sym-un-speed}
    Suppose that $u\colon[-1,1]\to\H^2$ is a smooth curve with $u(-x)=-\overline{u(x)}\quad\text{for all $x\in[-1,1]$}$. Then also the constant $g_{\H^2}$-speed reparametrization $\widetilde{u}\colon[-1,1]\to\H^2$ satisfies the same symmetry relation.     
\end{remark}

\subsection{Tangent vectors of simply closed $\lambda$-figure-eights}

Firstly, we consider the asymptotics of the parameters of $\lambda$-figure-eights more closely. This then helps in understanding the behavior of tangent vectors at the self-intersection at $K(p)/r$, cf. \Cref{lem:fig8-sym}.

\begin{lemma}\label{lem:lf8-vh}
    Consider a sequence of $\lambda_n$-figure-eights $(\gamma_n)_{n\in\N}$ with parameters $\lambda_n\searrow 0$, $p_n\nearrow 1$ parametrized in their canonical form. Then we have
    \begin{equation}\label{eq:lf8-vh-1}
        \lim_{n\to\infty}\frac{1-p_n^2}{\lambda_n^2} = \infty.
    \end{equation}
\end{lemma}
    A tediously computational proof can be found in Appendix~\ref{app:tot-curv}.

\begin{corollary}\label{cor:lf8-vh}
    Consider a sequence of $\lambda_n$-figure-eights $(\gamma_n)_{n\in\N}$ with $\lambda_n\searrow 0$, $p_n\nearrow 1$ parametrized in their canonical form. Then, we have
    \begin{equation}\label{eq:lf8-vh-mainres}
        \lim_{n\to\infty} \frac{\gamma_n'(\pm K(p_n)/r_n)}{|\gamma_{n}'(K(p_n)/r_n)|\mathrm{sign}(\scurv_{0,n})} =  \pm i.
    \end{equation}
\end{corollary}
\begin{proof}
    By \Cref{thm:par-el}, there are parameters $a_{n},c_{n}\in\R$ such that
    \begin{equation}
        \gamma_{n}' = \smash{\nicefrac{(a_{n}\gamma_{n}^2+c_{n})}{(\scurv_{\gamma_n}^2-\lambda_n+2i\scurv_{\gamma_n}')}.}
    \end{equation}
    Moreover, $\gamma_{n}(K(p_n)/r_n)=w_{n}i$ for $w_{n}>0$. By \Cref{prop:2.8} and \Cref{rem:der-jef},
    \begin{equation}
    \begin{aligned}
        \scurv_{\gamma_n}(K(p_n)/r_n)&= \scurv_{0,n}\, \cn(K(p_n),p_n) = \scurv_{0,n} \cos(\pi/2) = 0 \quad\text{and}\\
        \scurv_{\gamma_n}'(K(p_n)/r_n)&=-r_{n}\scurv_{0,n}\sn(K(p_n),p_n)\dn(K(p_n),p_n)=-r_n\scurv_{0,n}\sqrt{1-p_n^2}.
    \end{aligned}
    \end{equation}
    We therefore obtain
    \begin{equation}\label{eq:lf8-vert-0}
        \begin{aligned}
            \smash{\gamma_{n}'\left({K(p_n)}/{r_n}\right)} &= \nicefrac{(-a_{n}w_{n}^2+c_{n})}{(-\lambda_n-2ir_{n}\scurv_{0,n}\sqrt{1-p_n^2})}\\
            &= {\underbrace{\frac{(-a_{_n}w_{n}^2+c_{n})}{\lambda_n^2+4r_{n}^2\scurv_{0,n}^2(1-p_n^2)}}_{\in\,\R}}\cdot ( -\lambda_n + 2ir_{n}\scurv_{0,n}\sqrt{1-p_n^2} ).
        \end{aligned}
    \end{equation}
    By \Cref{rem:a-c-for-lf8}, $-a_nw_n^2+c_n>0$. So combined with \eqref{eq:lf8-vert-0}, 
    \begin{equation}\label{eq:lf8-ver-1}
        \text{$\Real(\gamma_{n}'(K(p_n)/r_n))<0$ and $\sgn(\scurv_{0,n})\cdot \Imag(\gamma_{n}'(K(p_n)/r_n))>0$}.
    \end{equation}
    Moreover,
    \begin{equation}\label{eq:lf8-vert}
        \sgn{\scurv_{0,n}}\cdot \frac{\Imag(\gamma_{n}'(K(p_n)/r_n))}{\Real(\gamma_{n}'(K(p_n)/r_n))} = -2r_{n}|\scurv_{0,n}|\frac{\sqrt{1-p_n^2}}{\lambda_n} \longrightarrow -\infty
    \end{equation}
    using \eqref{eq:lf8-vh-1} and the asymptotics for $r_n$ and $\scurv_{0,n}$ in \Cref{lem:ex-fig8}. Thus,
    \begin{equation}
        \frac{\gamma_{n}'(K(p_n)/r_n)}{|\gamma_{n}'(K(p_n)/r_n)|} = e^{i(\pi + \arctan\left(\frac{\Imag(\gamma_{n}'(K(p_n)/r_n))}{\Real(\gamma_{n}'(K(p_n)/r_n))}\right))},
    \end{equation}
    and the claim follows. The limit at $-{K(p_n)}/{r_n}$ follows by \eqref{eq:fig-8-sym}.
\end{proof}

\begin{remark}\label{rem:el-en-lf8}
    Let $\gamma_{\lambda}$ again denote a canonically parametrized $\lambda$-figure-eight and w.l.o.g. suppose that $\gamma_{\lambda}(0)=i$. By \eqref{eq:fig-8-sym} and since $\gamma_{\lambda}$ has period $4{K(p)}/{r}$ by \cite[Proposition 3.4]{muellerspener2020}, we obtain $\gamma_{\lambda}({2K(p)}/{r}) = -\overline{\gamma_{\lambda}(-{2K(p)}/{r})}=-\overline{\gamma_{\lambda}({2K(p)}/{r})}$. Therefore, $\gamma_{\lambda}(2{K(p)}/{r})\in i\R$ so that one can choose $\omega>0$ with $\gamma_{\lambda}(2{K(p)}/{r})=\omega i$.
    
    For $F_{\omega}\colon\H^2\to\H^2$, $z\mapsto -\omega/z$, write $\widetilde{\gamma}\defeq F_{\omega}\circ \gamma_{\lambda}(2K(p)/r-\cdot)$, it is clear that $\gamma_{\lambda}(0)=i$. Furthermore, as in the proof of \Cref{cor:lf8-vh}, one deduces that $\gamma_{\lambda}'(2K(p)/r) = (-a\omega^2+c)/(\scurv_0^2-\lambda)\in (0,\infty)$. As $|\gamma_{\lambda}'|_g\equiv 1$, this yields $\gamma_{\lambda}'(2K(p)/r)=\omega+0i$. Thus, $\widetilde{\gamma}'(0)=1+0i$. Using that $F_{\omega}$ is an isometry of $\H^2$, one can compute that $ \curv_{\widetilde{\gamma}}(x)= (dF_{\omega})_{\widetilde{\gamma}(x)}(\curv_{\gamma_{\lambda}}(2K(p)/r-x))$ as well as $N_{\widetilde{\gamma}} = -(dF_{\omega})_{\widetilde{\gamma}(x)}(N_{\gamma_{\lambda}}(2K(p)/r-x))$. So $\scurv_{\widetilde{\gamma}}(x)=-\scurv_{\gamma_{\lambda}}(2K(p)/r-x)=\scurv_{\gamma_{\lambda}}(x)$, using \Cref{prop:2.8}. By the fundamental theorem of curve theory,
    \begin{equation}
        F_{\omega}\circ \gamma_{\lambda}(2K(p)/r-x)=\widetilde{\gamma}(x)=\gamma_{\lambda}(x) \quad\text{for all $x\in\R$}.
    \end{equation} 
    So $\gamma_{\lambda}|_{[-K(p)/r,K(p)/r]}$ and $\gamma_{\lambda}|_{[K(p)/r,3K(p)/r]}$ have the same elastic energy. Since $\E(\gamma_{\lambda}|_{[-K(p)/r,3K(p)/r]})\searrow 16$ for $\lambda\searrow 0$ by \Cref{lem:ex-fig8}, we conclude that
    \begin{equation}
        \E_{\H^2}(\gamma_{\lambda}|_{[-K(p)/r,K(p)/r]})\searrow 8.
    \end{equation}
\end{remark}

\subsection{Construction of singular examples}

By \Cref{lem:ex-fig8} and \Cref{cor:lf8-vh}, we can consider a sequence $(\gamma_n)_{n\in\N}$ of canonically parametrized $\lambda_n$-figure-eights with $\lambda_n\searrow 0$ and $\scurv_{0,n}>0$ for each $n\in\N$ for which the tangent vectors at $\smash{\pm {K(p_n)}/{r_n}}$ satisfy 
\begin{equation}
    \lim_{n\to\infty} \gamma_n'(\pm K(p_n)/r_n)/|\gamma_n'(\pm K(p_n)/r_n)|=\pm (0,1)^t.
\end{equation} 

\textbf{Construction of the initial datum.} For $x\in\R$, consider the circle $C_x\subseteq\R^2$ with center $(x,\sqrt{2})^t$ and radius $1$. Further, for fixed $h > \sqrt{2}(1+\sqrt{2})/(\sqrt{2}-1)$, write $C'\defeq (h/\sqrt{2})\cdot C_0 + (-h/\sqrt{2},0)^t$. Note that, since scalings and translations in the first component are isometries of $\H^2$, both $C_x$ and $C'$ can be parametrized as circular elastica, i.e. the corresponding surfaces of revolution are Clifford tori. Thus, one has $|\curv|_g^2\equiv 2$ everywhere on $C_x$ and $C'$. 

A short computation yields that, for $x\in (0,1)$, $z^{(2)} > w^{(2)}$ for all $z\in C'$ and $w\in C_x$. Thus, for any $x\in (0,1)$, there is a largest value $\varepsilon_x\in (0,1)$ such that $\varepsilon_x\cdot C'$ touches $C_x$ tangentially in exactly one point $z_x^*$ with $(z_x^*)^{(1)}<0$. Denote by $z_x$ the point in $C_x\cap (\{0\}\times\R)$ with minimal second component. Then write $\Gamma_x$ for the curve which concatenates the segment of $C_x$ between $z_x$ and $z_x^*$ with the segment of $\varepsilon_x \cdot C'$ connecting $z_x^*$ and $\varepsilon_x\cdot(0,h)^t=\vcentcolon w_x$. Write $\Gamma_x'\defeq \{-\overline{z}\in\C:z\in\Gamma_x\}$. For an illustration, cf. \Cref{fig:CxandCprime}.

We now use this construction to suitably extend the $\lambda_n$-figure-eights $\gamma_n$. This will yield the singular initial data. By \eqref{eq:lf8-ver-1} and $\scurv_{0,n}>0$, for each $n\in\N$, there exists $x_n\in(0,1)$ such that ${\gamma_n'({K(p_n)}/{r_n})}\in T_{z_{x_n}} (C_{x_n})$. Since scaling does not affect the normalized tangent vectors, we may w.l.o.g. suppose that $\gamma_n$ is scaled such that ${\gamma_n( {K(p_n)}/{r_n})=z_{x_n}}$. 

\begin{figure}[htb]
    \centering
    \begin{subfigure}{5cm}
        \includegraphics[width=\linewidth]{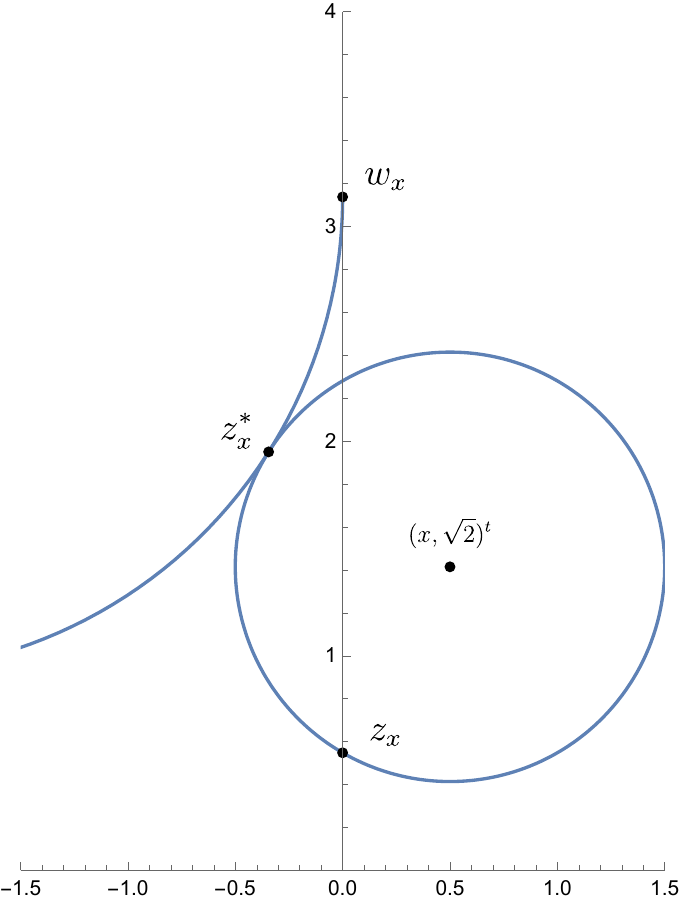}
        \caption{$x=\frac12$}
    \end{subfigure}
    \qquad
    \begin{subfigure}{5cm}
        \includegraphics[width=\linewidth]{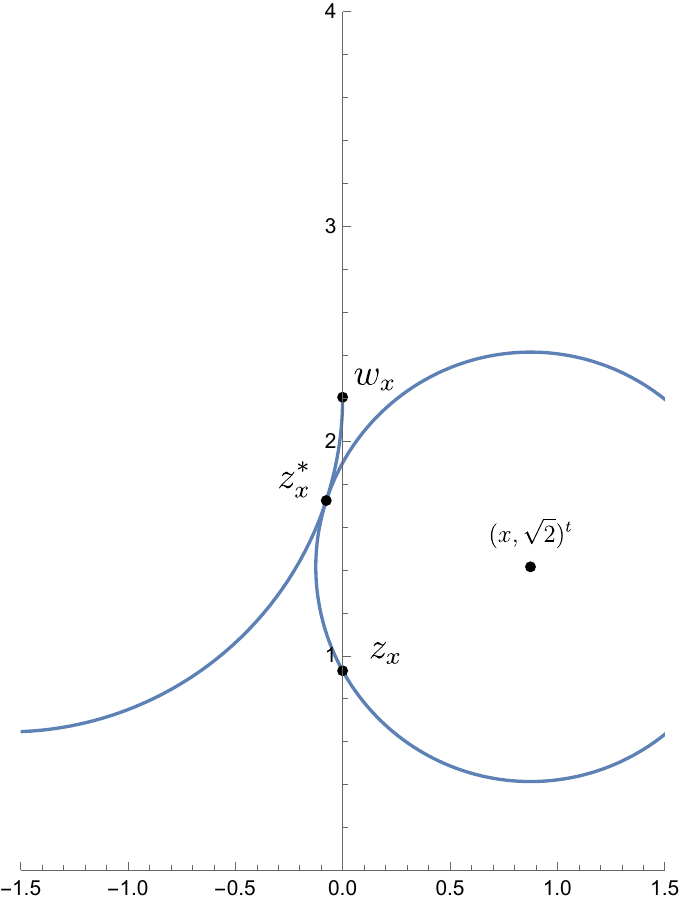}
        \caption{$x=\frac78$}
    \end{subfigure}
    \caption{Illustration of $C_x$ and a section of $\varepsilon_x\cdot C'$.}
    \label{fig:CxandCprime}
\end{figure}

\begin{remark}\label{rem:en-of-cons}
    It holds that $\E(\Gamma_{x_n})=\E(\Gamma_{x_n}')\to 0$ for $n\to\infty$. Firstly, since $\Gamma_{x_n}$ and $\Gamma_{x_n}'$ differ only by an isometry of $\H^2$, their elastic energies agree. Further, since $|\curv|_g^2\equiv 2$ on $\Gamma_{x_n}$ by construction, we only need to argue that the hyperbolic length of $\Gamma_{x_n}$ converges to $0$ for $n\to\infty$. An immediate geometric consequence of the above construction is that 
    \begin{equation}
        |\tau - \langle \tau,(0,1)^t\rangle(0,1)^t| \leq |\tau_n - (0,1)^t|
    \end{equation}
    for all $\tau\in T_p\Gamma_{x_n}$ with $|\tau|=1$ and for all $p\in\Gamma_{x_n}$ where $\tau_n=\gamma_n'(\frac{K(p_n)}{r_n})/|\gamma_n'(\frac{K(p_n)}{r_n})|$. That is, the least vertical of all tangent vectors in $\Gamma_{x_n}$ is the one at $z_{x_n}$. Since $\tau_n\to (0,1)^t$ and since normalized tangent vectors are invariant under scaling, the claim follows. 
\end{remark}

We construct the initial datum $u_0^n$ then as follows. First choose a suitable $W^{2,2}$-parametrization of $\Gamma_{x_n}$ starting in $w_{x_n}$ and ending in $z_{x_n}$. Then, concatenate with a suitable order-reversing reparametrization of $\gamma_n|_{[-K(p_n)/r_n,K(p_n)/r_n]}$ such that the concatenation is in $W^{2,2}$. Finally, concatenate again with a suitable $W^{2,2}$-parametrization of $\Gamma_{x_n}'$ starting in $z_{x_n}$ and ending in $w_{x_n}$ such that the entire curve is in $W^{2,2}$. Now rescale the entire curve we obtained by $1/(w_{x_n})^{(2)}$. We denote by $u_0^n$ the constant $g_{\H^2}$-speed reparametrization of the final curve on $[-1,1]$. Compare also \Cref{fig:singular_initial_datum} for plots of $u_0^n$ for some values of $\lambda_n$. By construction, one obtains
\begin{equation}\label{eq:constr-tv}
    u_0^n(\pm1) = (0,1)^t\quad\text{and}\quad  (\partial_s u_0^n)(\pm 1) = \pm (0,1)^t.
\end{equation}
Note that, by \Cref{rem:el-en-lf8,rem:en-of-cons}, $\E(u_0^n)\searrow 8$ for $n\to\infty$. 

\begin{figure}
    \centering
    \begin{subfigure}{4cm}
        \includegraphics[width=\linewidth]{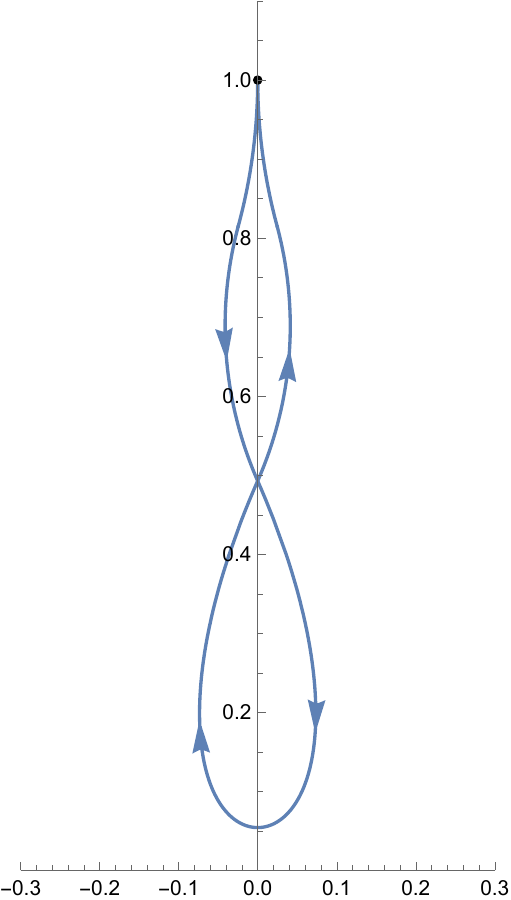}
        \caption{$\lambda_n=\frac25$}
    \end{subfigure}
    \qquad
    \begin{subfigure}{4cm}
        \includegraphics[width=\linewidth]{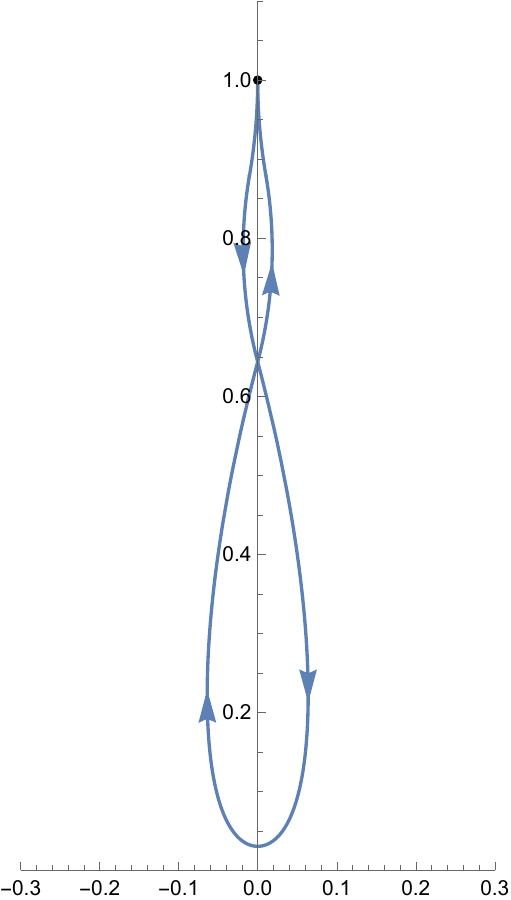}
        \caption{$\lambda_n=\frac{1}{10}$}
    \end{subfigure}
    \caption{Illustration of the initial datum $u_0^n$.}%
    \label{fig:singular_initial_datum}%
\end{figure}
 
\textbf{Assumption: There is no singularity along the Willmore flow.} By short-time existence, there is a maximal solution $u^n\colon [0,T_{\mathrm{max},n})\times I\to \H^2$ of the Willmore flow (i.e. \eqref{eq:lte-1} with $a(x,y)=-1/(4y^4)$) with initial datum $u_0^n$ as constructed above. The lack of smoothness poses no issue, cf. \Cref{rem:ste}. Suppose that $\sup_{t\in[0,T_{\mathrm{max},n})}\Ll_{\H^2}(u^n)<\infty$. By \Cref{thm:lte,thm:conv}, $T_{\mathrm{max},n}=\infty$ and the constant $g_{\H^2}$-speed reparametrizations of $u^n$ converge to a critical point $u^n_{\infty}$ of $\E$ parametrized by constant $g_{\H^2}$-speed still satisfying the Dirichlet boundary data induced by $u_{0}^n$, cf. \eqref{eq:constr-tv}. Moreover, $\E(u^n_{\infty})\leq\E(u_0^n)\to 8$ for $n\to\infty$. 

\begin{remark}
    We wish to obtain a contradiction by showing that $u_{\infty}^n$ cannot be a critical point of $\E$ (for $n$ sufficiently large). To this end, we proceed as follows.
    \begin{enumerate}[(I)]
        \item Argue that $u_{\infty}^n$ is necessarily a reparametrization of a simply closed, canonically parametrized free \emph{orbit-like} elastica.
        \item Show that there is no canonically parametrized segment of a free orbit-like elastica satisfying the boundary data in \eqref{eq:constr-tv} to reach a contradiction.
    \end{enumerate}
\end{remark}

\textbf{Ad (I).} By the boundary data \eqref{eq:constr-tv}, $u_{\infty}^n$ has a self-intersection. Consequently, as we will argue in \Cref{rem:exl-norb}, for $n$ sufficiently large, the only remaining category for $u_{\infty}^n$ is that of a free orbit-like elastica. 

Since the initial datum $u_0^n$ satisfies the symmetry relation $u_0^n(-x)=-\overline{u_0^n(x)}$ by construction and \Cref{lem:fig8-sym}, \Cref{lem:sym-pres} yields $u^n(t,-x) = -\overline{u^n(t,x)}$ for any $t\geq 0$ and $x\in[-1,1]$. So by the convergence and \Cref{rem:sym-un-speed}, 
\begin{equation}\label{eq:u-inf-sym}
    u^n_{\infty}(-x) = -\overline{u^n_{\infty}(x)}\quad\text{for all $x\in[-1,1]$}.
\end{equation}
We will show in \Cref{lem:exl-heart} that, choosing $n$ sufficiently large, this symmetry relation is sufficient for w.l.o.g. supposing that $u_{\infty}^n$ is given as a reparametrization of a canonically parametrized orbit-like elastica $\gamma_{\infty}^n$. Note that, by \eqref{eq:u-inf-sym}, $u_{\infty}^n(0)=iy$ and $\partial_su_{\infty}^n(0)=\pm y+0i$ for some $y>0$. Therefore, since $u_{\infty}^n$ is parametrized by constant $g_{\H^2}$-speed and since $\partial_su_{\infty}^n(0)= \pm y+0i$, we have $u_{\infty}^n(x) = \gamma_{\infty}^n(\pm \ell_n\cdot x)$ for $x\in[-1,1]$ where $\ell_n>0$ is the smallest positive number with $\gamma^n_{\infty}(\ell_n)=\gamma^n_{\infty}(-\ell_n)$.

\textbf{Ad (II).} Since $\gamma^n_{\infty}(\ell_n)=\gamma^n_{\infty}(-\ell_n)$, \Cref{rem:sym-scurv} yields $\gamma^n_{\infty}(\ell_n)=\omega i$ for some $\omega>0$. Moreover, since $\gamma^n_{\infty}$ is canonically parametrized, \eqref{eq:par-el-1} and \eqref{eq:def-theta} yield 
\begin{equation}\label{eq:ol-tan-vec}
    (\gamma_{\infty}^n)'(\ell_n) = \frac{-a\omega^2+c}{\scurv(\ell_n)^2+2i\scurv'(\ell_n)}=\frac{-a\omega^2+c}{\scurv(\ell_n)^4+(2\scurv'(\ell_n))^2}(\scurv(\ell_n)^2-2i\scurv'(\ell_n)).
\end{equation}
Particularly, since $|(\gamma_{\infty}^n)'(\ell_n)|_g=1$, we have $-a\omega^2+c\neq 0$. Further, \eqref{eq:constr-tv} yields $\Real((\gamma_{\infty}^n)'(\ell_n))=0$ so that \eqref{eq:ol-tan-vec} yields $\scurv(\ell_n)^2=0$. However, by \Cref{prop:2.8}, $\scurv^2(\ell_n)=\scurv_0^2\dn^2(r\ell_n,p) \neq 0$ by \Cref{def:jef}. A contradiction!

\subsection{Auxiliary results on free elastica}

The following closing-condition is a corollary of \Cref{thm:par-el}. For notation and some results on the elliptic integrals and functions below, refer to Appendix~\ref{app:ell}. Throughout this section, we use $C$ defined in \eqref{eq:char-el-0} and the notation of \Cref{prop:2.8}.

\begin{proposition}\label{cor:clos-cond}
    Consider a free elastica $\gamma\colon I\to\R$, $\alpha<\beta$ with $\alpha,\beta\in I$ and $\gamma(\alpha)=\gamma(\beta)$. Then $C<0$, i.e., $\gamma$ is either orbit-like or circular. If $\gamma$ is orbit-like and canonically parametrized, one obtains 
    \begin{equation}\label{eq:clos-cond-1}
        \begin{aligned}
            \pi\leq m\pi &= \frac{1}{2}\sqrt{1-p^2}\sqrt{2-p^2}\int_{\am(r\alpha,p)}^{\am(r\beta,p)} \frac{\sqrt{1-p^2\sin^2(\theta)}}{1-p^2(2-p^2)\sin^2(\theta)}\dd\theta\\
        \end{aligned}
    \end{equation}
    for some $m\in\N$. If $\gamma$ is circular, $\E(\gamma)\geq 4\pi$.
\end{proposition}
\begin{proof}
    Consider the case where $\scurv$ is non-constant. By \Cref{lem:isom-hyp} and \Cref{thm:par-el}, there is an isometry $\Phi\colon\H^2\to\H^2$ such that $\Phi\gamma(s_*)=i$ and $d\Phi_{\gamma(s_*)}(\gamma'(s_*))=1$ for a fixed $s_{*}\in\R$ with ${\scurv}(s_*)^2=\scurv_0^2$. Then
    \begin{equation}\label{eq:clos-cond-1.1}
        \Phi\gamma(x) = f\left(\int_{s_*}^x\frac{1}{\theta}\dd s+z_1\right)
    \end{equation}
    with $f$ as in \eqref{eq:par-el-3} and $\theta$ as in \eqref{eq:def-theta}. Now observe that, by \eqref{eq:def-theta}, \eqref{eq:char-el-0} and \Cref{thm:par-el}, $4\scurv^2+4C = |\theta|^2>0$. So
    \begin{equation}\label{eq:clos-cond-3}
        \frac{1}{\theta(s)} = \frac{{\scurv}^2-2i{\scurv}'}{{\scurv}^4+4({\scurv}')^2} \underset{\eqref{eq:char-el-0}}{=} \frac{{\scurv}^2}{4C+4{\scurv}^2} +i \frac{-2{\scurv}'}{4C+4{\scurv}^2}.
    \end{equation}    
    Note that, in either of the last two cases of \eqref{eq:par-el-3}, $\gamma(\alpha)=\gamma(\beta)$ yields
    \begin{equation}\label{eq:clos-cond-2}
        \int_{\alpha}^{\beta} \frac{1}{\theta}\dd s = 0. 
    \end{equation}
    In either of those cases, \eqref{eq:clos-cond-2} and \eqref{eq:clos-cond-3} together yield $\scurv\equiv 0$, a contradiction. 

    Now suppose that $ac<0$. Since $\tanh(z)=\tanh(w)$ if and only if $z-w\in i\pi \Z$, \eqref{eq:clos-cond-1.1}, \eqref{eq:clos-cond-3} and $C=-ac$ yield
    \begin{equation}
        i\pi \Z\ni \sqrt{C}\int_{\alpha}^{\beta} \frac{1}{\theta}\dd s = \sqrt{C}\int_{\alpha}^{\beta} \frac{\scurv^2}{4C+4\scurv^2}\dd s +i\sqrt{C}\int_{\alpha}^{\beta}\frac{-2\scurv'}{4C+4\scurv^2}\dd s.
    \end{equation}
    Particularly, we again have $\scurv\equiv 0$, a contradiction. 
    
    Therefore, only the first two cases in \eqref{eq:par-el-3} remain and $C=-ac <0$, i.e., $\gamma$ is necessarily orbit-like. Moreover, $\tan(z)=\tan(w)$ if and only if $\cot(z)=\cot(w)$ if and only if $z-w\in \pi\Z$ for $z,w\in \C$. Therefore, \eqref{eq:clos-cond-1.1} and \eqref{eq:clos-cond-3} yield again
    \begin{equation}
        \pi \Z\ni \sqrt{-C}\int_{\alpha}^{\beta} \frac{1}{\theta}\dd s = \sqrt{-C}\int_{\alpha}^{\beta} \frac{\scurv^2}{4C+4\scurv^2}\dd s +i\sqrt{-C}\int_{\alpha}^{\beta}\frac{-2\scurv'}{4C+4\scurv^2}\dd s.
    \end{equation}
    Particularly,
    \begin{equation}\label{eq:clos-cond-4}
        \int_{\alpha}^{\beta}\frac{-2\scurv'}{4C+4\scurv^2}\dd s = 0 \quad \text{and}\quad
        m\pi = \sqrt{-C }\int_{\alpha}^{\beta} \frac{\scurv^2}{4C+4\scurv^2}\dd s,
    \end{equation}
    for some $m\in\N$. If $\gamma$ is canonically parametrized, by \Cref{prop:2.8}, using 
    \begin{equation}
        \text{$\scurv^2(s)=\scurv_0^2\dn^2(rs,p)=\scurv_0^2(1-p^2\sin^2(\am(rs,p)))$ and $\partial_s\am(s,p)=\dn(s,p)$,}
    \end{equation} we have 
    \begin{equation}
        \sqrt{-C}\int_{\alpha}^{\beta} \frac{\scurv^2}{4C+4\scurv^2}\dd s = \frac{\scurv_0^2\sqrt{-C}}{r} \int_{\am(r\alpha,p)}^{\am(r\beta,p)} \frac{\sqrt{1-p^2\sin^2(\theta)}}{4C+4\scurv_0^2(1-p^2\sin^2(\theta))}\dd\theta.
    \end{equation}
    Moreover, \Cref{prop:2.8} yields the following relations:
    \begin{equation}\label{eq:exl-orb-2-1}
        \scurv_0^2=\frac{4}{2-p^2},\quad
             \sqrt{-C} = (\scurv_0^2-\frac{1}{4}\scurv_0^4)^{\frac12} = \frac{2\sqrt{1-p^2}}{2-p^2},\quad
             r = \frac{1}{\sqrt{2-p^2}}
    \end{equation}
    and finally, using \eqref{eq:char-el-0},
    \begin{equation}\label{eq:exl-orb-2}
        \begin{aligned}
             4C+4\scurv_0^2(1-p^2\sin^2(\theta)) &= \scurv_0^4-4\scurv_0^2p^2\sin^2(\theta)\\
             &= \frac{16}{(2-p^2)^2}\cdot (1 - p^2(2-p^2)\sin^2(\theta)).
        \end{aligned}
    \end{equation}
    Altogether, we conclude  \eqref{eq:clos-cond-1}. Now consider the case where $\scurv$ is constant. If $\scurv\equiv 0$, then $\gamma$ has no self-intersections. Otherwise, $\gamma$ is circular and $\E(\gamma)\geq 4\pi$ by \Cref{rem:en-circ-el}. 
\end{proof}

\begin{remark}\label{rem:exl-norb}
    \Cref{cor:clos-cond} immediately yields the following. 
    If we are considering (segments of) free elastica with self-intersections whose elastic energy is sufficiently close to $8$, only orbit-like elastica need to be taken into account. So we collect further properties of orbit-like elastica in the next lemma.
\end{remark}

\begin{lemma}\label{lem:exl-orb}
    For a free canonically parametrized orbit-like elastica $\gamma\colon I\to\H^2$, $\gamma(\alpha)=\gamma(\beta)$ for $\alpha<\beta$ already implies that $\am(r\beta,p)-\am(r\alpha,p) > \pi$. If additionally $\alpha',\beta'\in I$ with $\alpha'\leq\alpha<\beta<\beta'$ and $\gamma(\alpha')=\gamma(\beta')$, then $\am(r\beta',p)-\am(r\alpha',p)>2\pi$. 
\end{lemma}
\begin{proof}
    We study the term on the right-hand-side of \eqref{eq:clos-cond-1}. Firstly,
    \begin{equation}\label{eq:exl-orb-3-1}
    \begin{aligned}
        \frac{1}{2}\sqrt{1-p^2}&\sqrt{2-p^2}\int_{-\frac{\pi}{2}}^{\frac{\pi}{2}} \frac{\sqrt{1-p^2\sin^2(\theta)}}{1-p^2(2-p^2)\sin^2(\theta)}\dd\theta\\&=\sqrt{1-p^2}\sqrt{2-p^2}\int_0^{\frac{\pi}{2}} \frac{\sqrt{1-p^2\sin^2(\theta)}}{1-p^2(2-p^2)\sin^2(\theta)}\dd\theta.\\
    \end{aligned}
    \end{equation}
    Moreover, by \Cref{def:ell-int} and a short computation,
    \begin{equation}\label{eq:exl-orb-3-2}
        K(p)+ (1-p^2)\Pi(p^2(2-p^2),p) 
        = (2-p^2)\int_0^{\frac{\pi}{2}} \frac{\sqrt{1-p^2\sin^2(\theta)}}{1-p^2(2-p^2)\sin^2(\theta)}\dd\theta.
    \end{equation}
    Combining \eqref{eq:exl-orb-3-1} and \eqref{eq:exl-orb-3-2}, one thus obtains from \eqref{eq:bf-2} that
    \begin{equation}\label{eq:exl-orb-3}
    \begin{aligned}
        \frac{1}{2}\sqrt{1-p^2}&\sqrt{2-p^2}\int_{-\frac{\pi}{2}}^{\frac{\pi}{2}} \frac{\sqrt{1-p^2\sin^2(\theta)}}{1-p^2(2-p^2)\sin^2(\theta)}\dd\theta\\
        &< \frac{(1-p^2)^{\frac{3}{2}}}{\sqrt{2-p^2}} \cdot \Big( \frac{K(p)}{1-p^2} + \frac{\pi}{2}\sqrt{\frac{p^2(2-p^2)}{(1-p^2(2-p^2))(p^2-p^4)}} \Big)
    \end{aligned}
    \end{equation}
    for all $p\in (0,1)$. 
    Next, observe that, by \eqref{eq:ieq-Kp}, 
    \begin{equation}\label{eq:ol-nh-4.02}
        \frac{\sqrt{1-p^2}}{\sqrt{2-p^2}}K(p) \leq \min\Big\{\frac{\pi}{2\sqrt{2-p^2}},\sqrt{1-p^2}K(p)\Big\}.
    \end{equation}
    As by \eqref{eq:as-Kp} $\sqrt{1-p^2}K(p)\to 0$ for $p\nearrow 1$, \eqref{eq:ol-nh-4.02} shows that there is $\delta\in (0,1)$ with
    \begin{equation}\label{eq:ol-nh-4.03}
        \frac{\sqrt{1-p^2}}{\sqrt{2-p^2}}K(p) \leq \delta\cdot\frac{\pi}{2}\quad\text{for all $p\in (0,1)$}.
    \end{equation}
    Combining \eqref{eq:exl-orb-3} and \eqref{eq:ol-nh-4.03}, we obtain for all $p\in (0,1)$
    \begin{equation}\label{eq:exl-orb-4}
        \frac{1}{2}\sqrt{1-p^2}\sqrt{2-p^2}\int_{-\frac{\pi}{2}}^{\frac{\pi}{2}} \frac{\sqrt{1-p^2\sin^2(\theta)}}{1-p^2(2-p^2)\sin^2(\theta)}\dd\theta < (1+\delta)\cdot\frac{\pi}{2}.
    \end{equation}
    
    Since $\sin^2$ is $\pi$-periodic and by \eqref{eq:exl-orb-4}, \eqref{eq:clos-cond-1} cannot be satisfied if $\am(r\beta,p)-\am(r\alpha,p)\leq \pi$. 
    
    Now additionally considering $\alpha'\leq\alpha<\beta<\beta'$ with $\gamma(\alpha')=\gamma(\beta')$, \eqref{eq:clos-cond-1} yields
    \begin{equation}\begin{aligned}
            \pi &\leq \frac{1}{2}\sqrt{1-p^2}\sqrt{2-p^2}\int_{\am(r\alpha,p)}^{\am(r\beta,p)} \frac{\sqrt{1-p^2\sin^2(\theta)}}{1-p^2(2-p^2)\sin^2(\theta)}\dd\theta \\
            &<\frac{1}{2}\sqrt{1-p^2}\sqrt{2-p^2}\int_{\am(r\alpha',p)}^{\am(r\beta',p)} \frac{\sqrt{1-p^2\sin^2(\theta)}}{1-p^2(2-p^2)\sin^2(\theta)}\dd\theta= \widetilde{m}\pi
    \end{aligned}
    \end{equation}
    for some $\widetilde{m}\in\N$. Particularly, $\widetilde{m}\geq 2$. 
    Again by \eqref{eq:exl-orb-4} and the $\pi$-periodicity of $\sin^2$, $\am(r\beta',p)-\am(r\alpha',p)>2\pi$ which shows the remainder of the claim.
\end{proof}

\begin{lemma}\label{lem:en-orb}
    Let $\gamma\colon[\alpha,\beta]\to\H^2$ be the canonical parametrization of an orbit-like elastica with parameter $p\in(0,1)$ such that $\am(r\beta,p)-\am(r\alpha,p)>m\pi$ for some $m\in\N$. Then $\E(\gamma)> 8m$.
\end{lemma}
\begin{proof}
    Using \Cref{prop:2.8}, we have $\scurv(s)^2 = \scurv_0^2 \cdot \dn^2(rs,p)\text{ for all $s\in[\alpha,\beta]$}$ where $\scurv_0^2 = \frac{4}{2-p^2}$ and $r=\frac{1}{\sqrt{2-p^2}}$. Consequently, by the $\pi$-periodicity of $\sin^2$,
    \begin{align}
        & \E(\gamma) = \int_{\alpha}^{\beta} \scurv^2(s)\dd s = \sqrt{2-p^2}\frac{4}{2-p^2} \int_{r\alpha}^{r\beta} \dn^2(s,p)\dd s\label{eq:en-orb-2} \\
        &= \frac{4}{\sqrt{2-p^2}}\int_{\am(r\alpha,p)}^{\am(r\beta,p)} \sqrt{1-p^2\sin^2(\theta)}\dd\theta \geq \frac{8m}{\sqrt{2-p^2}}E(p) \underset{\eqref{eq:ieq-Ep}}{>}8m.
        \qedhere
    \end{align}
\end{proof}

In the proof of (I) in the previous section, we argue that the limit $u_{\infty}^n$ satisfies the same symmetry relation as $u_0^n$. In order to fully classify the limit with this information, we require the following.

\begin{lemma}\label{lem:exl-heart}
    Let $\gamma\colon[-L,L]\to\H^2$ parametrize a free orbit-like elastica with $\gamma(-L)=\gamma(L)$ and with parameter $p\in (0,1)$. Moreover, suppose that
    \begin{equation}\label{eq:exl-heart-1}
        \gamma(-x) = -\overline{\gamma(x)} \quad\text{for all $x\in[-L,L]$}.
    \end{equation}
    Then $\gamma(0)= iy$ and $\gamma'(0)=\pm y+0i$ for some $y>0$. Moreover, we either have $\scurv(0)^2=\scurv_0^2(1-p^2)$ or $\scurv(0)^2=\scurv_0^2$ with $\scurv_0^2$ as in \Cref{prop:2.8}. 
    
    Lastly, there exists $\varepsilon>0$ such that, for any $\gamma$ as above with $\scurv(0)^2=\scurv_0^2(1-p^2)$, 
    \begin{equation}
        \E(\gamma)\geq 8+\varepsilon.
    \end{equation}
\end{lemma}
\begin{proof}
    Indeed, by \Cref{rem:sym-scurv} and \eqref{eq:exl-heart-1}, $\scurv^2$ needs to be an even function. As $\scurv(s)^2=\scurv_0^2\dn(rs,p)^2$ by \Cref{prop:2.8}, one concludes that $\scurv(0)^2$ either is a local minimum or maximum of $\scurv^2$. Thus, using \Cref{def:jef}, $\scurv(0)^2\in\{\scurv_0^2,\scurv_0^2(1-p^2)\}$.
    
    Now suppose that $\scurv(0)^2=\scurv_0^2(1-p^2)$. Particularly, \Cref{prop:2.8} yields that $\scurv$ is a translation of $s\mapsto \scurv_0\cdot \dn(rs,p)$. As $\dn(x,p)=\sqrt{1-p^2}$ if and only if $\am(x,p)\in \pi/2+\pi\Z$ if and only if $x\in K(p)+2K(p)\Z$ by \Cref{def:jef} and \Cref{rem:der-jef}, $\scurv(s) = \scurv_0 \cdot \dn(rs+K(p),p)$.     
    Moreover, using $m\in\N$ in \eqref{eq:clos-cond-4}, the assumption $\gamma(-L)=\gamma(L)$ yields that
    \begin{equation}\label{eq:exl-heart-2}
        \pi \leq \sqrt{-C}\int_{-L}^{L} \frac{\scurv(s)^2}{4C+4\scurv(s)^2}\dd s.
    \end{equation}
    Thus, using \eqref{eq:exl-orb-2-1} and \eqref{eq:exl-orb-2} and \Cref{rem:der-jef}, 
    \begin{equation}\label{eq:exl-heart-3}
        \begin{aligned}
            \sqrt{-C}\int_{-L}^L&\frac{\scurv(s)^2}{4C+4\scurv(s)^2}\dd s 
            = \frac{\sqrt{-C}\scurv_0^2}{r} \int_{-rL+K(p)}^{rL+K(p)} \frac{\dn(s,p)\partial_s\am(s,p)}{4C+4\scurv_0^2\dn^2(s,p)}\dd s\\
            &= 8 \frac{\sqrt{1-p^2}}{(2-p^2)^{3/2}} \int_{\am(-Lr+K(p),p)}^{\am(Lr+K(p),p)} \frac{\sqrt{1-p^2\sin^2(\theta)}}{4C+4\scurv_0^2(1-p^2\sin^2(\theta))}\dd\theta\\
            &= \frac{1}{2} \sqrt{1-p^2}\sqrt{2-p^2} \int_{\am(-Lr+K(p),p)}^{\am(Lr+K(p),p)} \frac{\sqrt{1-p^2\sin^2(\theta)}}{1-p^2(2-p^2)\sin^2(\theta)}\dd\theta.
        \end{aligned}
    \end{equation}
    Moreover, note that, for $0<\delta<\pi/4$,
    \begin{equation}
    \begin{aligned}
        \int_{-\delta}^0&\frac{\sqrt{1-p^2\sin^2(\theta)}}{1-p^2(2-p^2)\sin^2(\theta)}\dd\theta = 
        \int_0^{\delta}\frac{\sqrt{1-p^2\sin^2(\theta)}}{1-p^2(2-p^2)\sin^2(\theta)}\dd\theta \\
        &= \int_0^{\delta} \frac{\sqrt{1+(1-p^2)\tan^2(\theta)}\sqrt{1+\tan^2(\theta)}}{1+(1-p^2)^2\tan^2(\theta)}\dd\theta\\
        &\leq \int_0^{\tan(\delta)} \frac{\sqrt{1+(1-p^2)s^2}}{1+(1-p^2)^2s^2}\dd s \leq \int_0^{\tan(\delta)} \frac{1+(1-p^2)s^2}{1+(1-p^2)^2s^2}\dd s \leq 2\tan(\delta)\to 0
    \end{aligned}
    \end{equation}
    as $\delta\searrow 0$. Combining the previous estimates and using \eqref{eq:exl-orb-4}, we obtain that there exists $\delta>0$ sufficiently small such that, for all $p\in (0,1)$, 
    \begin{equation}\label{eq:exl-heart-4}
        \frac{1}{2} \sqrt{1-p^2}\sqrt{2-p^2} \int_{-\delta}^{\pi+\delta} \frac{\sqrt{1-p^2\sin^2(\theta)}}{1-p^2(2-p^2)\sin^2(\theta)}\dd\theta < \pi.
    \end{equation}
    Using \cite[(122.03)]{byrdfriedman1954}, $(\am(-Lr+K(p),p),\am(Lr+K(p),p))$ is an open interval centered at $\frac{\pi}{2}$. In view of \eqref{eq:exl-heart-2} and \eqref{eq:exl-heart-3}, by \eqref{eq:exl-heart-4} and since this interval is the domain of integration on the right-hand-side of \eqref{eq:exl-heart-3},
    \begin{equation}
        (\am(-Lr+K(p),p),\am(Lr+K(p),p)) \supseteq (-\delta,\pi+\delta).
    \end{equation}
    Consequently, similarly as in \eqref{eq:en-orb-2},
    \begin{align}
        \E(\gamma) &= \frac{4}{\sqrt{2-p^2}} \int_{\am(Lr+K(p),p)}^{\am(-Lr+K(p),p)} \sqrt{1-p^2\sin^2(\theta)}\dd\theta \\
        &\geq \frac{4}{\sqrt{2-p^2}} \int_{0}^{\pi} \sqrt{1-p^2\sin^2(\theta)}\dd\theta + \frac{8}{\sqrt{2-p^2}}\int_0^{\delta} \sqrt{1-p^2\sin^2(\theta)}\dd\theta\\
        &\geq 8 + \frac{8}{\sqrt{2}} \int_0^{\delta} \sqrt{1-\sin^2{\theta}}\dd\theta =\vcentcolon 8+\varepsilon.\qedhere
    \end{align}
\end{proof}


\section{A Li-Yau inequality for open curves in $\H^2$}\label{sec:liyau}

A  similar analysis to this section for \emph{closed} planar curves is undertaken in \cite{muellerrupp2021}. 

Set $W_{\mathrm{imm}}^{2,2}([0,1],\H^2)\defeq \{u\in W^{2,2}([0,1],\R^2):u([0,1])\subseteq\H^2,\text{ $u$ is immersed}\}$, equipped with the Euclidean Sobolev norm. 

\begin{theorem}\label{thm:li-yau}
    Any $u\in W_{\mathrm{imm}}^{2,2}([0,1],\H^2)$ with $\E(u)\leq 8$ is an embedding.
\end{theorem}
\begin{proof}
    Define 
    \begin{equation}
        \A \defeq \{ \gamma\in W_{\mathrm{imm}}^{2,2}([0,1],\H^2): \gamma\text{ non-embedded, }\partial_s^k\gamma(y)=\partial_s^ku(y)\text{ for }y,k=0,1 \}.
    \end{equation}
    We claim that $m\defeq \inf_{\gamma\in\A}\E(\gamma)\geq 8$. 

    To this end, consider a minimizing sequence $(\widetilde{\gamma}_n)_{n\in\N}\subseteq\A$ with $\E(\widetilde{\gamma}_n)\to m$. Suppose for contradiction that $m<8$. W.l.o.g., $\sup_{n\in\N} \E(\widetilde{\gamma}_n)<8$. Define $\gamma_n$ to be the constant $g_{\H^2}$-speed reparametrization of $\widetilde{\gamma}_n$. Further, using \Cref{cor:el-ctrl-len}, one obtains
    \begin{equation}\label{eq:liyau-1}
        \frac{1}{\gamma_n^{(2)}}|\partial_x\gamma_n| = |\partial_x\gamma_n|_g = \Ll_{\H^2}(\gamma_n) =\vcentcolon \ell_n \leq C
    \end{equation}
    for some $C>0$ and for all $n\in\N$. Further, by \Cref{rem:len-cpct-ctrl}, there are $L,l>0$ such that $l\leq \gamma_n^{(2)} \leq L$ uniformly in $n$. Thus, by \eqref{eq:liyau-1},  $|\partial_x\gamma_n|\leq L \ell_n$. Moreover, since $\partial_{s_{\gamma_n}} = \frac{1}{\ell_n}\partial_x$ and since, by \eqref{eq:cov-der-h2} and \eqref{eq:def-curv},
    \begin{equation}
        \ell_n^2 \curv_n = \partial_x^2\gamma_n + \frac{1}{\gamma_n^{(2)}}
        \begin{pmatrix}
            -2\partial_x\gamma_n^{(1)}\partial_x\gamma_n^{(2)}\\
            (\partial_x\gamma_n^{(1)})^2-(\partial_x\gamma_n^{(2)})^2
        \end{pmatrix},
    \end{equation}
    we have
    \begin{equation}
        \int_0^1 |\partial_x^2\gamma_n|^2\dd x \leq 2\ell_n^4 \int_0^1  |\curv_n|^2\dd x + {8\ell_n^4L^2} \leq 2\ell_n^4L^2 \int_0^1  |\curv_n|_g^2\dd s + {8\ell_n^4L^2}\leq 24C^4L^2
    \end{equation}
    uniformly in $n$. Consequently, $(\gamma_n)_{n\in\N}\subseteq W^{2,2}$ is uniformly bounded so that, w.l.o.g. passing to a subsequence, $\gamma_n\rightharpoonup \gamma$ for some $\gamma\in W^{2,2}$. Since $W^{2,2}([0,1],\R^2)\hookrightarrow C^1([0,1])$ compactly, we may also suppose that $\gamma_n\to\gamma$ in $C^1$. As $\gamma_n^{(2)}\geq l>0$, $\gamma\colon[0,1]\to\H^2$. 

    By \cite[Lemma 4.3 and Remark 4.4]{muellerrupp2021}, $\gamma$ is non-embedded. As at the beginning of the proof of \Cref{thm:lte}, one deduces from the bound on the elastic energy that the $g_{\H^2}$-length $\ell_n$ of each $\gamma_n$ is bounded from below uniformly away from $0$. Therefore, since $|\partial_x\gamma_n|_g=\ell_n$, we obtain $|\partial_x\gamma|_g>0$. Particularly, $\gamma$ is also immersed and $\gamma\in \A$. Furthermore, $\E(\gamma)\leq \liminf_{n\to\infty}\E(\gamma_n) = m$. As $\gamma$ is not embedded, there are $0\leq x_1<x_2\leq 1$ with $\gamma(x_1)=\gamma(x_2)$. 
    \paragraph{Claim:} $\gamma|_{[x_1,x_2]}$ is a reparametrization of a segment of a free elastica.

    To this end, consider $\varphi\in C_c^{\infty}((x_1,x_2),\R^2)$ arbitrary. For $|t|$ sufficiently small, $\gamma+t\varphi\in \A$.
    By the minimizing property of $\gamma$,
    \begin{equation}
        0 = \frac{\dd}{\dd t}\Big|_{t=0} \E(\gamma+t\varphi) = \frac{\dd}{\dd t}\Big|_{t=0} \E((\gamma+t\varphi)|_{(x_1,x_2)}).
    \end{equation}
    Now the arguments in \cite[Section 5]{eichmanngrunau2019} yield that $\gamma|_{(x_1,x_2)}\in C^{\infty}$ solves $\nabla\E(\gamma)=0$ on $(x_1,x_2)$. We wish to verify that $\gamma$ remains smooth also on the closure $[x_1,x_2]$. To this end, since $\gamma$ is again parametrized with constant $g_{\H^2}$-speed, \Cref{lem:ext-el} yields that there is a smooth extension $\widetilde{\gamma}\colon\R\to\H^2$ of $\gamma$. So $\gamma\in C^{\infty}([x_1,x_2])$ and the claim follows. 
    
    By \Cref{cor:clos-cond}, \Cref{lem:exl-orb} and \Cref{lem:en-orb}, we obtain $\E(\gamma|_{[x_1,x_2]})>8$; thus contradicting the assumption $\E(\gamma)<8$. Altogether, we have $m\geq 8$.
    
    This proves that any $u\in W^{2,2}_{\mathrm{imm}}([0,1],\H^2)$ with $\E(u)<8$ is necessarily embedded. Now suppose that $\E(u)=8$ and w.l.o.g. $u(0)=u(1)$. By \Cref{rem:ste}, there is a solution $h\colon [0,T)\times [0,1]\to\H^2$ of the clamped elastic flow, i.e., \eqref{eq:lte-1} with $a\equiv -1$, such that $h(0,\cdot)=u$ and $\partial_x^kh(t,y)=\partial_x^ku(y)$ for $y,k=0,1$. 

    Firstly, suppose that $h$ is non-constant. That is, using \eqref{eq:1-var-elen} there exists $t'>0$ with $\partial_t \E(h(t,\cdot)) |_{t=t'}<0$. Hence, $\E(h(t,\cdot))<\E(h(t',\cdot))\leq 8$ for all $t>t'$. However, $h(t,0)=u(0)=u(1)=h(t,1)$ for such $t$, a contradiction to the above.
    
    Conversely, suppose that $h(t,\cdot)=u$ for all $t\geq 0$. Especially, by \Cref{rem:ste}, $u\in C^{\infty}([0,1])$ and, by \eqref{eq:lte-1}, $u$ is a reparametrization of a segment of a free elastica. However, by \Cref{cor:clos-cond}, \Cref{lem:exl-orb} and \Cref{lem:en-orb}, there is no non-embedded segment of a free elastica with elastic energy below or equal to $8$. Thus, $\E(u) > 8$, a contradiction. 
\end{proof}

\begin{remark}\label{rem:li-yau-opti}
    The energy threshold of $8$ in \Cref{thm:li-yau} is optimal in its generality. Indeed, using a sequence $(\gamma_n)_{n\in\N}$ of $\lambda_n$-figure-eights as in \Cref{lem:ex-fig8} with $\lambda_n\searrow 0$, one can consider $\gamma_n|_{\left[-{K(p_n)}/{r_n},{K(p_n)}/{r_n}\right]}$, cf. \Cref{rem:el-en-lf8}.
\end{remark}

\begin{remark}
    Recall that, by the classical Li-Yau inequality for immersed surfaces in $\R^3$ (cf. \cite{liyau1982}) and \eqref{eq:el-vs-will}, one obtains that any closed, immersed $u\in W^{2,2}(\S^1,\H^2)$ with $\E(u)<16$ is necessarily embedded. Our sufficient energy threshold is lower, but the result allows for a larger class of curves. In fact, the Li-Yau inequality for closed curves can be concluded from our version.
\end{remark}

\begin{corollary}
    Let $u_0\colon I \to \H^2$ be a smooth immersion with $\E(u_0)\leq 8$. The solution $u$ to \eqref{eq:lte-1} exists globally, converges to an embedded critical point of $\E$ after reparametrization and is embedded for all times $t\geq 0$.
\end{corollary}

\begin{remark}
    Alternatively, one could try to extend the Li-Yau inequality known for closed surfaces to one for immersed surfaces with boundary and, in the setting of \Cref{thm:li-yau}, apply it to the surface of revolution $f_u$. Following the approach involving Simon's monotonicity formula (cf. \cite[Equation (2.1)]{novagapozzetta2020}), one obtains for an immersed surface $f\colon \Sigma\to\R^3$ with boundary that
    \begin{equation}
        \Haus^0(f^{-1}(p))\pi \leq \frac{1}{4}\W(f) + \frac{1}{2} \int_{\partial\Sigma} \frac{\frac{\partial f}{\partial\eta} \cdot (f-p)}{|f-p|^2} \dd \sigma_f
    \end{equation}
    for $p\in f(\Sigma)$, the outer conormal $\frac{\partial f}{\partial \eta}$ at the boundary and the induced volume element $\sigma_f$ on $\partial\Sigma$. One obtains that, if 
    \begin{equation}
        \W(f) <8\pi-2\sup_{p\in\R^3}\int_{\partial\Sigma} \frac{\frac{\partial f}{\partial\eta} \cdot (f-p)}{|f-p|^2} \dd \sigma_f,
    \end{equation}
    then $f$ is necessarily an embedding.  As the boundary is at least $C^2$, one can argue that the integral is finite for any $p\in\R^3$, cf. \cite[Equation (4.2)]{novagapozzetta2020}. Moreover, since the integrand converges to $0$ for $|p|\to\infty$, the supremum in the above is finite.
\end{remark}

\appendix

\section{Elliptic integrals and Jacobi elliptic functions}\label{app:ell}

\begin{definition}[Elliptic integrals]\label{def:ell-int}
    Let $\alpha,p\in[0,1)$ and $\varphi\in\R$. Then 
    \begin{enumerate}[(a)]
        \item $\displaystyle K(\varphi,p)\defeq \int_0^{\varphi} \frac{1}{\sqrt{1-p^2\sin^2(\theta)}}\dd\theta$ is the elliptic integral of first kind,
        \item $\displaystyle E(\varphi,p)\defeq \int_0^{\varphi} \sqrt{1-p^2\sin^2(\theta)}\dd\theta$ the elliptic integral of second kind and
        \item $\displaystyle \Pi(\varphi,\alpha^2,p)\defeq \int_0^{\varphi} \frac{1}{(1-\alpha^2\sin^2(\theta))\sqrt{1-p^2\sin^2(\theta)}}\dd\theta$ the elliptic integral of third kind.
    \end{enumerate}
    In case $\varphi=\frac{\pi}{2}$, one omits $\varphi$ in the above notation and calls the respective terms \emph{complete} elliptic integrals.
\end{definition}

\begin{remark}\label{rem:rem-a2}
    As on \cite[p.19]{langersinger1984}, one can show that
    \begin{equation}\label{eq:ieq-Ep}
        E(p) > \sqrt{2-p^2}
    \end{equation}
    for all $p\in (0,1)$. Moreover, we clearly have
    \begin{equation}\label{eq:ieq-Kp}
        \sqrt{1-p^2}K(p) = \int_0^{\frac{\pi}{2}} \frac{\sqrt{1-p^2}}{\sqrt{1-p^2\sin^2(\theta)}}\dd\theta \leq \int_0^{\frac{\pi}{2}} 1\dd\theta = \frac{\pi}{2}.
    \end{equation}
    Lastly, for $0<p^2<\alpha^2<1$ and $0\leq \varphi<\frac{\pi}{2}$, \cite[(117.02)]{byrdfriedman1954} yields
    \begin{equation}\label{eq:bf-1}
        \Pi(\alpha^2,p) + \underbrace{\Pi(p^2/\alpha^2,p)}_{\geq \, K(p)} = K(p) + \frac{\pi}{2}\sqrt{\frac{\alpha^2}{(1-\alpha^2)(\alpha^2-p^2)}}
    \end{equation}
    and thus, especially,
    \begin{equation}\label{eq:bf-2}
        \Pi(\alpha^2,p) \underset{\eqref{eq:bf-1}}{\leq} \frac{\pi}{2}\sqrt{\frac{\alpha^2}{(1-\alpha^2)(\alpha^2-p^2)}}
    \end{equation}
\end{remark}

\begin{lemma}
    One has $K(p)=O(\log\left[1/\sqrt{1-p^2}\right])$ for $p\nearrow 1$. Particularly,
    \begin{equation}\label{eq:as-Kp}
        \lim_{p\nearrow 1} \sqrt{1-p^2}K(p) = 0.
    \end{equation}
\end{lemma}
\begin{proof}
    Note that, for $p\nearrow 1$,
    \begin{equation}
        K(p) = \int_0^{\frac{\pi}{2}} \frac{1}{\sqrt{1-p^2(\sin(\theta))^2}}\dd\theta = O(1) + \int_0^{\frac{\pi}{4}} \frac{1}{\sqrt{1-p^2\cos^2(\theta)}}\dd\theta.
    \end{equation}
    Using $\sin^2+\cos^2=1$ and substituting $u=\tan(\theta)$, one obtains
    \begin{align}
        \int_0^{\frac{\pi}{4}} &\frac{1}{\sqrt{1-p^2\cos^2(\theta)}}\dd\theta = \int_0^{\frac{\pi}{4}} \frac{\sqrt{1+\tan^2(\theta)}}{\sqrt{\tan^2(\theta)+1-p^2}}\dd\theta = \int_0^{1} \frac{1}{\sqrt{u^2+1-p^2}\sqrt{1+u^2}}\dd u \\
        &\leq  \int_0^{1} \frac{1}{\sqrt{u^2+1-p^2}}\dd u = \frac{1}{2}\log\big[ \frac{1+\sqrt{2-p^2}}{\sqrt{2-p^2}-1} \big] \leq \log\big[ \frac{C}{\sqrt{1-p^2}} \big].\qedhere
    \end{align}
\end{proof}

\begin{definition}[Jacobi elliptic functions]\label{def:jef}
    Let $0<p<1$. We define $\am(\cdot,p)\colon\R\to\R$ to be the inverse function of the strictly increasing mapping $\R\to\R$, $\varphi\mapsto K(\varphi,p)$. Moreover, one defines $ \cn(\cdot,p)\defeq \cos(\am(\cdot,p))$,
       $ \sn(\cdot,p)\defeq \sin(\am(\cdot,p))$ and
         $ \dn(\cdot,p)\defeq \sqrt{1-p^2\sn(\cdot,p)}$.
\end{definition}

\begin{remark}\label{rem:der-jef}
    We collect some well-known identities. Firstly,
    \begin{align}
        &\partial_x \cn(x,p)=-\sn(x,p)\dn(x,p),\qquad &&\partial_x \sn(x,p)=\cn(x,p)\dn(x,p)\\
        &\partial_x\dn(x,p)=-p^2\cn(x,p)\sn(x,p),&&\partial_x\am(x,p)=\dn(x,p).
    \end{align}
    All periods of the elliptic functions are given as follows where $l\in\Z$ and $x\in\R$.
    \begin{align}
        &\am(lK(p),p)=l\frac{\pi}{2},\qquad &&\cn(x+4lK(p),p)=\cn(x,p),\\
        &\sn(x+4lK(p),p)=\sn(x,p),&&\dn(x+2lK(p),p)= \dn(x,p).
    \end{align} 
\end{remark}


\section{Proof of \Cref{lem:lf8-vh}}\label{app:tot-curv}

\begin{proof}[Proof of \Cref{lem:lf8-vh}.]
    For contradiction, after passing to a subsequence, suppose that there exists $M>0$ with 
    \begin{equation}\label{eq:lf8-vh-1.1}
        \frac{1-p_n^2}{\lambda_n^2} \leq M\quad\text{for all $n\in\N$}.
    \end{equation}
    By \cite[(6.1)]{muellerspener2020}, the parameters $p_n$ and $\lambda_n$ must satisfy
    \begin{equation}\label{eq:lf8-vh-2}
    \begin{aligned}
        0&=\int_0^{\frac{\pi}{2}} \frac{\sin^2(\theta)-\frac{\lambda_n}{\scurv_0^2}}{\left(1-\frac{4\scurv_0^2}{(\scurv_0^2-\lambda_n)^2}\cos^2(\theta)\right)\sqrt{1-p_n^2\cos^2(\theta)}}\dd\theta\quad\text{for all $n\in\N$}
    \end{aligned}
    \end{equation}
    where 
    \begin{equation}\label{eq:lf8-vh-2.1}
        \scurv_0^2 = \frac{(2\lambda_n+4)p_n^2}{2p_n^2-1}.
    \end{equation}
    One can compute that $p_n^2-\frac{4\scurv_0^2}{(\scurv_0^2-\lambda_n)^2}=\frac{p_n^2(4+\lambda_n-4p_n^2)^2}{(\lambda_n+4p_n^2)^2}>0$. Especially, $\frac{4\scurv_0^2}{(\scurv_0^2-\lambda_n)^2}<1$ and \eqref{eq:lf8-vh-2} is well-defined. Firstly, $\int_{\frac{\pi}{4}}^{\frac{\pi}{2}} \frac{\sin^2(\theta)-{\lambda_n}/{\scurv_0^2}}{\left(1-{4\scurv_0^2}/{(\scurv_0^2-\lambda_n)^2}\cos^2(\theta)\right)\sqrt{1-p_n^2\cos^2(\theta)}}\dd\theta$ is non-singular for $n\to\infty$. So we only consider the integral on $[0,\pi/4]$. Next, 
    \begin{equation}
        \begin{aligned}
            \int_{0}^{\frac{\pi}{4}}& \frac{\sin^2(\theta)}{\big(1-\frac{4\scurv_0^2}{(\scurv_0^2-\lambda_n)^2}\cos^2(\theta)\big)\sqrt{1-p_n^2\cos^2(\theta)}}\dd\theta \\
            &= \int_{0}^{\frac{\pi}{4}} \frac{\tan^2(\theta)\sqrt{1+\tan^2(\theta)}}{\big(\tan^2(\theta)+1-\frac{4\scurv_0^2}{(\scurv_0^2-\lambda_n)^2}\big)\sqrt{\tan^2(\theta)+1-p_n^2}}\dd\theta\\
            &= \int_{0}^{1} \frac{\vartheta^2}{\big(\vartheta^2+1-\frac{4\scurv_0^2}{(\scurv_0^2-\lambda_n)^2}\big)\sqrt{\vartheta^2+1-p_n^2}\sqrt{1+\vartheta^2}}\dd \vartheta.
        \end{aligned}
    \end{equation}
    Let $\varepsilon\in (0,1)$ and consider again the integral 
    \begin{equation}\label{eq:lf8-vh-3}
        \begin{aligned}
            \int_{0}^{1} &\frac{\vartheta^2}{\big(\vartheta^2+\underbrace{1-\frac{4\scurv_0^2}{(\scurv_0^2-\lambda_n)^2}}_{=\vcentcolon A_n}\big)\sqrt{\vartheta^2+\smash[b]{\underbrace{1-p_n^2}_{=\vcentcolon B_n}}}\sqrt{1+\vartheta^2}}\dd \vartheta \\
            &\leq \varepsilon \int_0^{\varepsilon} \frac{\theta}{(\theta^2+A_n)\sqrt{\theta^2+B_n}}\dd\theta + \int_{\varepsilon}^1 \frac{\theta}{(\theta^2+A_n)\sqrt{\theta^2+B_n}}\dd\theta\\
            &= \varepsilon \frac{\arctan\big(\sqrt{\frac{B_n+\varepsilon^2}{A_n-B_n}}\big)-\arcsin\big(\sqrt{\frac{B_n}{A_n}}\big)}{\sqrt{A_n-B_n}} + \frac{\arcsin\big(\sqrt{\frac{1+B_n}{1+A_n}}\big)-\arctan\big(\sqrt{\frac{B_n+\varepsilon^2}{A_n-B_n}}\big)}{\sqrt{A_n-B_n}}\\
            &\leq \frac{\varepsilon\pi}{2}\frac{1}{\sqrt{A_n-B_n}}+ \frac{\arcsin\big(\sqrt{\frac{1+B_n}{1+A_n}}\big)-\arctan\big(\sqrt{\frac{B_n+\varepsilon^2}{A_n-B_n}}\big)}{\sqrt{A_n-B_n}}.
        \end{aligned}
    \end{equation}
    Now note that, for $n$ sufficiently large, $p_n\geq \frac{1}{2}$ and $\lambda_n\leq 1$ so that, using \eqref{eq:lf8-vh-1.1},
    \begin{equation}\label{eq:lf8-vh-4}
        \sqrt{A_n-B_n} = \frac{p_n(4(1-p_n^2)+\lambda_n)}{\lambda+4p_n^2} \; \smash{\begin{cases}
            \geq \frac{1}{10} \lambda_n\text{ and}\\
            \leq 4\frac{1-p_n^2}{\lambda_n^2}\lambda_n^2 + \lambda_n \leq (1+4M) \lambda_n. 
        \end{cases}}
    \end{equation}
    Moreover, $\sqrt{\frac{1+B_n}{1+A_n}}\nearrow 1$ and, by \eqref{eq:lf8-vh-4},
    \begin{equation}
        \frac{\sqrt{B_n+\varepsilon^2}}{\sqrt{A_n-B_n}} \geq \frac{\varepsilon}{\sqrt{(1+4M)\lambda_n}} \to \infty.
    \end{equation}
    Thus,
    \begin{equation}
        \arcsin\frac{\sqrt{1+B_n}}{\sqrt{1+A_n}}-\arctan\frac{\sqrt{B_n+\varepsilon^2}}{\sqrt{A_n-B_n}} \to \frac{\pi}{2} - \frac{\pi}{2} = 0.
    \end{equation}
    Particularly, we conclude for \eqref{eq:lf8-vh-3} that
    \begin{equation}
        \begin{aligned}
            \int_{0}^{1} &\frac{\vartheta^2}{\big(\vartheta^2+1-\frac{4\scurv_0^2}{(\scurv_0^2-\lambda_n)^2}\big)\sqrt{\vartheta^2+1-p_n^2}\sqrt{1+\vartheta^2}}\dd \vartheta \leq \frac{5\pi\varepsilon+o(1)}{\lambda_n}.
        \end{aligned}
    \end{equation}
    For the second contribution in \eqref{eq:lf8-vh-2}, using that $A_n\geq B_n$, 
    \begin{equation}
        \begin{aligned}
            \int_0^{\frac{\pi}{4}}& \frac{-\frac{\lambda_n}{\scurv_0^2}}{\big(1-\frac{4\scurv_0^2}{(\scurv_0^2-\lambda_n)^2}\cos^2(\theta)\big)\sqrt{1-p_n^2\cos^2(\theta)}}\dd\theta\\
            &=-\frac{\lambda_n}{\scurv_0^2} \int_0^{\frac{\pi}{4}} \frac{(1+\tan^2(\theta))^{\frac{3}{2}}}{\big(\tan^2(\theta)+1-\frac{4\scurv_0^2}{(\scurv_0^2-\lambda_n)^2}\big)\sqrt{\tan^2(\theta)+1-p_n^2}}\dd\theta\\
            &= -\frac{\lambda_n}{\scurv_0^2} \int_0^{1} \frac{\sqrt{1+\vartheta^2}}{\big(\vartheta^2+1-\frac{4\scurv_0^2}{(\scurv_0^2-\lambda_n)^2}\big)\sqrt{\vartheta^2+1-p_n^2}}\dd \vartheta\\
            &\leq -\frac{\lambda_n}{\scurv_0^2} \int_0^{1} \frac{1}{\big(\vartheta^2+1-\frac{4\scurv_0^2}{(\scurv_0^2-\lambda_n)^2}\big)^{\frac{3}{2}}}\dd \vartheta.
        \end{aligned}
    \end{equation}
    This integral can be explicitly computed (using $\int_0^1 1/(\vartheta^2+\Lambda)^{3/2}\dd \vartheta = \frac{1}{\Lambda\sqrt{1+\Lambda}}$ for $\Lambda>0$). Also using \eqref{eq:lf8-vh-2.1} to evaluate $\scurv_0^2$, one arrives at
    \begin{equation}
        \begin{aligned}
            \int_0^{\frac{\pi}{4}}& \frac{-\frac{\lambda_n}{\scurv_0^2}}{\big(1-\frac{4\scurv_0^2}{(\scurv_0^2-\lambda_n)^2}\cos^2(\theta)\big)\sqrt{1-p_n^2\cos^2(\theta)}}\dd\theta\\
            &\leq -\frac{\lambda_n}{\scurv_0^2}\frac{(\lambda_n+4p_n^2)^3}{\sqrt{2}\sqrt{\lambda_n^2+8p_n^2+12\lambda_np_n^2-8\lambda_np_n^4}(\lambda_n^2+16p_n^2(1-p_n^2)+16\lambda_np_n^2(1-p_n^2))}\\
            &\leq -\frac{1}{\sqrt{42}\scurv_0^2} \frac{(\lambda_n+4p_n^2)^3}{\lambda_n + 16\frac{1-p_n^2}{\lambda_n}+16(1-p_n^2)} = -\frac{1}{\lambda_n} \cdot \frac{1}{\sqrt{42}\scurv_0^2} \frac{(\lambda_n+4p_n^2)^3}{1 + 16\frac{1-p_n^2}{\lambda_n^2}+16\lambda_n\frac{1-p_n^2}{\lambda_n^2}}\\
            &\underset{\eqref{eq:lf8-vh-1.1}}{\leq} -\frac{C(M)}{\lambda_n}.
        \end{aligned}
    \end{equation}
    Altogether, we have that
    \begin{equation}
    \begin{aligned}
        \int_0^{\frac{\pi}{4}}& \frac{\cos^2(\theta)-\frac{\lambda_n}{\scurv_0^2}}{\big(1-\frac{4\scurv_0^2}{(\scurv_0^2-\lambda_n)^2}\sin^2(\theta)\big)\sqrt{1-p_n^2\sin^2(\theta)}}\dd\theta \\
        &\leq \frac{5\pi\varepsilon+o(1) - C(M)}{\lambda_n} \to -\infty \quad\text{as $n\to\infty$,}
    \end{aligned}
    \end{equation}
    if we choose $\varepsilon$ sufficiently small. Thus, we obtain a contradiction to \eqref{eq:lf8-vh-2}!
\end{proof}

\section{On the well-posedness of the clamped Willmore flow}\label{app:wp}

We briefly comment on the well-posedness of the clamped Willmore flow equation
\begin{equation}\label{eq:cwf}
    \begin{cases}
        \partial_t u = -\frac{1}{2(u^{(2)})^4} ((\nabla_s^{\bot})^2\curv + \frac{1}{2}|\curv|_g^2\curv-\curv )&\text{on } [0,T)\times I\\
        u(0,\cdot)=u_0&\text{on }I\\
        u(t,y)=u_0(y)&\text{for }(t,y)\in [0,T)\times\partial I\\
        \partial_su(t,y)=\partial_su_0(y)&\text{for }(t,y)\in [0,T)\times\partial I
    \end{cases}
\end{equation}
where $u_0\colon I\to\H^2$ is a suitably regular immersion and $I\subseteq\R$ a non-degenerate compact interval. Well-posedness is not immediate for such geometric evolution equations since they typically have inherent invariances. In the case of \eqref{eq:cwf}, one has an invariance with respect to reparametrizations. As a result, the equation is only degenerate parabolic and classical theory cannot be immediately applied. Still, following the general strategy outlined in \cite[Section 3]{dziukkuwertschaetzle2002}, for a smooth immersion $u_0$ satisfying suitable compatibility conditions, one obtains existence of a smooth maximal solution $u\colon[0,T)\times I\to\H^2$ of \eqref{eq:cwf}, see also \cite[Theorem 1.3]{dallacqualinpozzi2021}. A detailed discussion however lies outside the scope of this article. We now focus on proving uniqueness of solutions of \eqref{eq:cwf}.

By prescribing a suitable tangential velocity, we can transform \eqref{eq:cwf} into a quasi-linear parabolic equation. To this end, one computes that
\begin{equation}\label{eq:dec-grad-e}
    \frac{1}{2(  u ^{(2)})^4}\nabla\E( u ) = \Big[ \frac{\partial_x^4 u }{|\partial_x u |^4} \Big]^{\perp,e} + P( \frac{1}{u^{(2)}},\frac{1}{|\partial_x u |},\partial_x u ,\dots,\partial_x^3 u )
\end{equation}
where $P$ is a polynomial in its arguments with $P( \frac{1}{u^{(2)}},\frac{1}{|\partial_x u |},\partial_x u ,\dots,\partial_x^3 u )\perp\partial_x u $ and $v^{\perp,e}=v-\langle v,\partial_x  u /|\partial_x u |\rangle\partial_x u /|\partial_x u |$ denotes the projection onto the orthogonal complement of $\mathrm{span}\{\partial_xu\}$ in $\R^2$.

Consider the following \emph{analytic equation} associated to \eqref{eq:cwf} with linearized boundary conditions
\begin{equation}\label{eq:ae-cwf}
    \begin{cases}
        \partial_t \widetilde{u} = -\frac12\Big( \frac{\partial_x^4 \widetilde{u} }{|\partial_x \widetilde{u}|^4}  + P( \frac{1}{\widetilde{u}^{(2)}},\frac{1}{|\partial_x \widetilde{u}|},\partial_x \widetilde{u},\dots,\partial_x^3 \widetilde{u})\Big) &\text{on } [0,T)\times I\\
        \widetilde{u}(0,\cdot)=u_0&\text{on }I\\
        \widetilde{u}(t,y)=u_0(y)&\text{for }(t,y)\in [0,T)\times\partial I\\
        \partial_x\widetilde{u}(t,y)=\partial_xu_0(y)&\text{for }(t,y)\in [0,T)\times\partial I.
    \end{cases}
\end{equation}
If one further analyzes the algebra of the terms in the polynomial $P$ in \eqref{eq:dec-grad-e}, with similar computations as in \cite[Appendix A]{ruppspener2020}, one can show that, for any immersion $u_0\in W^{2,2}(I,\H^2)$ in the energy space of the functional $\E$, without supposing further compatibility conditions, there exists a unique maximal solution $\widetilde{u}\colon [0,T)\times I\to \H^2$ to \eqref{eq:ae-cwf} with 
\begin{equation}
    \widetilde{u}\in C^{\infty}((0,T),\H^2)\cap  W^{1,2}([0,T'],L^2(I,\H^2))\cap L^2([0,T'],W^{4,2}(I,\H^2))
\end{equation}
for any $0<T'<T$. Especially, $\widetilde{u}(t,\cdot)\to u_0$ in $W^{2,2}(I,\H^2)$ for $t\searrow 0$. 

The subsequent proof of uniqueness of \eqref{eq:cwf} follows the arguments of \cite[Lemma 3.17]{goessweinmenzelpluda2020} and \cite[Appendix C]{dallacqualinpozzi2021}. Even though we only prove uniqueness for smooth solutions of \eqref{eq:cwf}, we employ techniques that generally allow for weaker notions of solution. This is mainly to avoid technical discussions of compatibility conditions of the initial datum $u_0$ in \eqref{eq:ae-cwf}. 

\begin{proposition}\label{prop:geo-un}
    Consider a smooth immersion $u_0\colon I\to\H^2$ and two smooth solutions $u_1,u_2\colon [0,T)\times I\to\H^2$ to \eqref{eq:cwf}. Then $u_1=u_2$.
\end{proposition}
This proposition is immediate once we have proved the following result. 
\begin{lemma}\label{lem:geo-un}
    In the setting of \Cref{prop:geo-un}, there exist $0<T'\leq T$ and $\varphi\colon[0,T')\times I\to I$ continuous and smooth in $(0,T')\times I$ with $\varphi(0,x)=x$ such that $u_1(t,x)=u_2(t,\varphi(t,x))$ and, for every $t\in[0,T')$, $\varphi(t,\cdot)$ is a diffeomorphism. 
\end{lemma}
\begin{proof}
    W.l.o.g. $I=[0,1]$. In the following, denote by $u$ either $u_1$ or $u_2$. We show that we can reparametrize $u$ to a solution $\widetilde{u}$ of \eqref{eq:ae-cwf} and then use well-posedness of the analytic equation. To this end, consider the following semi-linear partial differential equation with smooth coefficients: 
    \begin{align}
        \partial_t\xi &= -\frac12\partial_y^4\xi \frac{1}{|\partial_y u |^4} + 5 \frac{\partial_y^2\xi\partial_y^3\xi}{\partial_y\xi} \frac{1}{|\partial_y u |^4} - 2 \partial_y^3\xi \frac{\langle\partial_y^2 u ,\partial_y u \rangle}{|\partial_y u |^6} + \frac{15}{2} \frac{(\partial_y^2\xi)^2}{\partial_y\xi} \frac{\langle\partial_y^2 u ,\partial_y u \rangle}{|\partial_yu|^6} \\
        &\quad -\frac{15}{2} \frac{(\partial_y^2\xi)^3}{(\partial_y\xi)^2}\frac{1}{|\partial_y u |^4} - 3\partial_y^2\xi\frac{\langle\partial_y^3 u ,\partial_y u \rangle}{|\partial_yu|^6}+\frac12 \partial_y\xi \frac{\langle\partial_y^4u,\partial_yu\rangle}{|\partial_yu|^6}\label{eq:pde-xi}
    \end{align}
    with $\xi(0,y)=y$ for all $y\in I$ and the boundary conditions
    \begin{equation}\label{eq:bc-xi}
        \xi(t,y)=y\text{ and }\partial_y\xi(t,y)=\frac{|\partial_yu(t,y)|_g}{|\partial_yu_0(y)|_g}\in C^{\infty}([0,T))\text{ for }y=0,1.
    \end{equation}
    Again, by linearizing the equation and applying fixed point methods, following precisely the strategy in \cite{ruppspener2020}, one shows that there exist $0<T'<T$ and $\xi\colon[0,T']\times I\to\R$ with 
    \begin{equation}
        \xi\in C^{\infty}((0,T']\times I,\R) \cap W^{1,2}(0,T',L^2(I,\R)) \cap L^2(0,T',W^{4,2}(I,\R))
    \end{equation}
    and $\partial_x\xi>0$. Particularly, using \eqref{eq:bc-xi}, we have that $\xi(t)\colon I\to I$ is a diffeomorphism for all $0\leq t\leq T'$. Consider now $\Phi\colon[0,T')\times I\to I$ where $\Phi(t)\colon I\to I$ is the inverse map of $\xi(t)\colon I\to I$. Clearly, $\Phi\in C^{\infty}((0,T')\times I)$. Since derivatives of $\Phi$ can be computed classically in $(0,T')\times I$, one can simply use the embedding properties of the space $W^{1,2}(0,T',L^2(I,\R)) \cap L^2(0,T',W^{4,2}(I,\R))$ and some standard Hölder estimates to also obtain $\Phi\in W^{1,2}(0,T',L^2(I,\R)) \cap L^2(0,T',W^{4,2}(I,\R))$, cf. \cite[Proposition B.3 and proof of Proposition 2.7]{ruppspener2020}. For $\widetilde{u}(t,x)=u(t,\Phi(t,x))$, following the arguments of \cite[Lemma 5.3]{garcke2020}, one obtains $\widetilde{u}\in C^{\infty}((0,T')\times I)\cap W^{1,2}(0,T',L^2(I)) \cap L^2(0,T',W^{4,2}(I))$. Furthermore,
    \begin{equation}\label{eq:ev-eq-for-u-tilde}
        \partial_t\widetilde{u}(t,x) = \partial_t u(t,\Phi(t,x)) + \partial_y u(t,\Phi(t,x)) \partial_t\Phi(t,x).
    \end{equation}
    As $\nabla\E$ is invariant with respect to reparametrizations, \eqref{eq:dec-grad-e} and \eqref{eq:cwf} yield
    \begin{align}
        \partial_tu(t&,\Phi(t,x)) = -\frac{1}{4(u^{(2)}(t,\Phi(t,x)))^4}\nabla\E(u(t))(\Phi(t,x)) \\
        &= -\frac{1}{4(\widetilde{u}^{(2)}(t,x))^4}\nabla\E(\widetilde{u}(t))(x)\\
        &= -\frac12\left[\Big[ \frac{\partial_x^4 \widetilde{u}(t,x)}{|\partial_x \widetilde{u}(t,x) |^4} \Big]^{\perp,e} + P( \frac{1}{\widetilde{u}(t,x)^{(2)}},\frac{1}{|\partial_x \widetilde{u}(t,x) |},\partial_x \widetilde{u}(t,x) ,\dots,\partial_x^3 \widetilde{u}(t,x))\right].\label{eq:ev-eq-for-u-tilde-1}
    \end{align}
    Further, since $\Phi(t,\xi(t,y)) = y$, one has $\partial_t \Phi(t,\xi(t,y)) = -\partial_x\Phi(t,\xi(t,y))\partial_t\xi(t,y)$. Writing $x=\xi(t,y)$ and using 
    \begin{align}
        \partial_y\xi&=\frac{1}{\partial_x\Phi},\quad \partial_y^2\xi = -\frac{\partial_x^2\Phi}{(\partial_x\Phi)^3},\quad \partial_y^3\xi=-\frac{\partial_x^3\Phi}{(\partial_x\Phi)^4}+3\frac{(\partial_x^2\Phi)^2}{(\partial_x\Phi)^5},\\
        \partial_y^4\xi&=-\frac{\partial_x^4\Phi}{(\partial_x\Phi)^5}+10\frac{\partial_x^2\Phi}{(\partial_x\Phi)^6}\partial_x^3\Phi-15\frac{(\partial_x^2\Phi)^3}{(\partial_x\Phi)^7}
    \end{align}
    where we write $\xi=\xi(t,y)$ and $\Phi=\Phi(t,x)$, cf. \cite[p.1341]{dallacqualinpozzi2021}, as well as 
    \begin{equation}
        \partial_x^4\widetilde{u}=\partial_y^4u(\partial_x\Phi)^4+6\partial_y^3u(\partial_x\Phi)^2\partial_x^2\Phi+4\partial_y^2u\partial_x\Phi\partial_x^3\Phi+3\partial_y^2u(\partial_x^2\Phi)^2+\partial_yu\partial_x^4\Phi,
    \end{equation}
    where $\widetilde{u}=\widetilde{u}(t,x)$ and $u=u(t,y)$, \eqref{eq:pde-xi} yields that
    \begin{equation}
        \partial_t\Phi(t,x) = -\frac12\frac{1}{|\partial_x\widetilde{u}(t,x)|^6}\langle \partial_x^4\widetilde{u}(t,x),\partial_x\widetilde{u}(t,x) \rangle \partial_x\Phi(t,x)
    \end{equation}
    so that, by \eqref{eq:ev-eq-for-u-tilde} and \eqref{eq:ev-eq-for-u-tilde-1}, $\widetilde{u}$ solves \eqref{eq:ae-cwf}. Indeed, for $x\in\{0,1\}$, writing again $x=\xi(t,y)$, one has $\Phi(t,x)=x=\xi(t,y)=y$ and thus $\widetilde{u}(t,x)=u_0(x)$ by \eqref{eq:cwf} as well as
    \begin{equation}
        \partial_x\widetilde{u}(t,x) = \partial_yu(t,y)\partial_x\Phi(t,x) = \partial_yu(t,y) \frac{1}{\partial_y\xi(t,y)} = \frac{\partial_yu_0(y)\frac{|\partial_yu(t,y)|_g}{|\partial_yu_0(y)|_g}}{\frac{|\partial_yu(t,y)|_g}{|\partial_yu_0(y)|_g}}=\partial_xu_0(x),
    \end{equation}
    again using \eqref{eq:cwf}. Since \eqref{eq:ae-cwf} is well-posed in $W^{1,2}(0,T',L^2(I)) \cap L^2(0,T',W^{4,2}(I))$ by the previous discussion, the statement is proved with the above argument applied with $u=u_1$, $u=u_2$. 
\end{proof}
\begin{proof}[Proof of \Cref{prop:geo-un}.]
    Let 
    \begin{equation}
        t_0=\sup\{\tau\in[0,T):u_1(t,x)=u_2(t,x)\text{ for all $x\in I$ and $0\leq t\leq \tau$}\}.
    \end{equation} 
    If $t_0=T$, we are done. Thus, suppose that $t_0<T$. By continuity, we have $u_1(t_0,x)=u_2(t_0,x)$ for all $x\in I$. Using \Cref{lem:geo-un}, there exist $0<T'\leq T-t_0$ and, writing $t_1=t_0+T'$, $\varphi\colon [t_0,t_1)\times I\to I$ continuous and smooth on $(t_0,t_1)\times I$ with $u_1(t,x)=u_2(t,\varphi(t,x))$ for $t\in[t_0,t_1)$ and $x\in I$. Further, $\varphi(t_0,x)=x$ for all $x\in I$. Using the invariance of $\nabla\E$ with respect to reparametrizations, one obtains for $t\in (t_0,t_1)$, using \eqref{eq:cwf},
    \begin{align}
        -&\frac{1}{4(u_1(t,x)^{(2)})^4}\nabla\E(u_1(t))(x) = \partial_t[ u_2(t,\varphi(t,x))]\\
        &= -\frac{1}{4(u_2(t,\varphi(t,x))^{(2)})^4}\nabla\E(u_2(t))(\varphi(t,x)) + \partial_xu_2(t,\varphi(t,x))\partial_t\varphi(t,x)\\
        &= -\frac{1}{4(u_1^{(2)}(t,x))^4}\nabla\E(u_1(t))(x) + \partial_xu_2(t,\varphi(t,x))\partial_t\varphi(t,x),
    \end{align}
    i.e. $\partial_t\varphi(t,x)=0$ for all $t_0<t<t_1$ and $x\in I$. Since $\varphi(t_0,x)=x$, by continuity, we have $u_1(t,x)=u_2(t,x)$ for all $t\in [t_0,t_1)$ and $x\in I$, a contradiction!
\end{proof}

\section{Geometric formulae for cylindrical surfaces of revolution}\label{app:geom-form-surf-rev}

Let $I=[a,b]\subseteq\R$ be a compact interval, $u\colon I\to\H^2$ an immersion with $|\partial_xu|\equiv 1$, $\Sigma=I\times\S^1$ and $f_u\colon \Sigma\to\R^3$ the associated surface of revolution as in \eqref{eq:surf-rev}. Then $(\partial_x,\partial_{\theta})$ is a global orthogonal frame on $\Sigma$ in the orientation of $\Sigma$. Indeed, one computes in the associated local coordinates that
\begin{equation}
    (g_{ij})_{ij=1,2} = \begin{pmatrix}
        1&0\\0&(u^{(2)})^2
    \end{pmatrix}\quad\text{and}\quad(g^{ij})_{ij=1,2}=\begin{pmatrix}
        1&0\\0&(u^{(2)})^{-2}
    \end{pmatrix}.
\end{equation}
Moreover, the normal field $\nu$ satisfies
\begin{equation}
    \nu(x,\theta) = \frac{\partial_xf_u\times\partial_{\theta}f_u}{|\partial_xf_u\times\partial_{\theta}f_u|}(x,\theta) = \begin{pmatrix}
        \partial_xu^{(2)}(x)\\-\partial_xu^{(1)}(x)\cos\theta\\-\partial_xu^{(1)}(x)\sin\theta
    \end{pmatrix}.
\end{equation}
One computes for the second fundamental form $A_{ij}=\langle \partial^2_{ij}f_u,\nu\rangle$, the mean curvature $H=\frac12g^{ij}A_{ij}$ and $A^0_{ij}=A-Hg_{ij}$
\begin{align}
    (A_{ij})_{ij=1,2} &= \begin{pmatrix}
        \partial_x^2u^{(1)}\partial_xu^{(2)}-\partial_x^2u^{(2)}\partial_xu^{(1)}&0\\0&\partial_xu^{(1)}u^{(2)}
    \end{pmatrix}, \\
    H &= \frac12(\partial_x^2u^{(1)}\partial_xu^{(2)}-\partial_x^2u^{(2)}\partial_xu^{(1)}+\frac{\partial_xu^{(1)}}{u^{(2)}}),\\
    A_{11}^0 = -\frac{1}{(u^{(2)})^2}A_{22}^0 &= \frac12(\partial_x^2u^{(1)}\partial_xu^{(2)}-\partial_x^2u^{(2)}\partial_xu^{(1)}-\frac{\partial_xu^{(1)}}{u^{(2)}}) \quad\text{and}\quad A_{12}^0=A_{21}^0=0.
\end{align}
Thus $|A^0|^2 = g^{ij}g^{kl} A^0_{ik}A^0_{jl}=\frac12(\partial_x^2u^{(1)}\partial_xu^{(2)}-\partial_x^2u^{(2)}\partial_xu^{(1)}-\frac{\partial_xu^{(1)}}{u^{(2)}})^2$. If $\curv=\nabla_s\partial_su$ is the curvature of $u$ in $\H^2$, using \cite[Equation~(2.9)]{eichmanngrunau2019},
\begin{equation}
    |\curv(x)|_g^2 = (\partial_x^2u^{(1)}\partial_xu^{(2)}u^{(2)}-\partial_x^2u^{(2)}\partial_xu^{(1)}u^{(2)}-\partial_xu^{(1)})^2|_x = 2(u^{(2)}(x))^2 |A^0|^2(x,\theta) 
\end{equation}
for all $x\in I$ and $\theta\in\S^1$. In particular, 
\begin{align}
    \pi\E(u) &= \pi\int_{I}|\curv(x)|_g^2\dd s= \int_{\S^1}\int_I \frac12 |\curv(x)|^2 \frac{1}{u^{(2)}(x)}\dd x\dd\theta = \int_{\Sigma} |A^0|^2(x,\theta) \cdot u^{(2)}(x)\dd x\otimes\dd\theta \\
    &= \int_{\Sigma} |A^0|^2\dd\mu = \W_0(f_u).
\end{align}
By direct computation, $H^2-\frac{1}{2}|A^0|^2=K=\det A$, the Gauss curvature. Using the Gauss-Bonnet Theorem and $\chi(\Sigma)=0$, one thus finds
\begin{equation}
    \W(f_u)=\int_{\Sigma} H^2\dd\mu = \frac{1}{2}\W_0(f_u) + \int_{\Sigma}K\dd\mu = \frac12\W_0(f_u) - \int_{\partial\Sigma} \scurv_g
\end{equation}
where $\int_{\partial\Sigma} \scurv_g=2\pi \partial_xu^{(2)}\Big|_{\partial I}$ is the total geodesic curvature of $\partial\Sigma$.

\section*{Acknowledgments}

This project has been supported by the Deutsche Forschungsgemeinschaft (DFG,
German Research Foundation), project no. 404870139. The author would like to thank Anna Dall’Acqua for helpful discussions and comments. Moreover, the author is grateful to the referee for their valuable comments on the original manuscript.


\end{document}